\RequirePackage[british]{babel}
\documentclass[a4paper,11pt]{amsart}
\usepackage{amssymb,amscd,amsfonts,amsbsy,nicefrac}
\usepackage{amsmath,amsthm,amssymb,amsfonts,enumitem}
\usepackage{mathrsfs}
\usepackage{amsrefs}
\usepackage{hyperref}
\usepackage[pdftex]{graphicx,color}
\usepackage{mathrsfs}
\usepackage{graphicx}
\usepackage{epstopdf}
\usepackage[a4paper,tmargin=3cm,bmargin=3cm,lmargin=3cm,rmargin=3cm]{geometry}
\usepackage{enumitem}

\newtheorem{thm}{Theorem}[section]
\newtheorem{lem}[thm]{Lemma}

\newtheorem{cor}[thm]{Corollary}
\newtheorem{prop}[thm]{Proposition}
\newtheorem{defn}[thm]{Definition}
\newtheorem{rem}[thm]{Remark}

\def\Re{{\rm Re}\,}
\def\Im{{\rm Im}\,}
\def\R{\mathbb{R}}

\def\les{\lesssim}
\def\C{{\mathbb C}}

\def\2q{{\frac{2}{|B|}}}

\newcommand{\N}{\mathbb{N}}

\newcommand{\bmo}{\mathrm{bmo}}

\newcommand{\Z}{\mathbb{Z}}

\newcommand{\newsigma}{\widetilde{\sigma}}

\newcommand{\supp}{\mbox{supp}\,}

\renewcommand{\S}{{\mathcal{S}}}

\renewcommand{\leq}{\leqslant}
\renewcommand{\geq}{\geqslant}

\def\bra#1{{\langle{#1}\rangle}}

\newcommand{\phase}{\varphi}

\definecolor{red}{RGB}{255,0,0}
\definecolor{RED}{RGB}{255,0,0}
\definecolor{blue}{RGB}{0,0,255}
\definecolor{BLUE}{RGB}{0,0,255}
\newcommand{\abs}[1]{\left|#1\right|}
\newcommand{\set}[1]{\left\{#1\right\}}
\newcommand{\brkt}[1]{\left(#1\right)}
\newcommand{\dd}{\,\mathrm{d}}
\newcommand{\ddd}{\,\text{\rm{\mbox{\dj}}}}
\renewcommand{\d}{\partial}
\newcommand{\norm}[1]{\left\Vert#1\right\Vert}
\newcommand{\p}[1]{\langle{#1}\rangle}

\begin{document}

\title[A Seeger-Sogge-Stein theorem for bilinear Fourier integral operators] {A Seeger-Sogge-Stein theorem for bilinear Fourier integral operators}
\author[S.~Rodr\'iguez-L\'opez]{Salvador Rodr\'iguez-L\'opez}
\address{Department of Mathematics, Uppsala University, Uppsala, SE 75106,  Sweden}
\email{salvador@math.uu.se}
\author[D.~Rule]{David Rule}
\address{Mathematics Institute, Link\"oping University, 581 83 Link\"oping, Sweden}
\email{david.rule@liu.se}
\author[W.~Staubach]{Wolfgang Staubach}
\address{Department of Mathematics, Uppsala University, Uppsala, SE 75106,  Sweden}
\email{wulf@math.uu.se}
\thanks{The first author has been partially supported by the Grant MTM2010-14946. The third author is partially supported by a grant from the Crawfoord foundation}


\begin{abstract}
We establish the regularity of bilinear Fourier integral operators with bilinear amplitudes in $S^m_{1,0} (n,2)$ and non-degenerate phase functions, from $L^p \times L^q \to L^r$ under the assumptions that $m\leq  -(n-1)(|\frac{1}{p}-\frac{1}{2}|+|\frac{1}{q}-\frac{1}{2}|)$ and $\frac{1}{p}+\frac{1}{q}=\frac{1}{r}$. This is a bilinear version of the classical theorem of Seeger-Sogge-Stein concerning the $L^p$ boundedness of linear Fourier integral operators. Moreover, our result goes beyond the aforementioned theorem in that it also includes the case of non-Banach target spaces.
\end{abstract}

\maketitle

\section{Introduction}
Given a function $a(x,\xi)\in \mathcal{C}^{\infty}(\R^n\times\R^n)$ with compact support in the $x$-variable, satisfying the estimate
\begin{equation}\label{standard linear symbols}
\begin{aligned}
& |\partial_{\xi}^{\alpha}\partial_{x}^{\beta}a(x,\xi)|\leq C_{\alpha,\beta} (1+|\xi|)^{m-|\alpha|} \\
\end{aligned}
 \end{equation}
for $m \in \R$, and each multi-index $\alpha,\beta$, and given $\phase(x,\xi)\in\mathcal{C}^{\infty}(\R^n\times(\R^n \setminus\{0\}))$, positively homogeneous of degree one in the $\xi$ variable with
$\det\,\d^{2}_{x, \xi} \phase(x,\xi)\neq 0$ on the support of $a(x,\xi)$, L. H\"ormander \cite{H2} defined a {\it{Fourier integral operator}} (FIO for short) $T^{\phase}_{a}$ acting a-priori on Schwartz class functions $f$, by setting

\begin{equation}\label{defn: lin FIO}
T^{\phase}_{a}f(x) = \int a(x,\xi)\widehat{f}(\xi)e^{i\phase(x,\xi)} \ddd\xi.
\end{equation}
Here
$\ddd$ is Lebesgue measure normalised by the factor $(2\pi)^{-n}$ and
\[
        \widehat{f}(\xi)=\int f(x)e^{-i x.\xi}\, \dd x
\]
is the Fourier transform of $f$. The function $a \colon \R^n\times\R^n \to \C$ is called the {\it{amplitude}} of $T^{\phase}_{a}$ and $\phase\colon \R^n\times (\R^n\setminus\{0\}) \to \R$ is called the {\it{phase function}}.

The study of the regularity (i.e. boundedness) of local FIOs in $L^p$ spaces for $1<p<\infty$ goes back, in the case of $p=2$, to the pioneering work of H\"ormander where the $L^2$ boundedness was proven when the amplitude verifies \eqref{standard linear symbols} with $m=0$. The optimal $L^p$ boundedness of FIOs is due to A. Seeger, C. Sogge and E. Stein \cite{SSS} who showed that the operator is $L^p$-bounded for $1<p<\infty$ provided that
$m\leq (n-1)|\frac{1}{p}-\frac{1}{2}|$ in \eqref{standard linear symbols}.\\

The main goal of this paper is to prove the bilinear analogue of the aforementioned result of Seeger-Sogge-Stein. To this end, let $a(x,\xi, \eta)$ be a smooth function with compact support in the $x$ variable satisfying the estimate
\begin{equation}\label{standard symbols}
\begin{aligned}
& |\partial_{\xi}^{\alpha}\partial_{\eta}^{\beta}\partial_{x}^{\gamma} \sigma(x,\xi,\eta)|\leq C_{\alpha,\beta, \gamma} (1+|\xi|+|\eta|)^{m-|\alpha|-|\beta|} \\
\end{aligned}
\end{equation}
for $m \in \R$ and each triple of multi-indices $\alpha, \beta$ and $\gamma$. Let the phase functions $\phase_j (x, \xi) \in \mathcal{C}^\infty(\R^n \times (\R^n\setminus\{0\}))$ be homogeneous of degree one in the frequency variable $\xi$, and verify the {\it{non-degeneracy condition}}
\begin{equation} \label{intro nondeg}
\det\,\d^{2}_{x, \xi} \phase_{j}(x,\xi)\neq 0
\end{equation}
for $j=1,2$. Our goal is to prove that
\begin{equation}\label{intro defn RS FIO}
T^{\phase_1, \phase_2}_{a}(f,g)(x) = \iint a(x,\xi,\eta)\widehat{f}(\xi)\widehat{g}(\eta)e^{i\phase_1(x,\xi)+i\phase_2(x,\eta)} \ddd\xi \ddd\eta,
\end{equation}
defines a bounded bilinear operator on a product of $L^p$ or $H^p$ spaces were the target space could either be a Banach or a non-Banach $L^p$ space, whenever the phases verify the same conditions as in the Seeger-Sogge-Stein theorem. Therefore, our theorem is a bilinear version of this result. This can also be seen as an extension of the results of R. Coifman and Y. Meyer \cite{CM4}, and L. Grafakos and R.Torres for bilinear pseudodifferential operators, which can be regarded as particular examples of operators of the type  \eqref{intro defn RS FIO} with linear phases $\phi_j(x,\xi) = x\cdot\xi$.
\\

We shall briefly review what is known about the boundedness of bilinear FIOs of the type \eqref{intro defn RS FIO}.  The first result in this direction was proven by L. Grafakos and M. Peloso in \cite{GP} where the authors showed the boundedness of $T^{\phase_1, \phase_2}_{a}$ from $L^p \times L^q \to L^r$ under the assumptions that $m< -(n-1)(|\frac{1}{p}-\frac{1}{2}|+|\frac{1}{q}-\frac{1}{2}|)$, $\frac{1}{p}+\frac{1}{q}=\frac{1}{r}$ and $1\leq p,q\leq 2$. In \cite{RS}, S. Rodr\'iguez-L\'opez and W. Staubach extended this theorem to the full range of exponents $1\leq p,q\leq \infty$. The main result of this paper, namely Theorem \ref{main}, extends this boundedness to the endpoint
$m= -(n-1)(|\frac{1}{p}-\frac{1}{2}|+|\frac{1}{q}-\frac{1}{2}|)$.\\

It turns out that proving our main theorem is a technical task and requires some heavy machinery from various parts of harmonic analysis. To clarify why the situation becomes complicated, one could consider operators with phases $\phase_j(x,\xi) = x\cdot\xi + \lambda_j(x,\xi)$ with $\lambda_j(x,\xi)\in\mathcal{C}^{\infty}(\R^n \times\R^n\setminus \{0\} )$, positively homogeneous of degree 1 in $\xi$. We will illustrate here that operators with such phases and amplitudes satisfying \eqref{standard symbols} cannot easily be understood with previously available techniques without slightly artificial additional assumptions which will require more decay from the symbol and rather un-natural restrictions on the phase functions. Following the freezing of the variables argument of Coifman-Meyer (see \cite{CM4} page 157), one introduces the function
\[
F(x,y)= \chi(x) \chi(y)\iint e^{i\lambda_1(y,\xi)+i\lambda_2(y,\eta)}\,\sigma(x,\xi,\eta)\widehat{f}(\xi)\widehat{g}(\eta)e^{ix\cdot(\xi +\eta)}\, \ddd\xi \ddd\eta \]
where $\chi(x)$ is a smooth cut-off function. Using Sobolev's embedding theorem we have that $|F(x,x)| \lesssim\sum_{|\alpha|\leq n+1} \int |\partial^{\alpha}_y F(x,y)|\, dy$ and the desired $L^p \times L^q \to L^r$ (say for $r\geq 1$) boundedness of the Fourier integral operator with phases $x\cdot\xi + \lambda_j(x,\xi)$ would follow by estimating the $L^r$ norm of $F(x,x)$ with an upper bound in terms of the $L^p$ and $L^q$ norms of $f$ and $g$ repectively. Given the smoothness and the compact support in $y$, this in turn amounts to obtaining the aforementioned upper bound for $\Vert \partial^{\alpha}_y F(x,y) \Vert _{L^r_{x}}$.
Now using Leibniz's rule and the chain rule, matters reduce to the study of $L^r_{x}$ boundedness of operators of the form
\begin{equation}\label{type 1 operator}
\begin{split}
\chi(x) \partial^{\alpha_1}_{y}\chi(y) \iint [\partial^{\alpha_2}_{y}(\lambda_1(y,\xi)+\lambda_2(y,\eta))]^k\, \sigma(x,\xi,\eta)\, e^{i\lambda_1(y,\xi)+i\lambda_2(y,\eta)}\times \\ \hat{f}(\xi)\, \hat{g}(\eta)e^{ix\cdot(\xi +\eta)}\, \ddd\xi \ddd\eta,
\end{split}
\end{equation}
where $k\in \mathbb{Z}_{+}$ and $|\alpha_2|\geq 1$.
The operator in \eqref{type 1 operator} is a bilinear pseudodifferential operator with symbol
$$a_{y}(x,\xi,\eta):= \chi(x) \partial^{\alpha_1}_{y}\chi(y)[\partial^{\alpha_2}_{y}(\lambda_1(y,\xi)+\lambda_2(y,\eta))]^k\, \sigma(x,\xi,\eta)\, e^{i\lambda_1(y,\xi)+i\lambda_2(y,\eta)},$$
where $y$ is considered as a parameter. If for $m\leq -n\max \big\{\frac{1}{2},\frac{1}{p}, \frac{1}{q}, 1-\frac{1}{r},\frac{1}{r}-\frac{1}{2}\big\}$ one has the estimates
\begin{align*}
\sup_{x,\xi, \eta}|\partial_{x}^{\alpha} \partial^{\beta}_{\xi} \partial^{\gamma}_{\eta}  \sigma (x, \xi, \eta)| &\leq C_{\alpha,\beta, \gamma}(1+|\xi|+|\eta|)^{m}, \, \text{for all}\,\, \alpha\,, \beta \,\text{and}\, \gamma \\
\sup_{x, \xi} |\partial_{x}^{\alpha} \partial^{\beta}_{\xi} \lambda_{j} (x, \xi)|& \leq C_{\alpha,\beta},\, \text{for all}\, |\alpha|\geq 1\,, \beta\,\text{and}\, \gamma,
  \end{align*}
then in light of a recent result of A. Miyachi and N. Tomita \cite{MT}, the $L^p \times L^q \to L^r$ boundedness with $\frac{1}{p}+\frac{1}{q}=\frac{1}{r}$ is indeed valid. But if we for instance and for the sake of comparison consider the case of the $L^2 \times L^2 \to L^1$ boundedness for operators that satisfy \eqref{standard symbols} with $m=0$, then the conditions above on the phase $\lambda_j$ are not in general enough to yield the desired boundedness. However, in this paper we are able to apply our $L^2 \times L^2 \to L^1$ result to operators with amplitudes satisfying \eqref{standard symbols} with $m=0$ and phases that are for example of the form
\[
	\varphi_j(x,\xi)=x\cdot \xi+ (u\cdot \xi) \bra{x}, \qquad \mbox{$u\in \R^n$ and $\abs{u}<1$},
\]
where the notation $\langle \cdot\rangle$ stands for $(1+|\cdot|^2)^{1/2}$. Observe that that if we set $\lambda_j(x,\xi)=(u\cdot \xi) \bra{x}$ then it is clear that estimate $|\nabla_x\lambda_j (x, \xi)|\leq C$ is not valid and so even if the amplitude of the operator did satisfy \eqref{standard symbols} with $m=\frac{-n}{2}$, the discussion above doesn't apply to the corresponding bilinear operator in order to prove the $L^2 \times L^2 \to L^1$ boundedness. An attempt to eliminate the growth in $x$ by replacing $\bra{x}$ in the example above would not always reduce the operator to one which could be treated by the bilinear pseudodifferential theory. For example, consider an operator in dimension one given by
\[
	T_\varepsilon(f,g)(x)=\iint_{\R^2} a(x,\xi,\eta)e^{i\varepsilon \xi\sin x} \widehat{f}(\xi)\widehat{g}(\eta)e^{ix (\xi+\eta)}\ddd \xi\ddd \eta \quad {\text{for $\varepsilon\in (0,1)$.}}
\]
Once again, viewing this operator as a bilinear pseudodifferential operator, we observe that it has the symbol $a(x,\xi,\eta)e^{i\varepsilon \xi\sin x}$ which at best belongs to the class $S^0_{0,1}(1,2)$ (again see Definition \ref{defn of hormander symbols} below). This is not a favourable class of symbols to study, even for linear operators. Moreover setting $\lambda(x, \xi)= \varepsilon \xi\sin x,$ we still will not have the boundedness of $\partial _x \lambda$ and therefore not even the boundedness result sketched above could be applied to this case. But once again, our result applies in this case and yields the $L^2 \times L^2 \to L^1$ regularity.\\

To show the $L^p \times L^q \to L^r$  boundedness of bilinear FIOs we shall prove the boundedness on $L^2 \times L^2 \to L^1$, $L^2 \times L^\infty \to L^2$, $L^\infty \times L^2 \to L^2$, $L^\infty \times L^\infty \to \mathrm{BMO}$, $H^1 \times L^\infty \to L^1$, $L^\infty \times H^1 \to L^1$, $L^2 \times H^1\to L^\frac{2}{3}$, $H^1\times L^2 \to L^\frac{2}{3}$ and finally $H^1 \times H^1 \to L^{\frac{1}{2}}$. In the investigation of each case, there are certain technical difficulties that need to be overcome, and in tackling these we have taken a more abstract attitude to proving propositions and lemmas that could be useful even outside the context of the specific case. Bilinear interpolation yields then the desired result for the full range of exponents $p$ and $q$.\\

In proving these boundedness results, we start by decomposing the operator into low and high frequency parts. In Section \ref{sec:low_frequency}, we analyse the problem when the operators are localized in one of the frequency variables (see Theorem \ref{main_low}).

Then in Section \ref{sect:high}, in order to study the high frequency part of the operator when the target space is Banach,  we will introduce a continuous version of the symbol decomposition for bilinear paraproducts due to A. Calder\'on and Coifman-Meyer, which we have adapted to the study of bilinear FIOs.

Thereafter, using commutators and our parameter dependant composition formula for pseudodifferential and Fourier integral operators (Theorem \ref{left composition with pseudo} below), we reduce the study of bilinear FIOs to the analysis of main terms and error terms.

 The error terms of the commutators can  be mainly handled by applying known results concerning boundedness of linear FIOs. The analyses of the main terms, however, need to be carried out as separate case studies, where in some cases it behoves us to assume that the dimension of the space is at least two.

In the case of $L^2 \times L^2 \to L^1$, the main term can be reduced to a composition of a bilinear paraproduct and two linear FIOs. For the $L^2 \times L^\infty \to L^2$, $L^\infty \times L^2 \to L^2$ cases we use kernel estimates, quadratic estimates and Carleson measure estimates. For $L^\infty \times L^\infty \to \mathrm{BMO}$, we use once again Carleson measure estimates, the Fefferman-Stein theory of real Hardy spaces and the Stein-Weiss theory of harmonic functions on $\R^{n+1}_{+}$. In the cases of $H^1 \times L^\infty \to L^1$ and $L^\infty \times H^1 \to L^1$, we use Goldberg's theory of local Hardy and bmo spaces. These are all carrried out in subsection \ref{sec:Banach_boundedness_result}.\\

It is rather surprising that some of the endpoint results above actually fail in dimension one (in which case the order of the amplitude is $m=0$). Indeed in Section \ref{sect:1d}, we provide counterexamples to the $L^\infty \times L^\infty \to \mathrm{BMO}$, $L^\infty \times H^1 \to L^1$ and $L^\infty \times L^p \to L^p$ ($1<p<\infty$) boundedness of one dimensional bilinear FIOs. This means that,  for the order $m<0$, the one dimensional results in \cite{RS} concerning the boundedness in $L^\infty \times L^\infty \to L^\infty$ and $L^\infty \times L^1 \to L^1$ are sharp.\\

For the cases $L^2 \times H^1\to L^\frac{2}{3}$, $H^1\times L^2 \to L^\frac{2}{3}$ and $H^1 \times H^1 \to L^{\frac{1}{2}}$, we are outside the realm of Banach spaces, and therefore we will need to use a discrete decomposition of the operator.  Thereafter we use Peetre's maximal function, the Fefferman-Stein vector-valued maximal function estimates and the theory of Triebel-Lizorkin spaces. It is worth mentioning that the method we use to prove the first two boundedness results does not rely on any atomic decomposition of $H^1$ which is the standard approach in establishing regularity in Hardy spaces. These non-Banach results are established in a more general setting in Section \ref{sect:non_banach}. \\

Finally, Section \ref{sec:proof_of_asymptotic} is devoted to the proof of Theorem \ref{left composition with pseudo}, concerning composition of parameter dependant pseudodifferential and Fourier integral operators.
\section{Main results}

Here we recall some definitions and tools from the linear and bilinear theory of Fourier integral operators which will be needed in what follows. The basic definition of amplitudes and symbols goes back to H\"ormander \cite{H1}.

    \begin{defn}\label{defn of hormander symbols}
Let $m\in \R$, $0\leq \delta \leq 1$, $0\leq \rho\leq 1$ and $d\in\{1,2\}$. A function $\sigma \in \mathcal{C}^{\infty}(\R^{n} \times\R^{dn})$ belongs to the class $S^{m}_{\rho,\delta }(n,d)$, if for all multi-indices $\alpha$ and $\beta$, there exist constants $C_{\alpha,\beta}$ such that
   \begin{align*}
      |\partial_{\Xi}^{\alpha}\partial_{x}^{\beta}\sigma(x,\Xi)| \leq C_{\alpha,\beta} \langle\Xi\rangle^{m-\rho\vert \alpha\vert + \delta  |\beta|}, \qquad \text{for all $(x,\Xi)\in \R^n\times \R^{dn}$}.
   \end{align*}
\end{defn}
In the sequel, we shall use the notation $S^m_{\rho,\delta }$ instead of  $S^{m}_{\rho,\delta }(n,1)$. These are the most common classes of amplitudes.

\begin{defn} A function $\phase(x,\xi)\in C^\infty\brkt{\R^n\times (\R^n\setminus \{0\})}$ is called a non-degenerate phase function if it is real-valued, positively homogeneous of degree one in $\xi$ and for $\xi\neq 0$ verifies
\[
	\det \brkt{\frac{\d^2 \phase}{\d x \d\xi}}\neq 0,
\]
for $x$ in the spatial support of the amplitude.
\end{defn}

\begin{defn}
A \emph{linear Fourier integral operator} $T^{\phase}_{\sigma}$ is an operator which is defined to act on a Schwartz function $f$ by the formula
      \begin{equation} \label{eq:Linear_FIO}
      	T_\sigma^{\varphi}f(x)= \int_{\R^n} e^{i\phase(x,\xi)} \sigma(x,\xi) \widehat{f}(\xi) \, \ddd\xi,
      \end{equation}
where
$\ddd$ is Lebesgue measure normalised by the factor $(2\pi)^{-n}$ and
\[
	\widehat{f}(\xi)=\int f(x)e^{-i x.\xi}\, \dd x
\]
is the Fourier transform of $f$.
\end{defn}
We shall also deal with the class $L^{\infty}S^m_{\varrho}$ of rough symbols/amplitudes introduced by C. Kenig and W. Staubach in \cite{KS}.
\begin{defn}\label{defn of amplitudes}
   Let $m\in \R$ and $0\leq\varrho \leq 1$. A function $a(x,\xi)$ which is smooth in the frequency variable $\xi$ and bounded
   measurable in the spatial variable $x$, belongs to the symbol class $L^{\infty}S^{m}_{\varrho}$, if for all multi-indices $\alpha$
   it satisfies
   \begin{align*}
      \sup_{\xi \in \R^n} \langle \xi\rangle ^{-m+\varrho\vert \alpha\vert}
      \Vert \partial_{\xi}^{\alpha}a(\cdot\,,\xi)\Vert_{L^{\infty}(\R^{n})}< +\infty.
   \end{align*}
\end{defn}
In the following theorem we collect the boundedness results for linear FIOs which will be used throughout the paper.
In what follows, we shall use the notation
\[
	a(tD)f(x):=\int_{\R^n} a(t\xi) \widehat{f}(\xi)e^{ix\cdot \xi} \, \ddd\xi,
\]
for $t>0$ and $a$ in a suitable symbol class. Whenever $t=1$ we shall simply write $a(D)$.

Given a bump function $\widehat{\Theta}$  supported in a ball near the origin, such that $\widehat\Theta(0)=1$,  the Hardy space $H^p$ is the class of tempered distributions $f$ such that
\begin{equation*}\label{eq:H^p}
	\norm{f}_{H^p}=\brkt{\int \sup_{t>0} \abs{\widehat{\Theta}(tD) f(x)}^p\dd x}^{\frac 1 p}<+\infty.
\end{equation*}
The local Hardy space $h^p$ for $0<p<\infty$,  is the space of tempered distributions $f$ such that
\begin{equation}\label{eq:h1}
	\norm{f}_{h^p}=\brkt{\int \sup_{0<t<1} \abs{\widehat{\Theta}(tD) f(x)}^p\dd x}^{\frac 1 p}<+\infty.
\end{equation}
We should also mention that these definitions do not depend on the choice of $\Theta$.
It is known (see e.g. \cites{S,FS}) that
\[
	\norm{f}_{H^p}\thickapprox \brkt{\int \sup_{t>0}\sup_{\abs{x-y}<t} \abs{\widehat{\Theta}(tD) f(y)}^p\dd x}^{\frac{1}{p}}.
\]
A similar characterisation is valid for the local Hardy space $h^p$, namely
\begin{equation}\label{eq:h1}
	\norm{f}_{h^p}\thickapprox \brkt{\int \sup_{0<t<1}\sup_{\abs{x-y}<t} \abs{\widehat{\Theta}(tD) f(y)}^p\dd x}^{\frac{1}{p}}.
\end{equation}
In particular it is clear that $H^p$ is a subspace of $h^p$ and, for any $f\in H^p$, $\norm{f}_{h^p}\leq \norm{f}_{H^p}$. It is also well-known that, for $1<p<\infty$, $H^p$ coincides with $L^p$ and the two norms are equivalent.

We refer the reader to the work of D.~Goldberg \cite{G} and the paper of C.~Fefferman and E.M.~Stein \cite{FS} for further properties of $h^p$ and $H^p$ respectively.

The dual of $h^1$ is the space $\bmo$ which is the space of locally integrable functions for which
\begin{equation}\label{eq:bmo}
	\norm{f}_{\bmo}:=\norm{f}_{\mathrm{BMO}}+\norm{\widehat{\Theta}(D) f}_{L^\infty}<+\infty,
\end{equation}
see e.g.~\cites{Taylor, Gol}. Moreover, as in the definition of the $h^1$ norm, different choices of $\Theta$ produce equivalent norms.

\begin{thm}\label{linear_FIO}
      Let $T^\phase_\sigma$ be a Fourier integral operator given by  \eqref{eq:Linear_FIO} with an amplitude $\sigma(x,\xi)$ which is compactly supported in the $x$ variable and a non-degenerate phase function $\varphi(x,\xi)$ as above. If $0<p\leq\infty$, $m\leq -(n-1)\abs{\frac{1}{2}-\frac{1}{p}}$ and $\sigma(x,\xi) \in S^m_{1, 0}$, then there exists a constant $C>0$ such that
          \begin{enumerate}[label=\bf{T.{\arabic*}}]
      		\item\label{eq:SSS} $\Vert T^\phase_\sigma f\Vert_{L^{p}} \leq C \Vert f\Vert_{L^{p}}$ when $1<p<\infty$,
		\item \label{eq:SSS_infty} $\Vert T^\phase_\sigma f\Vert_{\mathrm{BMO}} \leq C \Vert f\Vert_{L^\infty}$ when $p=\infty$,
		\item \label{eq:SSS_1} $\Vert T^\phase_\sigma f\Vert_{L^{1}} \leq C \Vert f\Vert_{H^1}$ when $p=1$, and
		\item \label{eq:Peloso} $\Vert T^\phase_\sigma f\Vert_{h^{p}} \leq C \Vert f\Vert_{h^p}$ when $0<p\leq 1$.
	\end{enumerate}
If $\sigma(x,\xi) \in L^\infty S^m_{1}$ and $m< -(n-1)\abs{\frac{1}{2}-\frac{1}{p}}$, then
     \begin{enumerate}[label=\bf{T.{\arabic*}}]\setcounter{enumi}{4}
     \item \label{eq:DSS} $\Vert T^\phase_\sigma f\Vert_{L^{p}} \leq C \Vert f\Vert_{L^{p}}$ when $1\leq p\leq \infty$.
	\end{enumerate}
\end{thm}
\begin{proof}
Statements \ref{eq:SSS}, \ref{eq:SSS_infty}, \ref{eq:SSS_1} are classical results of A. Seeger, C. Sogge and E. Stein \cite{SSS}. The case $p=2$ is due to L. H\"ormander \cite{H2}.  Assertion \ref{eq:Peloso} is proved by M. Peloso and S. Secco in \cite{PS}.

For \ref{eq:DSS} observe that if $\mathcal{K}$ denotes the $x$-support of the amplitude, the continuity, homogeneity and non-degeneracy of the phase function yield
\begin{equation}\label{eq:weak_SND}
	\inf_{(x,\xi)\in \mathcal{K}\times (\R^{n}\setminus\{0\})} \abs{\det \brkt{\frac{\d^2 \phase}{\d x \d\xi}}}=c_\mathcal{K}>0.
\end{equation}
For the same reason, for any pair of multi-indices $\alpha$ and $\beta$, there exists a positive constant $C_{\alpha, \beta}$ such that
   \begin{align*}
      \sup_{(x,\,\xi) \in \mathcal{K}\times(\R^n \setminus \{0\})}  |\xi| ^{-1+\vert \alpha\vert}\vert \partial_{\xi}^{\alpha}\partial_{x}^{\beta}\phase(x,\xi)\vert
      \leq C_{\alpha , \beta}.
   \end{align*}
Now a careful examination of the proofs of Theorems 2.1.1, 2.2.1 and 2.3.1 in the paper of D. Dos Santos Ferreira and W. Staubach \cite{DSFS} yield the desired boundedness.
\end{proof}

\begin{defn}
A {bilinear oscillatory integral operator} $T^{\phase_1,\phase_2}_{\sigma}$ is an operator which is defined to act on Schwartz functions $f$ and $g$ by the formula
\begin{equation}\label{defn RS FIO}
T^{\phase_1,\phase_2}_{\sigma}(f,g)(x) = \iint \sigma(x,\xi,\eta)\widehat{f}(\xi)\widehat{g}(\eta)e^{i\phase_1(x,\xi)+i \phase_2(x,\eta)} \ddd\xi \ddd\eta.
\end{equation}
\end{defn}

Now we state the main theorem of this paper.\\

\begin{thm} \label{main}
Let $1\leq p,q\leq \infty$ and $0\leq r\leq \infty$ be such that $1/r=1/p+1/q$. Suppose that
\begin{equation}\label{eq:theorder}
	m\leq -(n-1)\brkt{\abs{\frac{1}{2}-\frac{1}{p}}+\abs{\frac{1}{2}-\frac{1}{q}}},		
\end{equation}
$\sigma \in S^{m}_{1,0}(n,2)$ is compactly supported in the spatial variable and $\phase_1,\phase_2$ are non-degenerate phase functions. Then there exists a constant $C$, depending only on a finite number of derivatives of $\sigma$, such that
\begin{equation}\label{eq:between}
\|T_\sigma^{\phase_1,\phase_2}(f,g)\|_{L^{r}(\R^n)} \leq C\|f\|_{L^p(\R^n)}\|g\|_{L^q(\R^n)}, \quad \mbox{for $1<p,q<\infty$}
\end{equation}
and all Schwartz functions $f$ and $g$. In particular, if $m=0$ we have
\begin{equation}
	\label{eq:TT1}\|T_\sigma^{\phase_1,\phase_2}(f,g)\|_{L^1(\R^n)} \leq C\|f\|_{L^2(\R^n)}\|g\|_{L^2(\R^n)}.
\end{equation}
Moreover, if $n\geq 2$, one has the extremal estimates
\begin{eqnarray}
	\label{eq:TIT}\|T_\sigma^{\phase_1,\phase_2}(f,g)\|_{L^2(\R^n)} \leq C\|f\|_{L^\infty(\R^n)}\|g\|_{L^2(\R^n)},\\
	\label{eq:ITT}\|T_\sigma^{\phase_1,\phase_2}(f,g)\|_{L^2(\R^n)} \leq C\|f\|_{L^2(\R^n)}\|g\|_{L^\infty(\R^n)},\\
	\label{eq:FFt1}\|T_\sigma^{\phase_1,\phase_2}(f,g)\|_{L^{\frac{2}{3}}(\R^n)} \leq C\|f\|_{L^2(\R^n)}\|g\|_{H^1(\R^n)},\\
	\label{eq:FtF1}\|T_\sigma^{\phase_1,\phase_2}(f,g)\|_{L^{\frac{2}{3}}(\R^n)} \leq C\|f\|_{H^1(\R^n)}\|g\|_{L^2(\R^n)},
\end{eqnarray}
for $m=-(n-1)/2$, and
\begin{eqnarray}
	\label{eq:IIBmo}\|T_\sigma^{\phase_1,\phase_2}(f,g)\|_{\mathrm{BMO}(\R^n)} \leq C\|f\|_{L^\infty(\R^n)}\|g\|_{L^ \infty(\R^n)},\\
	\label{eq:IH1}\|T_\sigma^{\phase_1,\phase_2}(f,g)\|_{L^1(\R^n)} \leq C\|f\|_{L^\infty(\R^n)}\|g\|_{H^1(\R^n)},\\
	\label{eq:HI1}\|T_\sigma^{\phase_1,\phase_2}(f,g)\|_{L^1(\R^n)} \leq C\|f\|_{H^1(\R^n)}\|g\|_{L^\infty(\R^n)},\\
	\label{eq:HHhalf}\|T_\sigma^{\phase_1,\phase_2}(f,g)\|_{L^{\frac{1}{2}}(\R^n)} \leq C\|f\|_{H^1(\R^n)}\|g\|_{H^1(\R^n)},
\end{eqnarray}
for $m=-(n-1)$. Except \eqref{eq:TT1}, \eqref{eq:FFt1},\eqref{eq:FtF1} and \eqref{eq:HHhalf}, the previous estimates do not hold in general in the case $n=1$.
\end{thm}

\begin{figure}[!htb]
\centering
\includegraphics[scale=.8]{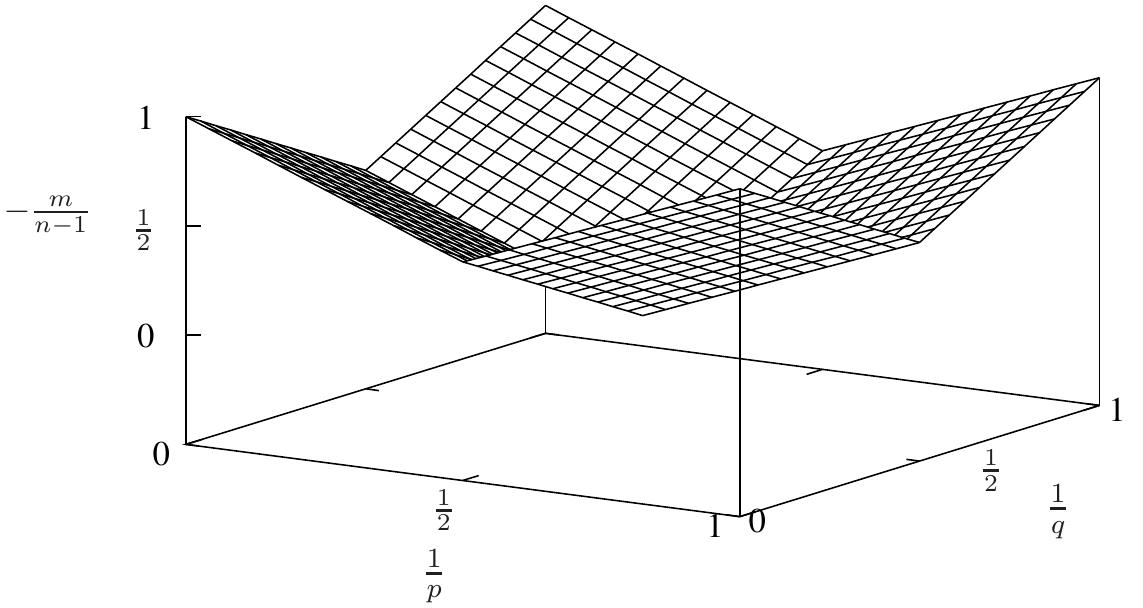}
\caption{Surface $\frac{-m}{n-1}={\abs{\frac{1}{2}-\frac{1}{p}}+\abs{\frac{1}{2}-\frac{1}{q}}}$.}
\label{fig:digraph}
\end{figure}
In light of the results in \cite{RS}, it is enough to prove the theorem for the case where we have an equality in \eqref{eq:theorder} (see Figure \ref{fig:digraph}). The proof will proceed by considering separately the low and high frequency parts of the amplitude.

The estimates on the low frequency part can more or less all be handled at the same time (see Section \ref{sec:low_frequency}).

The estimates for the high frequency parts when $n\geq 2$ will be carried out by obtaining the endpoint estimates \eqref{eq:TT1}-\eqref{eq:HHhalf}, which correspond to the  vertices of the surface depicted in Figure \ref{fig:digraph}.  Figure \ref{fig:Levels} shows the level sets of the surface for the three main values of $m$ corresponding to the endpoints above. Complex interpolation techniques then yield the intermediate cases.

In dimension $n=1$ the boundedness $L^\infty\times L^p\to L^p$, $L^p\times L^\infty\to L^p$ for $1<p<\infty$, $L^\infty\times H^1\to L^1$, $H^1\times L^\infty\to L^1$ and $L^\infty \times L^\infty \to \mathrm{BMO}$ all fail for general amplitudes of order $m=0$ (see Section \ref{sect:1d} for counterexamples in these cases). The positive results in this case are actually a consequence of the boundedness of bilinear paraproducts. Indeed, our method will reduce the study of one-dimensional bilinear Fourier integral operators to that of a composition of a bilinear paraproduct with two linear FIOs.

We will divide our analysis of the high frequency parts into two: First, the cases where the target space is Banach, i.e.~\eqref{eq:TT1}, \eqref{eq:TIT}, \eqref{eq:ITT}, \eqref{eq:IIBmo}, \eqref{eq:IH1} and \eqref{eq:HI1} (see Section \ref{sect:high}) and secondly the cases where the target spaces are non-Banach, i.e.~\eqref{eq:FFt1}. \eqref{eq:FtF1} and \eqref{eq:HHhalf}.

However, before proceeding with the detailed study, we shall describe how these end\-point estimates yield the general result. Real interpolation yields the $L^p\times L^q\to L^r$ boundedness in the cases that $m=-(n-1),-(n-1)/2$ and $0$, whenever $p,q$ satisfy
\[
	m= -(n-1)\brkt{\abs{\frac{1}{2}-\frac{1}{p}}+\abs{\frac{1}{2}-\frac{1}{q}}}
\]
and $1/r=1/p+1/q$ (see Figure \ref{fig:Levels}).
\begin{figure}[!htb]
\centering
\includegraphics[scale=.8]{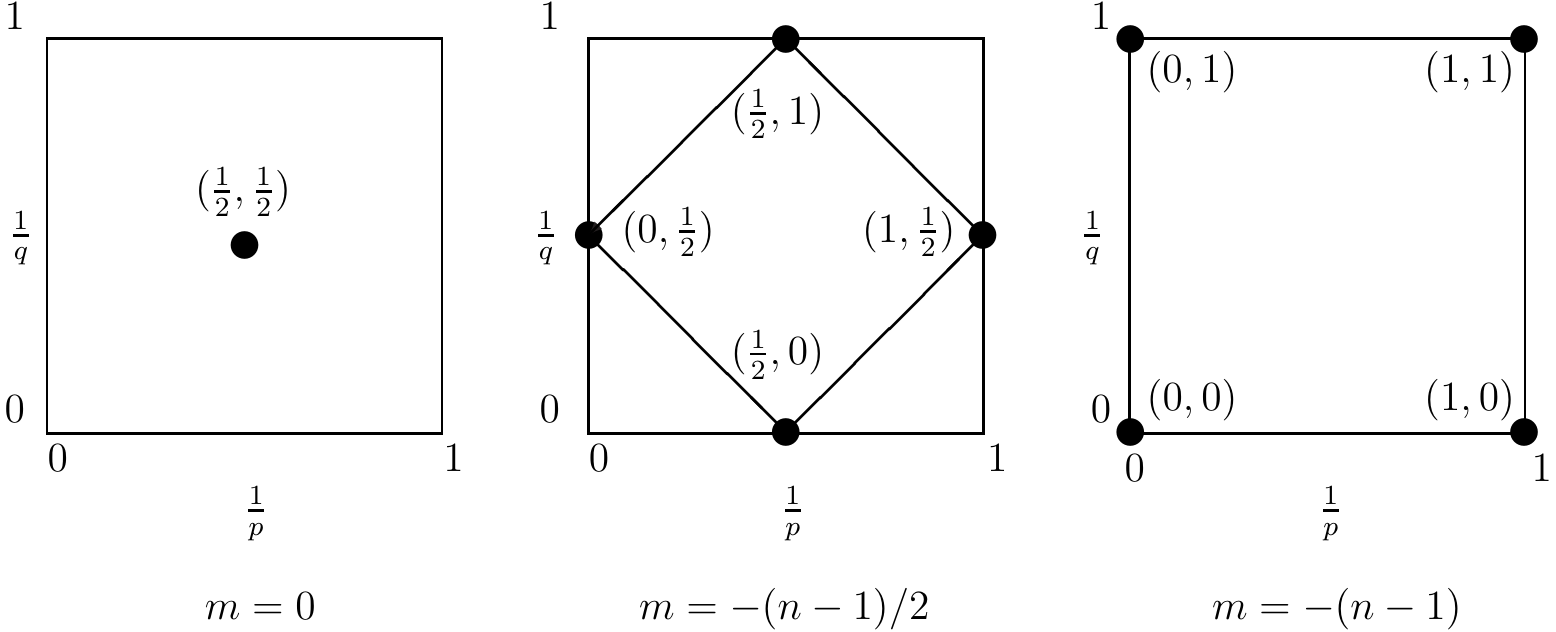}
\caption{Level sets for the surface $\frac{-m}{n-1}={\abs{\frac{1}{2}-\frac{1}{p}}+\abs{\frac{1}{2}-\frac{1}{q}}}$. }
\label{fig:Levels}
\end{figure}
For the other values of $m$, the result follows by bilinear complex interpolation. Here we only mention briefly how this works in the case $L^2\times L^q\to L^r$ with $2\leq q\leq \infty$. The other cases are treated in a similar way.

We fix $-(n-1)/2<m<0$ and $\sigma\in S^{m}_{1,0}(n,2)$, and let $q$ be such that $m=-(n-1)\brkt{\frac{1}{2}-\frac{1}{q}}$. As in the case of linear FIOs, one defines an analytic family of operators $\{T_z\}$ on the strip $0<\Re(z)\leq 1$ by
\[
	 T_z(f,g)(x)=\iint \sigma_z(x,\xi,\eta)\widehat{f}(\xi)\widehat{g}(\eta)e^{i\phase_1(x,\xi)+i \phase_2(x,\eta)} \ddd\xi \ddd\eta,
\]
where $\sigma_z(x,\xi,\eta)=\sigma(x,\xi,\eta)\brkt{1+\abs{\xi}^2+\abs{\eta}^2}^{\frac{\gamma(z)}{2}}$,  $\gamma(z)=-m-\frac{z(n-1)}{2}$. Now, when $\Re(z)=0$ and $\Re(z)=1$, one has that $\sigma_z\in S^0_{1,0}(n,2)$ and
$\sigma_z\in S^{\frac{-(n-1)}{2}}_{1,0}(n,2)$ respectively. In both cases, the seminorms depend polynomially on $\abs{\Im z}$. Using the $L^2\times L^2\to L^1$ boundedness for $\Re(z)=0$,  the $L^2\times L^\infty\to L^2$ boundedness for $\Re(z)=1$ and the interpolation theorem for analytic families of multilinear operators by L. Grafakos and M. Masti{\l}o \cite{GM}*{Theorem 1.1}, the desired $L^{2}\times L^{q}\to L^r$ boundedness follows.

We should mention that \cite{GM}*{Theorem 1.1} also applies to multilinear interpolation in the case of non-Banach $L^p$ spaces. \\

Now we shall turn to the proof of the aforementioned end-point estimates. Since the phase functions $\phase_j$ are smooth, homogeneous and non-degenerate, the mean value theorem yields that there exists a constant $C_1>0$ such that \begin{equation*}
	\abs{\nabla_x \phase_j(x,\xi)}\leq C_1\abs{\xi} \quad {\text{for any $\xi\in \R^n\setminus \{0\}$ and $x\in \mathcal{K}$.}}
\end{equation*}
On the other hand, the strong non-degeneracy condition on the phases, yields that there exists a constant $C_2>0$ \begin{equation*} C_2\abs{\xi}\leq\abs{\nabla_x \phase_j(x,\xi)} \quad {\text{for any $\xi\in \R^n\setminus \{0\}$ and $x\in \mathcal{K}$}}. \end{equation*}
Therefore, there exists $0<\lambda\leq 1$ such that for any $\xi\in \R^n$ and $x\in \mathcal{K}$
\begin{equation}\label{asymptotic condition1}
\lambda |\xi| \leq |\nabla_x\phase_j(x,\xi)| \leq \lambda^{-1}|\xi| \quad \mbox{for $j=1,2$,}
\end{equation}\\

We will prove Theorem \ref{main} in the following few sections. To this end, we introduce a smooth function $\mu \colon \R^n \to \R$ such that $\mu(\xi) = 0$ for $|\xi| \leq 1/(4\lambda)$ and $\mu(\xi) = 1$ for $|\xi| \geq 1/(3\lambda)$, where $\lambda$ is as in \eqref{asymptotic condition1}. Observe that
\begin{equation}\label{eq:first_step}
\begin{aligned}
	T_{\sigma}^{\phase_1,\phase_2}(f,g) & = T_{\sigma}^{\phase_1,\phase_2}(\mu(D)f,\mu(D)g)+T_{\sigma}^{\phase_1,\phase_2}(\mu(D)f,(1-\mu)(D)g)\\
& \qquad +T_{\sigma}^{\phase_1,\phase_2}((1-\mu)(D) f,g).
\end{aligned}
\end{equation}

\section{The low frequency cases}\label{sec:low_frequency}

To estimate the $L^p$-norm of the last two terms in \eqref{eq:first_step}  we can make use of linear boundedness results by viewing the bilinear operator as an iteration of linear operators.
We confine ourselves to the proof for the third term of \eqref{eq:first_step} since the second term is treated similarly.
Before proceeding with the analysis of the low frequency portion of the operator, we state and prove the following lemma.
\begin{lem} \label{lem:BMO_embedding}Let $a\in {\rm BMO}(\R^n)$ be such that it is supported in the cube $Q_R= [-R,R]^n$. Then, for any $1\leq q<\infty$ there exists a constant $\mathfrak{c}_{R,q,n}$ depending only on $R,q$ and $n$ such that
\[
	\norm{a}_{L^q(\R^n)}\leq \mathfrak{c}_{R,q,n} \norm{a}_{{\rm BMO}(\R^n)}.
\]
\begin{proof}
Observe that
\[
	\begin{split}
	&\left|\frac{1}{\abs{Q_R}}\int_{Q_R} a(x)\dd x -\frac{1}{\abs{Q_{2R}}}\int_{Q_{2R}} a(x)\dd x\right| \leq \frac{1}{\abs{Q_{R}}}\int_{Q_{R}}
	\abs{a(x)-{\rm Avg}_{Q_{2R}}a}\dd x \\
	&\qquad\leq %
	\frac{2^n}{\abs{Q_{2R}}}\int_{Q_{2R}}
	\abs{a(x)-{\rm Avg}_{Q_{2R}}a}\dd x \leq 2^n \norm{a}_{{\rm BMO}(\R^n)}.
	\end{split}
\]
On the other  hand, since $a$ is supported in $Q_R$,
\[
	\abs{\frac{1}{\abs{Q_R}}\int_{Q_R} a(x)\dd x -\frac{1}{\abs{Q_{2R}}}\int_{Q_{2R}} a(x)\dd x}=
	\abs{\int_{Q_R} a(x)\dd x}\abs{\frac{1}{\abs{Q_R}}-\frac{1}{\abs{Q_{2R}}}}=\abs{{\rm Avg }_{Q_R}\, a}(1-2^{-n}),
\]
which yields
\[
	\abs{{\rm Avg}_{Q_R}\, a}\leq \frac{2^{2n}}{2^n-1} \norm{a}_{{\rm BMO}(\R^n)}.
\]
Therefore, for every $1\leq q<\infty$,
\[
	\begin{split}
	\norm{a}_{q} &\leq \brkt{\int_ {Q_R} \abs{a(x)-{\rm Avg}_{Q_R} a}^q\dd x}^{1/q}+\abs{{\rm Avg}_{Q_R} a}\abs{Q_R}^{1/q}\\
	&\leq \abs{Q_R}^{1/q}\brkt{\sup_Q \brkt{\frac{1}{\abs{Q}}\int_ {Q} \abs{a(x)-{\rm Avg}_{Q} a}^q\dd x}^{1/q}+\abs{{\rm Avg}_{Q_R} a}}\\
	&\leq \abs{Q_R}^{1/q} (c_q+c_n) \norm{a}_{{\rm BMO}(\R^n)}.
	\end{split}
\]
\end{proof}
\end{lem}
\subsection*{Remark} Observe that the constant $c_q$ in the proof of the last theorem blows up as $q$ goes to infinity (see \cite[p. 528]{G}) .

Next we prove the regularity of the frequency-localised part of the operator.
\begin{thm}  \label{main_low} Let $1\leq p,q\leq \infty$, $0<r\leq \infty$ and $m$ be as in \eqref{eq:theorder} and suppose that $\sigma \in S^m_{1,0}(n,2)$  is compactly supported in the spatial variable $x$ and in one of the frequency variables. Assume that $\phase_1,\phase_2$ are real-valued $C^\infty(\R^n\times( \R^n\setminus\{0\}))$ functions homogeneous of degree one in the frequency variable on the support of $\sigma$ and satisfy the non-dengeneracy condition. Then there exists a constant $C$ such that
\[
	\|T_\sigma^{\phase_1,\phase_2}(f,g)\|_{L^r(\R^n)} \leq C\|f\|_{L^{p}(\R^n)}\|g\|_{L^{q}(\R^n)}
\]
for all Schwartz functions $f$ and $g$. If $p$ or $q$ are equal to $1$, the above inequality holds with the Lebesgue space $L^1$ replaced by and $H^1$.
\end{thm}
\begin{proof}
Without loss of generality we assume that the amplitude $\sigma(x,\xi,\eta)$ is compactly supported in the $x$ and $\xi$ variables. Let $\psi$ be a smooth cut-off function that is equal to one on the $\xi$ support of $\sigma$.
 We can write $T_{\sigma}^{\phase_1,\phase_2}(f,g)$  as
\begin{align*}
T_{\sigma}^{\phase_1,\phase_2}(f,g)(x) & = \int\left(\int \sigma(x,\xi,\eta)\widehat{g}(\eta)e^{i\phase_2(x,\eta)} \ddd\eta\right)\psi(\xi)\widehat{f}(\xi)e^{i\phase_1(x,\xi)} \ddd\xi\\
& = \int \mathfrak{a}_g(x,\xi) \widehat{f}(\xi)e^{i\phase_1(x,\xi)} \ddd\xi= T^{\phase_1}_{\mathfrak{a}_g}(f)(x),
\end{align*}
where
$
\mathfrak{a}_g(x,\xi) = \psi(\xi)\int \sigma(x,\xi,\eta)\widehat{g}(\eta)e^{i\phase_2(x,\eta)} \ddd\eta.
$

We need to show that
\[
\|T^{\phase_1}_{\mathfrak{a}_g}(f)\|_{L^{r}(\R^n)} \lesssim \|f\|_{L^{p}(\R^n)}\|g\|_{L^{q}(\R^n)}.
\]

Let $Q$ be a closed cube of side-length $L$ such that $\supp_{\xi} \mathfrak{a}_g(x,\xi) \subset \text{Int}(Q)$. We extend $\mathfrak{a}_g(x,\cdot)|_{Q}$ periodically with period $L$ to $\widetilde{\mathfrak{a}}_g(x,\xi)\in C^{\infty}(\R^{n}_\xi).$ Let $\zeta (\xi)$ be in $C^{\infty}_{0}$ with $\supp \zeta \subset Q$ and $\zeta=1$ on $\xi$-support of $\mathfrak{a}_g(x,\xi)$. Clearly, we have $\mathfrak{a}_g(x,\xi)=\widetilde{\mathfrak{a}}_g(x,\xi) \zeta(\xi)$. Now if we expand $\widetilde{\mathfrak{a}}_g(x,\xi) $ in a Fourier series, then setting $f_k(x)=f(x-\frac{2\pi k}{L})$ for any $k\in \Z^n$ we can write the FIO $T_{\mathfrak{a}_g}^{\phase_1}$ as
\begin{equation}\label{eq:Fourier_Serie}
	T_{\mathfrak{a}_g}^{\phase_1} f(x)=\sum_{k\in \Z^n} a_k(x) T_\zeta^{\phase_1} (f_k)(x),
\end{equation}
where
\[
	a_k(x)=\frac{1}{L^n}\int_{\R^n} \mathfrak{a}_g(x,\xi) e^{-i \frac{2\pi}{L}\p{k,\xi}}\, \ddd \xi,
\]
and $T^{\phase_1}_\zeta(v)(x):=\frac{1}{(2\pi)^n}\int \zeta(\xi) e^{i {\phase_1}(x,\xi)}\widehat{v}(\xi)\, \ddd \xi.$
Then integration by parts yields
\[
	a_k(x)= \frac{c_{n,N}}{|k_{l}|^N}\int_{\R^n} \partial^N_{\xi_l} \mathfrak{a}_g(x,\xi) e^{-i  \frac{2\pi}{L}\p{k,\xi}}\, \ddd \xi,
\]
for $l=1,\ldots,n$.

Assume first that $p\neq 2$, then $m<-(n-1)\abs{\frac{1}{q}-\frac{1}{2}}$. Since $\sigma \in S^m_{1,0}(n,2)$,
\[
	\sup_{x,\xi,\eta\in \R^n}\p{\eta}^{-m+\abs{\beta}}|\partial^\alpha_\xi\partial^\beta_\eta\partial^\gamma_x\sigma(x,\xi,\eta)| \lesssim 1.
\]
So, by {\ref{eq:SSS}} in Theorem \ref{linear_FIO}
and taking into account that  the $x$ and $\xi$ supports of $\mathfrak{a}_g$ are compact, we find that for any $M\leq 0$
\begin{equation}\label{eq:local_good}
\sup_{\xi\in \R^n} \langle \xi\rangle^{-M}\|\partial_\xi^\alpha \mathfrak{a}_g(\cdot,\xi)\|_{L^{q}(\R^n)}\leq c_\alpha \|g\|_{L^{q}(\R^n)},\quad \hbox{if  $1\leq q\leq \infty$}.
\end{equation}
Therefore
\begin{equation*}
	\max_{s=0,\ldots, N}\int_{\R^n} \norm{\partial^s_{\xi_{l}} \mathfrak{a}_g(\cdot,\xi)}_{L^{q} (\R^n)}\, \ddd\xi\leq c_{n, N}\norm{g}_{L^{q}(\R ^n)},
\end{equation*}
which yields
\begin{equation}\label{estimak}
	\norm{a_k}_{L^{q}}\leq c_{n,N,q}\norm{g}_{L^{q}(\R ^n)} (1+\abs{k})^{-N},
\end{equation}
for any $N\geq 0$. Let $\chi$ be a smooth cut-off function in the $x$ variable such that is equal to $1$ on the support of all  the $a_k(x)$. Assume first that $r\geq 1$. By the Minkowski and H\"older inequalities, we have
\begin{equation}\label{eq:bound}
\begin{split}
	\norm{T_{\mathfrak a_g}^{\phase_1} f}_{L^r(\R^n)}\leq  \sum_{k\in \Z^n}  \norm{a_k T_{\zeta}^{\phase_1}(f_k)}_{L^r(\R^n)}
\leq  \sum_{k\in \Z^n} \norm{a_k}_{L^{q}(\R^n)} \norm{\chi T_{\zeta}^{\phase_1}(f_k)}_{L^{p}(\R^n)}.
\end{split}
\end{equation}
Since $\chi(x)\zeta(\xi)\in S^{-\infty}(n,1)$ with compact $x$ and $\xi$ support, {\ref{eq:SSS}} yields
\[
	\norm{\chi T_\zeta^{\phase_1} (f_k)}_{L^{p}}\lesssim \norm{f_k}_{L^{p}}=\norm{f}_{L^{p}}, \quad \text{for $1\leq p\leq \infty$.}
\]
Thus \eqref{estimak} and \eqref{eq:bound}  yield
\[
\begin{split}
\|T_{\sigma}^{\phase_1,\phase_2}(f, g)\|_{L^r(\R^n)} &\lesssim \norm{g}_{L^{q}(\R^n)} \sum_{k\in \Z^n} (1+|k|)^{-N} \norm{f}_{L^{p}(\R^n)}\lesssim \norm{g}_{L^{q}(\R^n)}\norm{f}_{L^{p}(\R^n)}.
\end{split}
\]
For $r<1$, following the same argument, one can prove that
\begin{equation*} %
\begin{split}
	\norm{T_{\mathfrak a_g}^{\phase_1} f}_{L^r(\R^n)}^r\leq  \sum_{k\in \Z^n}  \norm{a_k T_{\zeta}^{\phase_1}(f_k)}_{L^r(\R^n)}^r
\leq  \sum_{k\in \Z^n} \norm{a_k}_{L^{q}(\R^n)}^r \norm{\chi T_{\zeta}^{\phase_1}(f_k)}_{L^{p}(\R^n)}^r,
\end{split}
\end{equation*}
and arguing in a similar way we obtain the result.

Now, if $p=2$  then $m=-(n-1)\abs{\frac{1}{q}-\frac{1}{2}}$, hence by Theorem \ref{linear_FIO} we have for any $M\leq 0$
\begin{eqnarray}
\label{case_p1} \sup_{\xi\in \R^n} \langle \xi\rangle^{-M}\|\partial_\xi^\alpha \mathfrak{a}_g(\cdot,\xi)\|_{L^{q}(\R^n)}&\leq c_\alpha \|g\|_{L^{q}(\R^n)}\quad &\hbox{if  $1<q<\infty$},\\
\label{case_1}
	\sup_{\xi\in \R^n} \langle \xi\rangle^{-M}\|\partial_\xi^\alpha \mathfrak{a}_g(\cdot,\xi)\|_{L^{1}(\R^n)} &\leq c_\alpha \|g\|_{{H^1}(\R^n)} \quad &\hbox{if $q=1$, and}\\
\label{BMO_class}
	\sup_{\xi\in \R^n} \langle \xi\rangle^{-M}\|\partial_\xi^\alpha \mathfrak{a}_g(\cdot,\xi)\|_{{\rm BMO}(\R^n)}&\leq c_\alpha \|g\|_{{L^\infty}(\R^n)} \quad &\hbox{if $q=\infty$}.
\end{eqnarray}
The case $1\leq q<\infty$ can now be treated in the same way as before (where if $q=1$ we replace the $L^1$ norm of $g$ by its $H^1$ norm), with \eqref{case_p1}-\eqref{BMO_class} replacing \eqref{eq:local_good}.

It remains to prove the case $q=\infty$. To this end we observe that by \eqref{BMO_class}, we have
\begin{equation*}
	\max_{s=0,\ldots, N}\int_{\R^n} \norm{\partial^s_{\xi_{l}} \mathfrak{a}_g(\cdot,\xi)}_{{\rm BMO} (\R^n)}\, \ddd\xi\leq c_{n, N}\norm{g}_{L^\infty(\R ^n)}.
\end{equation*}
Therefore, using Lemma \ref{lem:BMO_embedding} with $q=4$ yields
\begin{equation}\label{estimak2}
	\norm{a_k}_{L^{4}}\leq c_{n,N}\norm{g}_{L^\infty(\R ^n)} (1+\abs{k})^{-N},
\end{equation}
for any $N\geq 0$. By the Minkowski and H\"older inequalities,
\begin{equation}\label{eq:bound2}
\begin{split}
	\norm{T_{\mathfrak a_g}^{\phase_1} f}_{L^2(\R^n)}\leq  \sum_{k\in \Z^n}  \norm{a_k T_{\zeta}^{\phase_1}(f_k)}_{L^2(\R^n)}
\leq  \sum_{k\in \Z^n} \norm{a_k}_{L^{4}(\R^n)} \norm{\chi T_{\zeta}^{\phase_1}(f_k)}_{L^{4}(\R^n)}.
\end{split}
\end{equation}
Let $s=n/4$, and observe that $\chi(x)\zeta(\xi)\p{\xi}^{\frac{s}2}\in S^{-\infty}(n,1)$. Then {\ref{eq:DSS}} in Theorem \ref{linear_FIO} and the Sobolev embedding theorem yield
\[
	 \norm{\chi T_{\zeta}^{\phase_1}(f_k)}_{L^{4}(\R^n)}= \norm{\chi T_{\zeta \p{\cdot}^{\frac s 2}}^{\phase_1}((1-\Delta )^{-\frac s 2}f_k)}_{L^{4}(\R^n)}\lesssim \norm{(1-\Delta )^{-\frac s 2}f_k}_{L^{4}(\R^n)}\lesssim \norm{f}_{L^{2}(\R^{n})}.
\]
Therefore \eqref{estimak2} and \eqref{eq:bound2}  yield
\[
\|T_{\sigma}^{\phase_1,\phase_2}(f, g)\|_{L^2(\R^n)} \lesssim \norm{g}_{L^\infty(\R^n)} \sum_{k\in \Z^n} (1+|k|)^{-N} \norm{f}_{L^{2}(\R^n)}\lesssim \norm{g}_{L^\infty(\R^n)}\norm{f}_{L^{2}(\R^n)}.
\]
\end{proof}

\section{The high frequency case: Banach endpoints}\label{sect:high}

To deal with the term $T_{\sigma}^{\phase_1,\phase_2}(\mu(D)f,\mu(D)g)$ from \eqref{eq:first_step} let us introduce two smooth cut-off functions $\chi,\nu \colon \R^{2n} \to \R$, such that $\chi(\xi,\eta) = 1$ for $|(\xi,\eta)| \leq 1$ and $\chi(\xi,\eta) = 0$ for $|(\xi,\eta)| \geq 2$, and $\nu(\xi,\eta) = 0$ for $\lambda^2|\xi| \leq 16 |\eta|$ and $\nu(\xi,\eta) = 1$ for $64|\eta| \leq \lambda^2|\xi|$.

Defining
\begin{align*}
\sigma_1(x,\xi,\eta) & = (1-\chi(\xi,\eta))\nu(\xi,\eta)\sigma(x,\xi,\eta) \quad \text{and} \\
\sigma_2(x,\xi,\eta) & = (1-\chi(\xi,\eta))(1-\nu(\xi,\eta))\sigma(x,\xi,\eta),
\end{align*}
we have that $\sigma_1,\sigma_2 \in S^m_{1,0}(n,2)$ and we can decompose
\begin{equation} \label{low-high}
T_{\sigma}^{\phase_1,\phase_2}(\mu(D)f,\mu(D)g) = T_{\sigma_1}^{\phase_1,\phase_2}(\mu(D)f,\mu(D)g)+ T_{\sigma_2}^{\phase_1,\phase_2}(\mu(D)f,\mu(D)g).
\end{equation}

One introduces an even real-valued smooth function $\psi$ whose Fourier transform is supported on the annulus $\{\xi \, | \, 1/2 \leq |\xi| \leq 2\}$ such that
\begin{equation}\label{eq:annulus}
\int_0^\infty |\widehat{\psi}(t\xi)|^2 \frac{\dd t}{t} = 1
\end{equation}
for $\xi \neq 0$. Let $\theta$ be another real-valued smooth function whose Fourier transform is equal to one on the ball $\{\xi \, | \, |\xi| \leq 1/8\}$ and supported in $\{\xi \, | \, |\xi| \leq 1/4\}$.
Here and in the sequel we denote $\mathcal{K} = \supp_x \sigma$.

We have that
\begin{equation} \label{rep1}
\begin{aligned}
& T_{\sigma_1}^{\phase_1,\phase_2}(\mu(D)f,\mu(D)g)(x) \\
& = \iiint_0^\infty\ \sigma_{1,t}(x,t\xi,t\eta)\left(\widehat{\psi}(t\nabla_x\phase_1(x,\xi))\abs{\nabla_x\phase_1(x,\xi)}^{m}\mu(\xi)\widehat{f}(\xi)e^{i\phase_1(x,\xi)}\right) \times \\
& \quad \times \left(\widehat{\theta}(t\nabla_x\phase_2(x,\eta))\mu(\eta)\widehat{g}(\eta)e^{i\phase_2(x,\eta)}\right) \frac{\dd t}{t}\ddd\xi\ddd \eta,
\end{aligned}
\end{equation}
for
\[
\sigma_{1,t}(x,\xi,\eta) := t^{m}\sigma_1(x,\xi/t,\eta/t)\abs{\nabla_x\phase_1(x,\xi)}^{-m}\widehat{\psi}(\nabla_x\phase_1(x,\xi))\widehat{\theta}(\nabla_x\phase_2(x,\eta)).
\]
We define
\begin{equation}\label{eq:newsigma1}
	\newsigma_1(t,x,\xi,\eta):=\sigma_{1,t}(x,\Psi_1(x,\xi), \Psi_2(x,\eta)),
\end{equation}
where
\[
	\Psi_j(x,\xi)=\brkt{\nabla_x\varphi_j\brkt{x,\cdot}}^{-1}(\xi), \quad j=1,2,
\]
whose existence follows from the assumption of compact support in $x$, homogeneity of degree one in $\xi$ and the non-degeneracy.

\begin{lem}\label{lem:technic_1} For any multiindices $\alpha$, $\beta$ and $\gamma$,
\[
	\sup_{0<t<1}\sup_{x\in \mathcal{K}}\sup_{\xi,\eta\in\R^n\setminus \{0\}}\abs{\d^\alpha_\xi \d^\beta_\eta \d^\gamma_x \newsigma_1(t,x,\xi,\eta)}\lesssim 1.
\]
\end{lem}
\begin{proof} Consider the map $(x,\xi,\eta)\mapsto (x,\Theta(x,\xi,\eta))$ where $\Theta=(\Psi_1(x,\xi),\Psi_2(x,\eta))$, which is a smooth map, positively homogeneous of degree 1 in $(\xi,\eta)$ and is such that
\[
	\abs{\xi}+\abs{\eta}\thickapprox \abs{\Theta(x,\xi,\eta)} \qquad {\text{for all } x\in \mathcal{K},\quad \xi,\eta\in \R^n\setminus \{0\}}.
\]
Since $\sigma_1(x,\xi,\eta)\in S^m_{1,0}(n,2)$, Proposition  1.1.7 in
\cite{H2} yields
\[
	\abs{\d^\alpha_\xi \d^\beta_\eta \d^\gamma_x \sigma_1(x,\Theta(x,\xi,\eta))}\lesssim (1+\abs{\xi}+\abs{\eta})^{m-\abs{\alpha}-\abs{\beta}}\quad {\text{for all } x\in \mathcal{K},\, \xi,\eta\in \R^n\setminus \{0\}}.
\]
Therefore, chain rule yields
\[
	\abs{\d^\alpha_\xi \d^\beta_\eta \d^\gamma_x \brkt{\sigma_1(x,t^{-1}\Theta(x,\xi,\eta))}}\lesssim t^{-\abs{\alpha}-\abs{\beta}}(1+t^{-1}(\abs{\xi}+\abs{\eta}))^{m-\abs{\alpha}-\abs{\beta}}\leq t^{-m}\abs{\xi}^{^{m-\abs{\alpha}-\abs{\beta}}}.
\]
The result follows by applying the Leibniz rule to the definition \eqref{eq:newsigma1} of $\newsigma_1$, and using the above estimate and the assumption of the compact support on $\xi$ and $\eta$.
\end{proof}
Using the Fourier inversion formula,
\begin{equation} \label{rep2}
\newsigma_{1}(t,x,\xi,\eta) = \iint e^{i\xi\cdot u + i\eta\cdot v} m(t,x,u,v) \frac{\dd u \dd v}{(1 + |u|^2 + |v|^2)^N}
\end{equation}
where
\begin{equation} \label{eq:exp_m}
m(t,x,u,v) := \iint e^{-i\xi\cdot u - i\eta\cdot v} \big[(1-\Delta _\xi-\Delta _\eta)^N\newsigma_{1}(t,x,\xi,\eta)\big] \ddd\xi\ddd \eta
\end{equation}
for any large fixed $N \in \N$. Since the $(\xi,\eta)$-support of $\newsigma_{1}$ is contained in a compact set independent of $t$ and $x$, Lemma \ref{lem:technic_1} yields that for all multi-indices $\gamma$
\begin{equation}\label{eq:boundedness_of_m}
	\sup_{0<t<1}\sup_{x\in \mathcal{K}}\sup_{u,v\in\R^n} \abs{\partial_x^\gamma m(t,x,u,v)}<+\infty.
\end{equation}
Then, we can write
\begin{equation} \label{rep3}
\begin{aligned}
\sigma_{1,t}(x,\xi,\eta)= \iint e^{i\nabla_x\phase_1(x,\xi)\cdot u + i\nabla_x\phase_2(x,\eta)\cdot v}m(t,x,u,v) \frac{\dd u \dd v}{(1 + |u|^2 + |v|^2)^N}.
\end{aligned}
\end{equation}
Combining \eqref{rep1} and \eqref{rep3} we arrive at the representation
\begin{equation} \label{three}
\begin{aligned}
T_{\sigma_1}^{\phase_1,\phase_2}(\mu(D)f,\mu(D)g)(x)= \iiint_0^\infty T^{\phase_1}_{\nu_1^{t,u}}(f)(x) T^{\phase_2}_{\mu_1^{t,v}}(g)(x) \frac{m(t,x,u,v)}{(1 + |u|^2 + |v|^2)^N} \frac{\dd t}{t} \dd u \dd v
\end{aligned}
\end{equation}
for any large fixed $N \in \N$. Here
\begin{align}
\label{eq:mu1}\mu_1^{t,v}(x,\eta) & = \mu(\eta)e^{it\nabla_x\phase_2(x,\eta)\cdot v}\widehat{\theta}(t\nabla_x\phase_2(x,\eta))\chi(x),\\
\label{eq:nu1}\nu_1^{t,u}(x,\xi) & = \mu(\xi)\abs{\nabla_x\phase_1(x,\xi)}^{m}e^{it\nabla_x\phase_1(x,\xi)\cdot u}\widehat{\psi}(t\nabla_x\phase_1(x,\xi))\chi(x),
\end{align}
where $\chi\in C^\infty_{0}(\R^n)$ and is equal to $1$ for $x\in \mathcal{K}$.

Since $\mu(\xi) = 0$ for $|\xi| \leq 1/(4\lambda)$ and $\widehat{\psi}(t\nabla_x\phase_1(x,\xi)) = 0$ when $|t\nabla_x\phase_1(x,\xi)| \geq 1/4$, then $\mu(\xi)\widehat{\psi}(t\nabla_x\phase_1(x,\xi)) = 0$ for $t > 1$. Consequently we know $T^{\phase_1}_{\mu_1^{t,v}}(f)(x) = 0$ for $t > 1$. Using this fact, \eqref{three} reduces to
\begin{equation} \label{three2}
\begin{aligned}
& T_{\sigma_1}^{\phase_1,\phase_2}(\mu(D)f,\mu(D)g)(x) \\
& = \iiint_0^1 T^{\phase_1}_{\nu_1^{t,u}}(f)(x) T^{\phase_2}_{\mu_1^{t,v}}(g)(x) \frac{m(t,x,u,v)}{(1 + |u|^2 + |v|^2)^N} \frac{\dd t}{t} \dd u \dd v
\end{aligned}
\end{equation}
and, repeating the argument for $\sigma_2$,
\begin{equation} \label{three1}
\begin{aligned}
& T_{\sigma_2}^{\phase_1,\phase_2}(\mu(D)f,\mu(D)g)(x) \\
& = \iiint_0^1 T^{\phase_1}_{\mu_2^{t,u}}(f)(x) T^{\phase_2}_{\nu_2^{t,v}}(g)(x) \frac{m(t,x,u,v)}{(1+|u|^2+|v|^2)^N} \frac{\dd t}{t} \dd u \dd v,
\end{aligned}
\end{equation}
where
\begin{align}
\label{eq:mu2}\mu_2^{t,u}(x,\xi) & = \mu(\xi)e^{it\nabla_x\phase_1(x,\xi)\cdot u}\widehat{\theta}_0(t\nabla_x\phase_1(x,\xi)))\chi(x), \\
\label{eq:nu2}\nu_2^{t,v}(x,\eta) & = \mu(\eta)\abs{\nabla_x\phase_2(x,\eta)}^{m}e^{it\nabla_x\phase_2(x,\eta)\cdot v}\widehat{\psi}(t\nabla_x\phase_2(x,\eta)))\chi(x),
\end{align}
where $\widehat{\theta}_0$ is supported on a ball centred at the origin, but unlike the support of $\widehat{\theta}$ this ball may be large compared to the support of $\widehat{\psi}$.

Since the target space $L^r$ is Banach, it is enough to see the desired boundedness for the inner integrals in \eqref{three2} and \eqref{three1}, with norm growing at most polynomially in $\abs{u}$ and $\abs{v}$.
Observe that if we where in the pseudodifferential case, the amplitudes above would reduce to Fourier multipliers and both inner integrals in \eqref{three2} and \eqref{three1}, could essentially be expressed as paraproducts.

Therefore, our initial goal is to reduce matters to the study of paraproducts. However as we shall see later, this is only possible in dimension one and in the case of the $L^2\times L^2\to L^1$ boundedness of operators of order zero.

A useful tool for this reduction is the following theorem, which is a parameter dependent version of H\"ormander's result on the composition of pseudodifferential and Fourier integral operators (see e.g. \cite{H2}). The advantage here over H\"ormander's classical result, which cannot be obtained directly from the latter, is that
we are able to get a favourable estimate of the dependence on the parameter for those terms beyond the highest order term, which is of utmost importance in implementing the arguments we use in the proof of the main theorem. Furthermore, this theorem also improves substantially on the composition formula in \cite{RRS}, since all the terms (with the exception of the main one) in the resulting expansion are also symbols with better decay compared to the main term.

\begin{thm}\label{left composition with pseudo}
Let $m \leq 0$, $\varepsilon \in (0,1/2)$ and $\Omega=\R^n \times \set{\abs{\xi}>1}$. Suppose that $(x,\xi) \mapsto a_t(x,\xi)$ belongs to $S^m_{1,0}(n,1)$ uniformly in $t \in (0,1)$ and it is supported in $\Omega$, $(x,\xi) \mapsto \rho(\xi)$ belongs to $S^{0}_{1,0}(n,1)$ and $\phase \in \mathcal{C}^{\infty}(\Omega)$ is such that
\begin{enumerate}[label={\upshape (\roman*)}]

\item \label{partone}for constants $C_1, C_2 > 0$,  $C_1 |\xi|\leq |\nabla_{x} \phase(x,\xi)| \leq C_2 |\xi|$ for all $(x,\xi) \in \Omega$, and

\item \label{parttwo}for all $|\alpha|,|\beta|\geq 1$  $|\partial^{\alpha}_{x} \phase(x,\xi)|\lesssim \langle \xi \rangle$ and $|\partial^{\alpha}_{\xi}\partial^{\beta}_{x} \phase(x,\xi)|\lesssim\abs{\xi}^{1-\abs{\alpha}},$ for all $(x,\xi) \in \Omega$.
\end{enumerate}
Let us consider the parameter dependant Fourier multiplier and the FIO
\[
\rho(tD) f(x)= \int \rho(t\xi)\widehat{f}(\xi)e^{ix\cdot\xi}\ddd \xi\quad \mbox{and}\quad T_{a_t}^{\phase}(f)(x) = \int a_t(x,\xi)\widehat{f}(\xi)e^{i\phase(x,\xi)} \ddd\xi.
\]
Let $\sigma_t$ be the amplitude of the composition operator $\rho(tD) T_{a_t}^{\phase}={T_{\sigma_t}^\phase}$ given by
\[
	\sigma_{t}(x,\xi):= \iint a_t(y,\xi)\rho(t\eta) e^{i(x-y)\cdot\eta + i\phase(y,\xi)-i\phase(x,\xi)} {\mathrm{\ddd}} \eta \dd y.
\]
Then, for each $M\geq 1$,  we can write $\sigma_t$ as
\begin{equation}  \label{main left composition estim}
\sigma_{t}(x,\xi) = \rho(t\nabla_{x}\phase(x,\xi))a_t(x,\xi) + \sum_{0<|\alpha| < M}\frac{t^{|\alpha|}}{\alpha!}\, \sigma_{\alpha}(t,x,\xi)+t^{M\varepsilon } r(t,x,\xi) \quad %
\end{equation}
for $t\in (0,1)$. Moreover, for all multi-indices $\beta,\gamma$ one has
\begin{equation}\label{eq:statement_1}
	 \abs{\d^\gamma_\xi \d_x^\beta \sigma_\alpha(t,x,\xi)}\lesssim  t^{\abs{\alpha}(\varepsilon -1)} \bra{\xi}^{m-\abs{\alpha}\brkt{\frac{1}{2}-\varepsilon }-\abs{\gamma}} \quad \text{for\quad  $0<\abs{\alpha}<M$},
\end{equation}
and
\begin{equation}\label{eq:statement_2}
	\sup_{t\in (0,1)}|\partial^{\gamma}_{\xi} \partial^{\beta}_{x} r(t,x,\xi)|  \lesssim \bra{\xi}^{m-M\brkt{\frac{1}{2}-\varepsilon }-\abs{\gamma}}.
\end{equation}
\end{thm}
\begin{rem}
It is worth mentioning that although in the applications of the above theorem, the operator $\rho(tD)$ will be a Littlewood-Paley operator $($i.e. $\rho(\xi)$ will be either ball-or annulus-supported$)$, Theorem $\ref{left composition with pseudo}$ doesn't put any conditions on the support of $\rho$. Therefore, the parameter $t$ and the frequency variable $\xi$ are entirely independent of each other.
\end{rem}

\subsection*{Notation} We shall use the notation
\begin{equation}\label{eq:PsandQs}
	Q_t^{u}:=\widehat{\Psi^u}\brkt{tD} \quad \mbox{and}\quad P_t^{u}:=\widehat{\Theta^u}\brkt{tD},
\end{equation}
where
\[
	\widehat{\Psi^u}(\xi):=\widehat{\psi}(\xi) e^{i \xi\cdot u},\qquad \widehat{\Theta^u}(\xi):=\widehat{\theta}(\xi) e^{i \xi\cdot u},
\]
and $\widehat{\psi},\widehat{\theta}\in C^\infty_0(\R^n)$ are any functions supported in an annulus and in a ball respectively. In the case that $u=0$ we shall omit the superscript in the above notation.

\subsection{Boundedness of $T_{\sigma_1}^{\phase_1,\phase_2}$: Endpoint Banach cases}\label{sec:Banach_boundedness_result}

From now on, we confine ourselves to the study of the boundedness of the operator $T_{\sigma_1}^{\phase_1,\phase_2}$.  Indeed, in the frequency support of $T_{\sigma_1}^{\phase_1,\phase_2}$, one has $\abs{\eta}\lesssim\abs{\xi}$, and since the situation is reversed for $T_{\sigma_2}^{\phase_1,\phase_2}$, the boundedness of  the latter will follow from that of the former and symmetry considerations.

Taking Theorem \ref{left composition with pseudo} (whose proof is postponed until Section \ref{sec:proof_of_asymptotic}) for granted, we shall now proceed with the proof of the main theorem for $T_{\sigma_1}$. To this end let
\begin{equation*}
	 m_1=-(n-1)\abs{\frac{1}{p}-\frac{1}{2}},\qquad \text{and}\qquad m_2=-(n-1)\abs{\frac{1}{q}-\frac{1}{2}}.
\end{equation*}
Theorem \ref{left composition with pseudo} yields
\begin{equation}\label{eq:lack_of_coffe_1}
	\begin{split}
	T^{\phase_1}_{\nu_1^{t,u}}(f)(x) &=\int  \mu(\xi)\abs{\nabla_x\phase_1(x,\xi)}^{m}e^{it\nabla_x\phase_1(x,\xi)\cdot u}\widehat{\psi}(t\nabla_x\phase_1(x,\xi)))\chi(x) \widehat{f}(\xi) e^{i\varphi_1(x,\xi)}\ddd \xi\\
	&=t^{-m_2}Q_t^u T_{\nu_{m_1}}^{\varphi_1}f+ t^{-m_2} E_t^1(f)
	\end{split}
\end{equation}
where the symbol of $Q_t^u$ is given by
\[
	\abs{t\xi}^{m_2} \widehat{\psi}(t\xi)e^{it u\cdot\xi}\in \mathcal{S},
\]
$T_{\nu_{m_1}}^{\varphi_1}$ is an FIO with phase $\varphi_1$ and amplitude
\begin{equation}\label{eq:amplitude}
	\nu_{m_1}(x,\xi)=\mu(\xi)\chi(x)\abs{\nabla_x\phase_1(x,\xi)}^{m_1}\in S^{m_1}_{1,0}.
\end{equation}
The operator $E_t^1$ is an FIO with the same phase and has an amplitude in $S^{m_1-(\frac{1}{2}-\varepsilon )}_{1,0}$ with seminorms bounded by the product of $t^\varepsilon $ for $\varepsilon \in (0,1/2)$ with a polynomial expression in $\abs{u}$.

When $m_2\neq 0$, let $\widehat{\theta}_1$ be a compactly supported smooth bump function that is equal to $1$ on the support of $\widehat{\theta}$. When $m_2=0$ we will take $\widehat{\theta}_1$ to be  identically equal to $1$.

Applying Theorem \ref{left composition with pseudo} we have
\begin{equation}\label{eq:lack_of_coffe_2}
	t^{-m_2}T^{\phase_2}_{\mu_1^{t,v}}(g)(x)=P_t^v T_{\mu_1^t}^{\phase_2} g+E_t^2(g)(x),
\end{equation}
where the symbol of $P_t^v$ is given by $\widehat{\theta}(t\eta)e^{it v\cdot \eta}$, which belongs to  $S^{0}_{1,0}$ (uniformly in $t\in (0,1]$) with seminorms that behave polynomially in $\abs{v}$, the operator $T_{\mu_1^t}^{\phase_2}$ is an FIO with phase function $\phase_1$ and amplitude
\begin{equation}\label{eq:technic_amplitude}
	\mu_1^t(x,\eta)=\mu(\eta)\chi(x)\widehat{\theta_1}(t\nabla_x\phase_2(x,\eta))t^{-m_2}\in S^{m_2}_{1,0},
\end{equation}
uniformly in $t\in (0,1]$, and $E_t^2$ is an FIO with the same phase and an amplitude in $S^{m_2-(\frac{1}{2}-\varepsilon )}_{1,0}$ with seminorms bounded by the product of $t^\varepsilon $ and a polynomial in $\abs{v}$.

Therefore, using \eqref{eq:lack_of_coffe_1} and \eqref{eq:lack_of_coffe_2} in  \eqref{three} we obtain that
\begin{equation*}
	\begin{split}
		&\int_0^1 T^{\phase_1}_{\nu_1^{t,u}}(f)(x)T^{\phase_2}_{\mu_1^{t,v}}(g)(x) m(t,x) \frac{\dd t}{t}=\int_0^1 \brkt{Q_t^u T_{\nu_{m_1}}^{\varphi_1}f(x)} \brkt{P_t^v T_{\mu_1^t}^{\phase_2} g(x)} m(t,x) \frac{\dd t}{t}+\\
		&+\int_0^1 \brkt{Q_t^u T_{\nu_{m_1}}^{\varphi_1}f(x)}\brkt{E_t^2(g)(x)}m(t,x) \frac{\dd t}{t}
		+\int_0^1 E_t^1(f)(x)t^{-m_2} \brkt{T^{\phase_2}_{\mu_1^{t,v}}(g)(x)}m(t,x) \frac{\dd t}{t}\\
		&=I_1+I_2+I_3,
	\end{split}
\end{equation*}
where we have set $m(t,x) = m(t,x,u,v)$ for brevity.

From now on, we shall denote by $C(u)$, $C(v)$  and $C(u,v)$ universal constants which depend polynomially on $\abs{u}$ and $\abs{v}$.

Now we claim that for any $\varepsilon >0$,
\begin{equation}\label{eq:lack_of_cake_1}
	\sup_{t\in (0,1]}\norm{t^{-m_2+\frac{\varepsilon }{2}}T^{\phase_2}_{\mu_1^{t,v}}(g)}_{L^q}\leq C(v) \norm{g}_{L^q}, \quad \mbox{for $1 \leq q \leq \infty$.}
\end{equation}
Indeed
\[
	t^{-m_2+\frac{\varepsilon }{2}}e^{it \xi\cdot v}\widehat{\theta}(t\eta)\in S^{m_2-\varepsilon /2}_{1,0},
\]
uniformly in $t\in (0,1]$ with seminorms depending polynomially on $\abs{v}$. Then, using \cite{H2}*{Prop. 1.1.7} we get that
\[
	t^{-m_2+\frac{\varepsilon }{2}}\mu_1^{t,v}(x,\eta)\in S^{m_2-\varepsilon /2}_{1,0},
\]
which by {\ref{eq:DSS}} in Theorem \ref{linear_FIO} is sufficient to prove the claim. Moreover, it follows from the properties of the amplitude of $E^1_t$ and {\ref{eq:DSS}}  that for any fixed $\varepsilon \in (0,\frac{1}{2})$
\begin{equation}\label{eq:lack_of_cake_2}
	\norm{E_t^1 f}_{L^p}\leq C(u) t^{\varepsilon }\norm{f}_{L^p}.
\end{equation}

Since $1\leq r\leq \infty$, the Minkowski integral inequality, \eqref{eq:boundedness_of_m}, \eqref{eq:lack_of_cake_1} and \eqref{eq:lack_of_cake_2} yield
\[
	\begin{split}
	\norm{I_3}_{L^r}&\lesssim \int_0^1 \norm{E_t^1(f)}_{L^p}t^{\frac{-\varepsilon }{2}}\norm{t^{-m_2+\frac{\varepsilon }{2}} T^{\phase_2}_{\mu_1^{t,v}}(g)}_{L^q} \frac{\dd t}{t}\lesssim C(u,v)
	\brkt{\int_0^1 t^{\frac{\varepsilon }{2}-1}\dd t} \norm{f}_{L^p}\norm{g}_{L^q}.
	\end{split}
\]

To deal with $I_2$ assume first that $1<p<\infty$.
In this case, the Minkowski integral inequality and {\ref{eq:SSS}} in Theorem \ref{linear_FIO} yield
\begin{equation}\label{eq:lack_of_cake_3_1}
 	\norm{Q_t^u T_{\nu_{m_1}}^{\varphi_1}f}_{L^p}\lesssim \norm{f}_{L^p},
\end{equation}
uniformly in $t$. Therefore, combining this with the analogue of \eqref{eq:lack_of_cake_2} for $E^2_t$, the Minkowski integral inequality and H\"older's inequality imply
 \[
 	\norm{I_2}_{L^r}\lesssim \int_0^1 t^{\varepsilon-1} \dd t \norm{f}_{L^p}\norm{g}_{L^q}\lesssim \norm{f}_{L^p}\norm{g}_{L^q}.
 \]
With minor modifications to the above argument, one proves that
$\norm{I_2}_{L^1}\lesssim \norm{f}_{H^1}\norm{g}_{L^\infty}.$
For the case $p=\infty$, we observe that the result would follow if one has that
\begin{equation}\label{eq:lack_of_cake_3}
	\norm{Q_t^u T_{\nu_{m_1}}^{\varphi_1}f}_{L^\infty}\leq C(u) \norm{f}_{L^\infty},
\end{equation}
uniformly in $t$.
But the latter is a consequence of {\ref{eq:SSS_infty}} in Theorem  \ref{linear_FIO} and the following lemma, whose proof can be found in \cite{S}*{p. 161}.

\begin{lem}\label{lem:stein} Given $f\in \mathrm{BMO}$ there exists a constant $C$ such that for any $t>0$
\begin{equation}\label{eq:Q1}
	\norm{Q_t f}_{L^\infty}\leq C \norm{f}_{\mathrm{BMO}},
\end{equation}
and by duality, for any $f\in L^1$,
\begin{equation}\label{eq:Q2}	
	\norm{Q_t f}_{H^1}\leq C \norm{f}_{L^1}.
\end{equation}

\end{lem}

Finally, we need to deal with the boundedness of the term $I_1$. To this end we need to consider several cases separately.

\subsection*{Case I: $m=0$}

In particular, $m_2=0$, so in this case, we have taken $\widehat{\theta}_1$ to be the constant function equal to $1$ in \eqref{eq:technic_amplitude}. Thus we have
\begin{equation}\label{eq:eq:I1}
	I_1=\int_0^1 Q_t^u \brkt{T^{\varphi_1}_{\nu_{m_1}} f} P_t^v \brkt{T^{\varphi_2}_{\mu_1}g} \, m(t,x) \frac{\dd t}{t},
\end{equation}
where $\mu_1(x,\eta)=\mu(\eta)\chi(x)\in S^0_{1,0}$. That is, $I_1$ is the composition of a bilinear paraproduct with two FIOs.

In this connection, for the sake of completeness we would like to state and prove a lemma, which although somewhat implicit in the literature, is useful for what follows.
\begin{lem}\label{lem:Paraproduct} Consider the bilinear paraproduct
\[
	\Pi(f,g)(x)=\int_0^1Q_t^u \brkt{f} P_t^v \brkt{g}\, m(t,x,u,v) \frac{\dd t}{t},
\]
with $P_t^u$ and $Q_t^v$ as in \eqref{eq:PsandQs} and $m(t,x,u,v)$ smooth satisfying \eqref{eq:boundedness_of_m}. This operator is a bilinear pseudodifferential operator with symbol in $S^{0}_{1,0}(n,2)$ and seminorms depending polynomially on $\abs{u}$ and $\abs{v}$.
\end{lem}
\begin{proof} Observe that one can easily write
\[
	\Pi(f,g)(x)=\iint \lambda(x,\xi,\eta) \widehat{f}(\xi)\widehat{g}(\eta)e^{ix\cdot(\xi+\eta)}\ddd \xi\ddd \eta,
\]
where
\[
	\lambda(x,\xi,\eta)=\int_0^1  \widehat{\psi}(t\xi)e^{it\xi \cdot v}\widehat{\theta}(t\eta)e^{it\eta \cdot u} m(t,x,u,v)\frac{\dd t}{t}.
\]
Then since $t\thickapprox \abs{\xi}^{-1}$, $\abs{\eta}\lesssim \abs{\xi}$ and $1\lesssim \abs{\xi}$, we have
\[
	\abs{\lambda(x,\xi,\eta)} \lesssim \int_{t\sim \abs{\xi}^{-1}} \frac{\dd t}{t}\thickapprox 1,
\]
and for $\abs{\alpha}+\abs{\beta}\geq 1$,
\[
	\abs{\d^\alpha_\xi\d^\beta_\eta\d^\gamma_x \lambda(x,\xi,\eta)} \lesssim C(u,v)\abs{\xi}^{-\abs{\alpha}-\abs{\beta}}\lesssim  C(u,v) \brkt{1+\abs{\xi}+\abs{\eta}}^{-\abs{\alpha}-\abs{\beta}}.
\]
\end{proof}
\begin{rem}
As is well-known, bilinear pseudodifferential operators with symbols in $S^{0}_{1,0}(n,2)$ are examples of bilinear Calder\'on-Zygmund operators, see e.g. \cite{GT1}. Therefore, $\Pi(f,g)$ satisfies all the boundedness properties that are listed in \cite{GK}*{Theorem 1.1}. These include boundedness on products of Banach spaces such as $L^p$, $H^1$ and $\mathrm{BMO}$.
\end{rem}
Thus we can write \eqref{eq:eq:I1} as
\begin{equation}\label{eq:paraproduct}
	I_1=\Pi(T^{\varphi_1}_{\nu_1}f,T^{\varphi_2}_{\mu_1}g).
\end{equation}

Since $m=0$, then either $n=1$ or $p=q=2$. If $p=q=2$, the $L^2\times L^2 \to L^1$ boundedness of $ T_{\sigma_1}^{\phase_1,\phase_2}(\mu(D)f,\mu(D)g)$ follows  from the previous lemma and the $L^2\to L^2$ boundedness given by \ref{eq:SSS} in Theorem \ref{linear_FIO}. Observe that in dimension $n=1$, the matters  reduce to the study of \eqref{eq:paraproduct} with
\[
	\mu_1(x,\cdot)=\nu_1(x,\cdot)=\chi(x)\mu(\cdot).
\]
This means that the boundedness of the bilinear FIO in the case $L^{p}\times L^{q}\to L^r$ for $1<p,q<\infty$ follows from Lemma \ref{lem:Paraproduct} and  \ref{eq:SSS} in Theorem \ref{linear_FIO}.
It will be shown in Section \ref{sect:1d} that the endpoint results for the one dimensional bilinear FIOs are false.  So, from now on we shall assume that $n\geq 2$.

\begin{rem}\label{rem:Remark} Observe that by a similar argument to above, if $m=-(n-1)\abs{\frac{1}{p}-\frac{1}{2}}$ for $1<p<\infty$ and $m_2=0$, we can obtain
the $L^p\times L^2\to L^r$  boundedness of  the operator $ T_{\sigma_1}^{\phase_1,\phase_2}(\mu(D)f,\mu(D)g)$. Moreover,  since paraproducts map $\mathrm{BMO}\times L^2\to L^2$ $($see e.g. \cite{CM4}*{Proposition 2}$)$, by \ref{eq:SSS_infty} in Theorem $\ref{linear_FIO}$ we would have the $L^\infty \times L^2\to L^2$ boundedness of $ T_{\sigma_1}^{\phase_1,\phase_2}(\mu(D)f,\mu(D)g)$ for $m=-\frac{n-1}{2}$.
\end{rem}

\subsection*{Case II: $m<0$}

By the previous remark, from now on, we can assume that $m_2<0$.

In order to continue the proof of the theorem we introduce another Littlewood-Paley partition of unity of the type
\begin{equation}\label{eq:L-P}
	1=\int_0^\infty \widehat{\psi}(s\nabla_{x}\phase_2(x,\eta))^2\, \frac{\dd s}{s}, \qquad \forall \eta\neq 0, x\in \R^n.
\end{equation}
In this way
\[
	T_{\mu_1^t} g(x)= \int_0^\infty S_{t,s} g(x)\,  \frac{\dd s}{s}
\]
where
\[
	S_{t,s}g(x)=\int \chi(x)\mu(\eta)\widehat{\theta_1}(t\nabla_x\phase_2(x,\eta))t^{-m_2}\widehat{\psi}(s\nabla_x\phase_2(x,\eta))^2 \widehat{g}(\eta) e^{i\phase_2(x,\eta)}\ddd \eta.
\]
The support properties of $\widehat{\theta_1}$ and $\widehat{\psi}$ imply that there exist positive constants $\kappa_1, \kappa_2$ such that, for any $t\in (0,1]$,
\[
	\kappa_2 t \les s\leq \kappa_1.
\]
Without loss of generality we can take $\kappa_1\geq 1$.

Hence, if we define $\kappa=\kappa_2/\kappa_1$,
\[
	\begin{split}
	T_{\mu_1^t} g(x)&=\int_{\kappa_2 t}^{\kappa_1} S_{t,s} g(x)\, \frac{\dd s}{s}=\int_{\kappa}^{1/t} S_{t,ts\kappa_1} g(x)\,\frac{\dd s}{s}.
	\end{split}
\]
Observe now that
\[
	\begin{split}
	 S_{t,t s\kappa_1} g(x) &=\int	\chi(x)\mu(\eta)\widehat{\theta_1}(t\nabla_x\phase_2(x,\eta))t^{-m_2}\widehat{\psi}(ts\kappa_1\nabla_x\phase_2(x,\eta))^2\,\, \widehat{g}(\eta) e^{i\phase_2(x,\eta)}\ddd \eta\\
	 &=s^{m_2}\int	\chi(x)\mu(\eta)\widehat{\theta_1}(t\nabla_x\phase_2(x,\eta))\widehat{\Psi}(ts\nabla_x\phase_2(x,\eta)) \abs{\nabla_x\phase_2(x,\eta)}^{m_2}\,\, \widehat{g}(\eta) e^{i\phase_2(x,\eta)}\ddd \eta\\
	 \end{split}
\]
where
\begin{equation}\label{eq:new_psi}
	\widehat{\Psi}(\eta)=\widehat{\psi}(\kappa_1\eta)^2 \abs{\eta}^{-m_2}.
\end{equation}
Now if $P_{1,t}$ denotes the multiplier operator with symbol $\widehat{\theta_1}(t\eta)$, then Theorem \ref{left composition with pseudo} yields
\[
	\begin{split}
	 S_{t,ts\kappa_1} g(x) &=s^{m_2} P_{1,t} \left[\int	\chi(\cdot)\mu(\eta)\widehat{\Psi}(ts\nabla_x\phase_2(\cdot,\eta)) \abs{\nabla_x\phase_2(\cdot,\eta)}^{m_2}\,\, \widehat{g}(\eta) e^{i\phase_2(\cdot,\eta)}\ddd \eta\right](x)\\
	 &+ s^{m_2}\int	R_t^1(x,\eta)\,\widehat{g}(\eta)\, e^{i\phase_2(x,\eta)}\ddd \eta,
	 \end{split}
\]
where $R_t^1(x,\eta)\in S^{m_2 -(\frac{1}{2}-\varepsilon)}_{1,0}$ with seminorms bounded by $t^\varepsilon$ with $\varepsilon\in (0,1/2)$ and compact spatial support.

Observe that, since $ts\leq 1$, $\widehat{\Psi}(ts\eta)\in S^{0}_{1,0}$ unifomly in $ts$, Theorem \ref{left composition with pseudo} yields
\[
	\begin{split}
	&\int	\chi(x)\mu(\eta)\widehat{\Psi}(ts\nabla_x\phase_2(x,\eta)) \abs{\nabla_x\phase_2(x,\eta)}^{m_2}\,\, \widehat{g}(\eta) e^{i\phase_2(x,\eta)}\ddd \eta=\\
	&\qquad =Q_{ts} \left[\int	\chi(\cdot)\mu(\eta)\abs{\nabla_x\phase_2(\cdot,\eta)}^{m_2}\,\, \widehat{g}(\eta) e^{i\phase_2(\cdot,\eta)}\ddd \eta\right]\\
	&\qquad+ \int R_{ts}^2(x,\eta)\,\widehat{g}(\eta)\, e^{i\phase_2(x,\eta)}\ddd \eta,
	\end{split}
\]
where $R_{ts}^2(x,\eta)\in S^{m_2 -(\frac{1}{2}-\varepsilon)}_{1,0}$ with seminorms bounded by $(ts)^\varepsilon$ with $\varepsilon\in (0,1/2)$ and compact spatial support.
Hence
\[
	\begin{split}
	 S_{t,ts\kappa_1} g(x) &=s^{m_2} P_{1,t} Q_{ts} \left(T_{\gamma}^{\phase_2} g\right)(x)+s^{m_2}P_{1,t}\left[\int R_{ts}^2(\cdot,\eta) \, \widehat{g}(\eta)\, e^{i\phase_2(\cdot,\eta)}\ddd \eta\right]\\
	 &+s^{m_2}\int R_t^1(x,\eta)\, \widehat{g}(\eta) e^{i\phase_2(x,\eta)}\ddd \eta\\
	 &=S_{s,t}^1g(x)+S_{s,t}^2g(x)+S_{s,t}^3g(x),
	 \end{split}
\]
where $\gamma(x,\eta)=\chi(x)\mu(\eta)\abs{\nabla_x\phase_2(x,\eta)}^{m_2}\in S^{m_2}_{1,0}$.

Then
\[
\begin{split}
I_1&=\int_0^1 \brkt{Q_t^u T_{\nu_{m_1}}^{\varphi_1}f(x)} \brkt{P_t^v T_{\mu_1^t}^{\phase_2} g(x)} m(t,x) \frac{\dd t}{t}\\
&=\sum_{j=1,2,3}\int_{\kappa}^{\infty} \int_0^1 \chi_{(0,1)}(st)\brkt{Q_t^u T_{\nu_{m_1}}^{\varphi_1}f(x)} \brkt{P_t^v S^{j}_{s,t} g(x)} m(t,x) \frac{\dd t}{t}\frac{\dd s}{s}
\end{split}
\]
Observe that
\[
	\norm{P_t S^{3}_{s,t}g}_{L^q}\lesssim s^{m_2} t^{\varepsilon} \norm{g}_{L^q},
\]
\[
	\norm{P_t S^{2}_{s,t}g}_{L^q}\lesssim s^{m_2+\varepsilon} t^{\varepsilon} \norm{g}_{L^q},
\]
and
\[
	\norm{Q_t T_{\nu_{m_1}} f}_{L^p}\lesssim \norm{f}_{L^p}.
\]
for $1<p\leq \infty$ and $\norm{Q_t T_{\nu_{m_1}} f}_{L^1}\lesssim \norm{f}_{H^1}$.
Then, for $1<p\leq \infty$
\[
\begin{split}
	&\norm{\int_{\kappa}^{\infty} \int_0^1 \chi_{(0,1)}(st)\brkt{Q_t^u T_{\nu_{m_1}}^{\varphi_1}f(x)} \brkt{P_t^v S^{3}_{s,t} g(x)} m(t,x) \frac{\dd t}{t}\frac{\dd s}{s}}_{L^r}\\
	&\quad \lesssim \int_{\kappa}^\infty \int_0^{1/s} s^{m_2} t^{\varepsilon}\frac{\dd t}{t}\frac{\dd s}{s} \norm{f}_{L^p}\norm{g}_{L^q}\lesssim \norm{f}_{L^p}\norm{g}_{L^q}.
\end{split}
\]
For $p=1$, one has similar bound with $H^1$ instead of $L^1$.
Moreover if $n\geq 2$,
\[
\begin{split}
	&\norm{\int_{\kappa}^{\infty} \int_0^1 \chi_{(0,1)}(st)\brkt{Q_t^u T_{\nu_{m_1}}^{\varphi_1}f(x)} \brkt{P_t^v S^{2}_{s,t} g(x)} m(t,x) \frac{\dd t}{t}\frac{\dd s}{s}}_{L^r}\\
	&\quad \lesssim \int_{\kappa}^\infty \int_0^{1/s} s^{m_2+\varepsilon} t^{\varepsilon}\frac{\dd t}{t}\frac{\dd s}{s} \norm{f}_{L^p}\norm{g}_{L^q}\lesssim \norm{f}_{L^p}\norm{g}_{L^q}.
\end{split}
\]
Hence it only remains to prove the desired bound for
\begin{equation}\label{eq:final_reduction}
	\int_\kappa^\infty  s^{m_2} \int_0^1\chi_{(0,1)}(st)\brkt{Q_t^u T_{\nu_{m_1}}^{\varphi_1}f(x)} \brkt{P_t^v P_{1,t} Q_{ts} T_{\gamma}^{\phase_2}g(x)}m(t,x)\frac{\dd t}{t}\frac{\dd s}{s}.
\end{equation}
Let us define
\begin{equation}\label{eq:Miracle}
	R_t (G)(x):=\int_\kappa^{1/t} s^{m_2}Q_{ts}(G)(x)\frac{\dd s}{s}=\int K_t(x-y) G(y)\dd y,
\end{equation}
where
\begin{equation}\label{eq:Miracolous_kernel}
	K_t(z)=\int_\kappa^{1/t} s^{m_2} \Psi\brkt{\frac{z}{ts}} (ts)^{-n}\frac{\dd s}{s}.
\end{equation}
\begin{lem} \label{lem:miracle} The function $K_t$ defined above is such that
\begin{equation}\label{eq:cancellation}
	\int K_t(z) \dd z=0
\end{equation}
and for each $0<\delta<-m_2$ we have the estimates
\begin{equation}\label{eq:decay}
	\abs{K_t(x-y)}\lesssim {t^{-n}}{\brkt{1+\frac{\abs{x-y}}{t}}^{-n-\delta}}
\end{equation}
and
\begin{equation}\label{eq:mean_value}
	\abs{K_t(x-y)-K_t(x-y')}\lesssim {t^{-n-1}}\abs{y-y'}
\end{equation}
for all $x,y,y'\in \R^n$ and $0<t<1$
\end{lem}
\begin{proof} The first assertion follow from the fact that $\int \Psi \dd z=0$. For the second one, since $\Psi\in \S$ and $0<\delta<-m_2$ we have that
\[
	\abs{K_t(z)}\leq \brkt{\frac{\abs{z}}{t}}^{-n-\delta} t^{-n} \int_\kappa^{1/t} s^{m_2+\delta}\frac{\dd s }{s}\lesssim  \brkt{\frac{\abs{z}}{t}}^{-n-\delta} t^{-n}.
\]
Moreover
\[
	\abs{K_t(z)}\leq t^{-n}\int_\kappa^{1/t} s^{m_2-n-1}\dd s\lesssim  t^{-n}.
\]
These estimates yield \eqref{eq:decay}. Estimate \eqref{eq:mean_value} follows from the mean value theorem and  the fact that
\[
	\abs{\nabla_z K_t(z)}=\abs{ \int_\kappa^{1/t} t^{-n-1}s^{-n-1+m_2} \brkt{\nabla\Psi} \brkt{\frac{x}{ts}}\frac{\dd s}{s}}\lesssim t^{-n-1}.
\]

\end{proof}
\begin{prop} The operator $R_t$ defined above satisfies
\begin{equation}\label{eq:Lq_bound}
	\sup_{t>0}\norm{R_t f}_{L^q}\lesssim \norm{f}_{L^q},\qquad 1\leq q\leq \infty,
\end{equation}
and
\begin{equation}\label{eq:BMO_bound}	\sup_{t>0}\norm{R_t f}_{L^\infty}\lesssim \norm{f}_{\mathrm{BMO}}.
\end{equation}
\end{prop}
\begin{proof} The first assertion is a consequence of the Minkowski integral inequality and \eqref{eq:decay}.
For the second assertion,  using \eqref{eq:cancellation} and \eqref{eq:decay}
\[
	\abs{R_t f(x)}\lesssim \int {t^{-n}}{\brkt{1+\frac{\abs{x-y}}{t}}^{-n-\delta}} \abs{f(y)-{\rm Avg}_{B(x,t)}f}\dd y.
\]
By a change of variables, if $F_{t,x}(y)=f(ty+x)$ then,
\[
	\abs{R_t f(x)}\lesssim \int {\brkt{1+{\abs{y}}}^{-n-\delta}} \abs{F_{t,x}(y)-{\rm Avg}_{B(0,1)}F_{t,x}}\dd y\lesssim \norm{f}_{\mathrm{BMO}},
\]
where the last inequality follows from \cite{G}*{Proposition 7.1.5} and the dilation and translation invariance of the BMO norm.
\end{proof}
Now, it suffices to see that
\begin{equation}\label{eq:X}
	 \abs{\int \int_0^{1} h(x) \brkt{Q_t^u T_{\nu_{m_1}}^{\varphi_1}f(x)} \brkt{\tilde{P}_{t}^v R_t T_{\gamma}^{\phase_2}g(x)}m(t,x)\frac{\dd t}{t}\dd x}\lesssim \norm{f}_{L^p}\norm{g}_{L^q} \norm{h}_{L^{r'}},
\end{equation}
where $h\in L^{r'}$ has support contained in the $x$ support of $m$,  $1/{r'}=1-1/p-1/q$ and $\tilde{P}_{t}^v=P_t^v P_{1,t}$.
Now we decompose $\tilde{P}_{t}^v=P_t^{2,v}+Q_t^{2,v}$ where $P_t^{2,v}$  has symbol a smooth function of the type $\widehat{\theta_2}(t\xi)e^{i t \xi\cdot v}$ with $\supp \widehat{\theta_2}$ contained in a sufficiently small ball, and the symbol of $Q_t^{2,v}$  is $(\widehat{\theta}(t\xi)-\widehat{\theta_2}(t\xi))e^{i t \xi\cdot v}$. We take the support of $\widehat{\theta_2}$ small enough to assure that, for a suitable smooth radial function $\widehat{\psi_2}(\xi)$ supported in an annulus, we can write
\[
	\brkt{Q_t^u T_{\nu_{m_1}}^{\varphi_1}f(x)} \brkt{P_t^{2,v} R_t T_\gamma^{\phase_2}g (x)}=\tilde{Q}_{t}\brkt{\brkt{Q_t^u T_{\nu_{m_1}}^{\varphi_1}f} \brkt{P_t^{2,v} R_t T_\gamma^{\phase_2}g}}(x),
\]
where $\tilde{Q}_{t}$ is the multiplier operator with symbol $\widehat{\psi_2}(t\xi)$. Furthermore, we can find a smooth radial function $\widehat{\theta_3}$ supported in a ball such that
\[
	\brkt{Q_t^u T_{\nu_{m_1}}^{\varphi_1}f(x)} \brkt{Q_t^{2,v}R_t  T_\gamma^{\phase_2}g (x)}=\tilde{P}_{t}\brkt{\brkt{Q_t^u T_{\nu_{m_1}}^{\varphi_1}f} \brkt{Q_t^{2,v} R_t T_\gamma^{\phase_2}g}}(x),
\]
where $\tilde{P}_{t}$ has symbol $\widehat{\theta_3}(t\xi)$.
Then, the left hand side of \eqref{eq:X} is equal to
\[
	\begin{split}
		I+II:=&\int \int_0^{1} \brkt{Q_t^u T_{\nu_{m_1}}^{\varphi_1}f(x)} \brkt{P_t^{2,v} R_t T_\gamma^{\phase_2}g (x)} \tilde{Q}_{t}((h m(t,\cdot))(x) \frac{\dd t}{t}\dd x\\
	&\quad+\int \int_0^{1} \brkt{Q_t^u T_{\nu_{m_1}}^{\varphi_1}f(x)} \brkt{Q_t^{2,v} R_t T_\gamma^{\phase_2}g (x)}\tilde{P}_{t}((h m(t,\cdot))(x)\frac{\dd t}{t}\dd x
	\end{split}
\]
Define $M_{m}$ by $M_m(f)(x) = m(t,x)f(x)$. Observe that
\[
	\begin{split}
	I =I_{4}+I_5 &=\int \int_0^{1} \brkt{Q_t^u T_{\nu_{m_1}}^{\varphi_1}f(x)} \brkt{P_t^{2,v} R_t T_\gamma^{\phase_2}g (x)} \tilde{Q}_{t}h(x) m(t,x) \frac{\dd t}{t}\dd x\\
	&+\int\int_0^{1} \brkt{Q_t^u T_{\nu_{m_1}}^{\varphi_1}f(x)} \brkt{P_t^{2,v} R_t T_\gamma^{\phase_2}g (x)} \left[\tilde{Q}_{t}, M_m\right]h(x)\frac{\dd t}{t}\dd x
	\end{split}
\]

Then  we could apply  \ref{eq:SSS} in Theorem \ref{linear_FIO}, \eqref{eq:Lq_bound} and the Minkowski integral inequality to obtain
\begin{equation}\label{eq:lack_of_cake_6}
	\norm{P_t^{2,v} R_t T_\gamma^{\phase_2}g}_{L^q}\leq C(v) \norm{g}_{L^q},
\end{equation}
whenever $1<q<\infty$. Similarly
\begin{equation}\label{eq:lack_of_cake_7}
	\norm{P_t^{2,v}R_t  T_\gamma^{\phase_2}g }_{L^1}\leq C(v) \norm{g}_{H^1},
\end{equation}
For $q=\infty$, \eqref{eq:BMO_bound} and  \ref{eq:SSS_infty} in Theorem  \ref{linear_FIO} yield that for any $t>0$,
\begin{equation}\label{eq:lack_of_cake_8}
	\norm{P_t^{2,v}R_t T_\gamma^{\phase_2}g}_{L^\infty}\lesssim \norm{R_t T_\gamma^{\phase_2}g}_{L^\infty}\lesssim\norm{T_\gamma^{\phase_2}g}_{\mathrm{BMO}}
	\lesssim \norm{g}_{L^\infty}.
\end{equation}

For the analysis of $I_5$ we observe that, by the mean-value theorem, we have that
\[
|m(t,y,u,v) - m(t,x,u,v)| \lesssim |y-x|
\]
with an implicit constant that by \eqref{eq:boundedness_of_m}, is independent of $t$, $u$ and $v$. Using this, we compute
\begin{align*}
 |[\tilde{Q}_{t},M_m](f)(x)| & \lesssim \left|\int t^{-n}{\psi}\brkt{\frac{x-y}{t}}\big(m(t,y,u,v) - m(t,x,u,v)\big) f(y) \dd y\right| \\
& \lesssim t\int \abs{t^{-n}\psi\brkt{\frac{x-y}{t}}}\frac{|x-y|}{t} |f(y)| \dd y.
\end{align*}
Therefore
\begin{equation}\label{eq:lack_of_cake_9}
	\norm{\left[\tilde{Q}_{t}, M_m\right]h}_{L^{r'}}\lesssim t \norm{h}_{L^{r'}}.	
\end{equation}
Similar estimate holds for the commutator $\left[\tilde{P}_{t}, M_m\right]$.

Then,  \eqref{eq:lack_of_cake_3_1}, \eqref{eq:lack_of_cake_3},  \eqref{eq:lack_of_cake_6} and \eqref{eq:lack_of_cake_9},  imply   that $\abs{I_5} \leq C(u,v,\varepsilon ) \norm{f}_{L^{p}} \norm{g}_{L^{q}} \norm{h}_{L^{r'}}$ for $1<q<\infty$. For $q=1$, \eqref{eq:lack_of_cake_3} \eqref{eq:lack_of_cake_7} and \eqref{eq:lack_of_cake_9} yield $\abs{I_5} \leq C(u,v,\varepsilon ) \norm{f}_{L^{\infty}} \norm{g}_{H^1} \norm{h}_{L^{\infty}}$. For $q=\infty$,  \eqref{eq:lack_of_cake_3_1}, \eqref{eq:lack_of_cake_3}, \eqref{eq:lack_of_cake_8} and \eqref{eq:lack_of_cake_9} yield
\[
	\abs{I_5}\lesssim \norm{f}_{L^p}\norm{g}_{L^\infty} \norm{h}_{L^{r'}},
\]
for $p>1$, and
$\abs{I_5}\lesssim \norm{f}_{H^1}\norm{g}_{L^\infty} \norm{h}_{L^{\infty}}$, for $p=1$.

Arguing in a similar way,
\[
	\begin{split}
	II =I_{6}+I_7 &=\int \int_0^{1} \brkt{Q_t^u T_{\nu_{m_1}}^{\varphi_1}f(x)} \brkt{Q_t^{2,v} R_t  T_\gamma^{\phase_2}g (x)} \tilde{P}_{t}h(x) m(t,x) \frac{\dd t}{t}\dd x\\
	&+\int\int_0^{1} \brkt{Q_t^u T_{\nu_{m_1}}^{\varphi_1}f(x)} \brkt{Q_t^{2,v} R_t T_\gamma^{\phase_2}g (x)} \left[\tilde{P}_{t}, M_m\right]h(x) \frac{\dd t}{t}\dd x
	\end{split}
\]
with
$\abs{I_7}\leq C(u,v,\varepsilon ) \norm{f}_{L^{p}} \norm{g}_{L^{q}} \norm{h}_{L^{r'}}$ and the obvious modifications for the extremal cases.
So we have reduced the problem to prove the desired estimates for
\begin{equation}\label{eq:I_4}
\abs{I_4}=\abs{\int \int_0^{1} \brkt{Q_t^u T_{\nu_{m_1}}^{\varphi_1}f(x)} \brkt{P_t^{2,v} R_t T_\gamma^{\phase_2}g (x)} \tilde{Q}_{t}h(x) m(t,x) \frac{\dd t}{t}\dd x},
\end{equation}
\begin{equation}\label{eq:I_6}
	\abs{I_6}=\abs{\int\int_0^{1} \brkt{Q_t^u T_{\nu_{m_1}}^{\varphi_1}f(x)} \brkt{Q_t^{2,v} R_t T_\gamma^{\phase_2}g (x)} \tilde{P}_{t}h(x) m(t,x) \frac{\dd t}{t}\dd x}.
\end{equation}

At this point we proceed by studying four subcases.

\subsubsection*{{\bf Case II.a:} $m=-(n-1)/2$} Using Remark \ref{rem:Remark}, it is enough to consider the $L^2\times L^\infty\to L^2$ case. Therefore let us assume that $f,h\in L^2$, $g\in L^\infty$. Observe that we can assume, without loss of generality, that $\norm{g}_{L^\infty}=1$. Then, by the Cauchy-Schwarz inequality and \eqref{eq:boundedness_of_m} we have
\[
	\abs{I_4}\lesssim \brkt{\int \int_0^1  \abs{Q_t^u T_{\nu_{m_1}}^{\varphi_1}f(x)}^2 \frac{\dd t}{t} \dd x }^{\frac{1}{2}}
	 \brkt{\int \int_0^1  \abs{P_t^{2,v} R_t T_\gamma^{\phase_2}g (x)}^2 \abs{\tilde{Q}_{t} h(x)}^2 \frac{\dd t}{t} \dd x }^{\frac{1}{2}}.
\]
A standard quadratic estimate together with the $L^2$ boundedness of $T_{\nu_{m_1}}^{\varphi_1}$ given by  \ref{eq:SSS} in Theorem \ref{linear_FIO} yield that the first factor on the right hand side is bounded by $C(u)\norm{f}_{L^2}$. For the second factor, a quadratic estimate and \eqref{eq:lack_of_cake_8} yield that it is bounded by $\norm{g}_{L^\infty}\norm{h}_{L^2}$. Therefore $\abs{I_4}\lesssim C(u)  \norm{f}_{L^2}\norm{g}_{L^\infty}\norm{h}_{L^2}$.

In order to deal with $I_6$, let us recall that a Borel measure $\dd\mu$ on $\R^n_+$ is called a Carleson measure if
\[
\norm{\mathrm{d}\mu}_{\mathcal{C}}=\sup_{\epsilon>0}\sup_{y\in \R^n} \frac{1}{\epsilon^n}\int_0^{\epsilon}\int_{\abs{x-y}<\epsilon} \mathrm{d}\abs{\mu}(x,t)<+\infty.
\]
The quantity $\norm{\mathrm{d}\mu}_{\mathcal{C}}$ is called the Carleson norm of $\dd\mu$.
We shall also need the following lemma, whose proof is a modification of an argument which goes back to A. Uchiyama.
\begin{lem}\label{lem:composition} For any Carleson measure $\mathrm{d}\mu$ and $K_t$ satisfying \eqref{eq:decay} for some $\delta>0$,
\[
	\dd \tilde{\mu}(x,t):=\brkt{\int \abs{K_t(x-y)}\dd \mu(y,t)}\dd x,
\]
defines a Carleson measure and
$
	\norm{\dd \tilde{\mu}}_{\mathcal{C}}\lesssim \norm{\mathrm{d}\mu}_{\mathcal{C}}.
$
\end{lem}
\begin{proof} Let $\epsilon>0$ and $y\in \R^n$. First, we cover $\R^n$ with balls of radius $\epsilon\sqrt{n}/{2}$ and centre $y+\epsilon k$, with $k\in \Z^n$. Then changing the order of integration followed by a change of variables yield
\begin{equation}\label{eq:Uchiyama}
	\begin{split}
	I(\epsilon,y)&:=\int_0^\epsilon \int_{\abs{x-y}<\epsilon} \brkt{\int \abs{K_t(x-z)}\dd \mu(z,t)} \dd x\\
	&\leq \sum_{k\in \Z^n} \int_0^\epsilon \int_{\abs{x-y}<\epsilon} \brkt{\int_{\abs{z-y-k\epsilon}<\epsilon\frac{\sqrt{n}}{2}} \abs{K_t(x-z)}\dd \mu(z,t)} \dd x\\
	&= \sum_{k\in \Z^n} \int_0^\epsilon \int_{\abs{z-y-k\epsilon}<\epsilon\frac{\sqrt{n}}{2}}  \brkt{\int_{\abs{z-y+u}<\epsilon}  \abs{K_t(u)}\dd u}\dd \mu(z,t).
	\end{split}
\end{equation}
Using the triangle inequality we have that $\abs{u+k\epsilon}\leq \epsilon \brkt{1+\frac{\sqrt{n}}{2}}$, and therefore we have
$
	\abs{u}\geq \epsilon \brkt{\abs{k}-\brkt{1+\frac{\sqrt{n}}{2}}}.
$

Now if $\abs{k}\geq 2+\sqrt{n}$ then $\abs{u}\geq \epsilon \abs{k}/2$ and for such $k$'s, \eqref{eq:decay} yields that
\[
	 \brkt{\int_{\abs{z-y+u}<\epsilon}  \abs{K_t(u)}\dd u}\lesssim \abs{k}^{-n-\delta} \brkt{\frac{t}{\epsilon}}^{\delta}\leq \abs{k}^{-n-\delta},
\]
for $t\leq \epsilon$. Then, enlarging the domain of integration as necessary (depending on the dimension) we have
\[
	\begin{split}
	 \sum_{\abs{k}\geq 2+\sqrt{n}} \int_0^\epsilon \int_{\abs{z-y-k\epsilon}<\epsilon\frac{\sqrt{n}}{2}} & \brkt{\int_{\abs{z-y+u}<\epsilon}  \abs{K_t(u)}\dd u}\dd \mu(z,t)\\
	 &\lesssim   \epsilon^{n}\sum_{\abs{k}\geq 2+\sqrt{n}} \abs{k}^{-n-\delta} \norm{\dd\mu}_{\mathcal{C}}\lesssim   \epsilon^{n}\norm{\dd\mu}_{\mathcal{C}}.
	 \end{split}
\]
On the other hand, using once again \eqref{eq:decay}
\[
	 \brkt{\int_{\abs{z-y+u}<\epsilon}  \abs{K_t(u)}\dd u}\lesssim \int \brkt{1+\frac{\abs{u}}{t}}^{-n-\delta} t^{-n}\dd u\lesssim 1.
\]
Hence, again enlarging the domain of integration as necessary,
\[
	\begin{split}
	 \sum_{\abs{k}< 2+\sqrt{n}} \int_0^\epsilon \int_{\abs{z-y-k\epsilon}<\epsilon\frac{\sqrt{n}}{2}}  \brkt{\int_{\abs{z-y+u}<\epsilon}  \abs{K_t(u)}\dd u}\dd \mu(z,t)\lesssim  \epsilon^{n}\norm{\dd\mu}_{\mathcal{C}}.
	 \end{split}
\]
Putting the estimates together we obtain the desired result.
\end{proof}

Observing that the commutator $[R_t, Q_{t}^{2,v}]=0$ and returning to $I_6$, we claim that
\begin{equation}\label{eq:carleson}
	\abs{R_t Q_{t}^{2,v}T_{\gamma}^{\phase_2} g(x)}^2\frac{\dd t}{t}\dd x,
\end{equation}
is a Carleson measure with norm bounded by a constant times $\norm{g}_{L^\infty}^2$= 1. Indeed, it is well-known (see e.g. \cite{S}) that, for any $f\in \mathrm{BMO}$
\begin{equation}\label{eq:standard_CM}
	\mathrm{d}\mu(x,t)=\abs{Q_t^{2,v} f(x)}^2\frac{\dd t}{t}\dd x,
\end{equation}
is a Carleson measure and
\[
	\norm{\mathrm{d}\mu}_{\mathcal{C}}\lesssim \norm{f}_{\mathrm{BMO}}^2.
\]
The Cauchy-Schwarz inequality and \eqref{eq:decay} yield
\[
	\begin{split}
	\abs{R_t Q_{t}^{2,v}T_{\gamma}^{\phase_2} g(x)}^2 &\leq \brkt{\int\abs{K_t(y)}\dd y}\brkt{\int \abs{K_t(x-y)} \abs{Q_t^{2,v} T_{\gamma}^{\phase_2} g(y)}^2\dd y}\\
	&\lesssim  \int \abs{K_t(x-y)} \abs{Q_t^{2,v} T_{\gamma}^{\phase_2} g(x)(y)}^2\dd y.
	\end{split}
\]
Therefore,  \ref{eq:SSS_infty} in Theorem \ref{linear_FIO} and Lemma \ref{lem:composition} prove that \eqref{eq:carleson} is a Carleson measure.

Proceeding as in the analysis of $I_4$  we have
\[
	\abs{I_6}\lesssim \brkt{\int \int_0^1  \abs{Q_t^u T_{\nu_{m_1}}^{\varphi_1}f(x)}^2 \frac{\dd t}{t} \dd x }^{\frac{1}{2}}
	 \brkt{\int \int_0^1  \abs{\tilde{P}_{t} h(x)}^2 \abs{Q_t^{2,v} R_t T_\gamma^{\phase_2}g (x)}^2 \frac{\dd t}{t} \dd x }^{\frac{1}{2}}.
\]
The first factor is treated as before using a standard quadratic estimate and for the second factor we use \eqref{eq:carleson}, \cite{FS}*{Theorem 4} and the boundedness of the non-tangential maximal operator to obtain that
\[
	\abs{I_6}\lesssim \norm{f}_{L^2} \norm{g}_{L^\infty} \norm{h}_{L^2}.
\]

\subsubsection*{{\bf Case II.b:}  $m=-(n-1)$ and $L^\infty\times L^\infty \to BMO$ boundedness}

In this case we need to estimate $I_4$ and $I_6$ with $m_1=m_2=-(n-1)/2$, for $f,g\in L^\infty$ and $h\in H^1$. Without loss of generality, we can assume that $\norm{f}_{L^\infty}=\norm{g}_{L^\infty}= 1$. To control $I_4$ we need the following proposition.
\begin{prop}\label{prop:yabuta}
Let $F\in H^1$, $v\in L^\infty_{t,x}$, and $T_{t}$ be the convolution operator given by
\[
	T_{t} (f)(x)=\int f(x-y)\dd \nu_t(y),
\]
with $\{\nu_t\}_t$  be finite measures such that for some $\delta>0$ and for any $t>0$,
\begin{equation}\label{eq:decay_measure}
	 \int \brkt{1+\frac{\abs{x-y}}{t}}^{-n-\delta} \dd \abs{\nu_t}(y)\lesssim \brkt{1+\frac{\abs{x}}{t}}^{-n-\delta}.
\end{equation}
Let $G(t,x)$ be a measurable function on $\R^{n+1}_+$ such that $\dd \mu_G(t,x)=\abs{G(t,x)}^2\frac{\dd t}{t}\dd x$ is a Carleson measure. Then
\[
	\abs{\int\int_0^\infty Q_t T_t F(x)\, G(t,x) v(t,x)\frac{\dd t}{t}\dd x}\leq C \norm{F}_{H^1}\norm{\dd \mu_G}_{\mathcal{C}}^{\frac{1}{2}}\norm{v}_{L^\infty_{t,x}}.
\]
\end{prop}
\begin{proof}
By a density argument (see \cite{S0}*{pp.~231-232}) it is enough to prove the result for $F\in H^1$ with  compact frequency support away from the origin.
Let $\widehat{\psi}$ be the multiplier associated to $Q_t$. Since $\psi$ is supported in an annulus we can write
\[
	\widehat{\psi}(\xi)=-\abs{\xi}e^{-\abs{\xi}} \widehat{u}(\xi),
\]
where $\widehat{u}\in C^\infty_0(\R^n)$ and is  supported in the same annulus as $\widehat{\psi}$. Then
\[
	\widehat{\psi}(t\xi)= -t \abs{\xi}e^{-t\abs{\xi}} \widehat{u}(t\xi)= t \d_t \brkt{\widehat{\mathbb{P}_t}(\xi)}\widehat{u}(t\xi)
\]
where $\mathbb{P}_t$ stands for the Poisson kernel of $\R^{n+1}_+$. Therefore, if $\tilde{Q_t}$ stands for the convolution operator associated to the symbol $\widehat{u}(t\xi)$, using the commutativity of multipliers we have
\[
	Q_t T_t F(x)=\tilde{Q}_t T_t \brkt{ t\d_t\mathbb{P}_t F}(x).
\]
Let $A_t=\tilde{Q}_t T_t$. Then
\[
	I=\int\int_0^\infty Q_t T_t F(x)\, G(t,x) v(t,x)\frac{\dd t}{t}\dd x=\int\int_0^\infty t \d_t \mathbb{P}_t F(x)\, A_t^*\brkt{G(t,\cdot) v(t,\cdot)}\frac{\dd t}{t}\dd x.
\]
Let $\Phi$ be the generalised Cauchy-Riemann system for $F$ i.e.
\[
	\Phi=\brkt{\mathbb{P}_t F, \mathfrak{R}_1\mathbb{P}_t F,\ldots, \mathfrak{R}_n\mathbb{P}_t F},
\]
where $\mathfrak{R}_j$ denotes the $j$-th Riesz transform given by
\[
	\mathfrak{R}_j(f)(x)=-i\int_{\R^n} \frac{\xi_j}{\abs{\xi}} \widehat{f}(\xi) e^{i x\xi}\ddd \xi.
\]
For further information see e.g. \cite{SW}. It is known (see \cite{S}) that $\Phi$ satisfies
\[
	\abs{\d_t\mathbb{P}_t F}^2\lesssim \abs{\nabla_{t,x} \Phi}^2\thickapprox \abs{\Phi}\Delta_{t,x}\abs{\Phi}.
\]

Now, the Cauchy-Schwarz inequality yields
\begin{equation}\label{eq:technic_yabuta}
	\abs{I}\lesssim\brkt{\int \int_0^\infty t{\Delta_{t,x}\abs{\Phi}}\dd t\dd x}^{1/2}
	\brkt{\int \int_0^\infty \abs{\Phi}\abs{A_t^*\brkt{G(t,\cdot) v(t,\cdot)}}^2 \frac{\dd t}{t}\dd x}^{1/2}.
\end{equation}
Integrating by parts, the first factor in the previous inequality becomes
\[
	\brkt{\int \abs{\Phi(x,0)}\dd x}^{\frac{1}2}\lesssim \norm{F}_{H^1}^{1/2}.
\]
By the hypothesis of the proposition, $\abs{G(t,x)}^2\frac{\dd t}{t}\dd x$ is Carleson measure, therefore we have that
$\abs{G(t,x) v(t,x)}^2\frac{\dd t}{t}\dd x$ is also a Carleson measure with norm bounded by $\norm{v}_{L^\infty_{t,x}}^2\norm{\mathrm{d} \mu_G}_{\mathcal{C}}$. Now, by the hypothesis \eqref{eq:decay_measure} and the fact that $u\in \S$, the kernel of $A_t$, which we denote by $K_t$,  satisfies the condition \eqref{eq:decay}. So we can apply Lemma \ref{lem:composition} to deduce that
\[
	\int \abs{K_t(x-y)} \abs{\brkt{G(t,\cdot) v(t,\cdot)}(y)}^2\dd y \frac{\dd t}{t}\dd x
\]
is a Carleson measure. Since  $\sup_t \norm{K_t}_{L^1}<+\infty$ and
\[
	\abs{A_t^*\brkt{G(t,\cdot) v(t,\cdot)}(x)}\lesssim \brkt{\sup_{t}\norm{K_t}_{L^1}}^{1/2} \brkt{\int \abs{K_t(x-y)}\abs{\brkt{G(t,\cdot) v(t,\cdot)}(y)}^2\dd y}^{1/2},
\]
we deduce that
\begin{equation}\label{eq:carleson_technical}
	\abs{A_t^*\brkt{G(t,\cdot) v(t,\cdot)}(x)}^2\frac{\dd t}{t}\dd x,
\end{equation}
is a Carleson measure with norm bounded by a constant times $\norm{\mathrm{d}\mu_G}_{\mathcal{C}}\norm{v}_{L^\infty_{t,x}}^2$.
Therefore, the second term in \eqref{eq:technic_yabuta} is bounded by
\[
	\norm{\mathrm{d}\mu_G}_{\mathcal{C}}^{1/2}\norm{v}_{L^\infty_{t,x}}\brkt{\int \sup_{t>0} \sup_{\abs{x-y}<t}  \abs{\Phi(t,y)}\dd x}^{1/2}\lesssim \norm{F}_{H^1}^{1/2}\norm{\mathrm{d}\mu_G}_{\mathcal{C}}^{1/2}\norm{v}_{L^\infty_{t,x}}.
\]
So, putting this all together
\[
	\abs{I}\lesssim \norm{F}_{H^1}\norm{\mathrm{d}\mu_G}_{\mathcal{C}}^{1/2}\norm{v}_{L^\infty_{t,x}},
\]
which concludes the proof.
\end{proof}

\begin{cor}
\label{cor:yabuta} There exists a constant $C$ such that for any $F\in H^1$, $G\in \mathrm{BMO}$ and $v\in L^\infty_{t,x}$,
\[
	\abs{\int\int_0^\infty Q_t F(x)\, \tilde{Q}_t G(x)\, v(t,x)\frac{\dd t}{t}\dd x}\leq C \norm{F}_{H^1}\norm{G}_{\mathrm{BMO}}\norm{v}_{L^\infty_{t,x}}.
\]
\end{cor}
\begin{proof} Put $G(t,x)=\tilde{Q}_t G(x)$, $\nu_t=\delta_0$ in Proposition \ref{prop:yabuta} and use \eqref{eq:standard_CM}.
\end{proof}

We wish to apply Corollary \ref{cor:yabuta} with $F=h$, $G=T_{\nu_{m_1}}^{\phase_1} f$ and \begin{equation}\label{eq:v_Yabuta}
	v(t,x)= {P_t^{2,v} R_t T_\gamma^{\phase_2}g (x)}\, m(t,x).
\end{equation}
By \eqref{eq:boundedness_of_m} and \eqref{eq:lack_of_cake_8} we have $\sup_{t,x}\abs{v(t,x)}\lesssim \norm{g}_{L^\infty}$, and by \ref{eq:SSS_infty} in Theorem \ref{linear_FIO} we have $T_{\nu_{m_1}}^{\phase_1} f\in \mathrm{BMO}$. Consequently we can indeed apply Corollary \ref{cor:yabuta} to obtain
\begin{equation}
\abs{I_4}\lesssim \norm{T_{\nu_{m_1}}^{\phase_1} f}_{\mathrm{BMO}} \norm{h}_{H^1} \norm{g}_{L^\infty}\lesssim
\norm{f}_{L^\infty} \norm{h}_{H^1} \norm{g}_{L^\infty} .
\end{equation}

For $I_6$, since \begin{equation}\label{eq:Carleson_technic}
	 \abs{\brkt{Q_t^u T_{\nu_{m_1}}^{\varphi_1}f(x)} \brkt{Q_t^{2,v} R_t T_\gamma^{\phase_2}g (x)}}\leq \frac{1}{2}\brkt{\abs{Q_t^u T_{\nu_{m_1}}^{\varphi_1}f(x)}^2+ \abs{Q_t^{2,v} R_t T_\gamma^{\phase_2}g (x)}^2}.
\end{equation}
and  both terms on the right hand side define a Carleson measure by \eqref{eq:standard_CM} and \eqref{eq:carleson} respectively, we have that
\[
	\dd\mu(x,t)=\brkt{Q_t^u T_{\nu_{m_1}}^{\varphi_1}f(x)} \brkt{Q_t^{2,v} R_t T_\gamma^{\phase_2}g (x)}\dd x \frac{\dd t}{t},
\]
is a Carleson measure and $	\norm{\dd\mu}_{\mathcal{C}}\lesssim \norm{f}_{L^\infty}^2+ \norm{g}_{L^\infty}^2=2$. Now \cite{FS}*{Theorem 4} and  \eqref{eq:boundedness_of_m} imply that
\[
	\begin{split}
	\abs{I_6}&\lesssim \int \int_0^1 \abs{\tilde{P}_t h(x)} \dd\abs{\mu}(x,t)\lesssim  \int \sup_{0<t<1} \sup_{\abs{x-y}<t} \abs{\tilde{P}_t h(y)} \dd x\lesssim \norm{h}_{H^1}.
	\end{split}
\]
\subsubsection*{{\bf Case II.c:} $m=-(n-1)$ and $H^1\times L^\infty \to L^1$ boundedness}

In this case, it remains to estimate $I_4$ and $I_6$ with $m_1=m_2=-(n-1)/2$, for $f\in H^1$ $g,h\in L^\infty$.

At this point, we shall recall some facts from the theory of local Hardy spaces useful for our current purposes.  It is shown in \cite{Gol} that a function $f$ belongs to the local Hardy space $h^1$ if, and only if $f\in L^1$ and $\mathfrak{R}_j((1-\widehat{\theta})(D)f)\in L^1$ where $\widehat{\theta}$ is a bump function supported near the origin and $\mathfrak{R}_j$  denotes the $j$-th Riesz transform. Moreover
\begin{equation}\label{eq:local_H1}
\begin{split}
	\norm{f}_{h^1} &\thickapprox \norm{f}_{L^1}+\sum_{j=1}^n \norm{\mathfrak{R}_j((1-\widehat{\theta})(D)f)}_{L^1}\thickapprox \norm{\widehat{\theta}(D)f}_{L^1} +\norm{(1-\widehat{\theta})(D) f}_{H^1},
\end{split}
\end{equation}
where the last estimate follows from the well-known characterization of the $H^1$ norm (see e.g. \cite{S0}*{p. 221}). We would like to mention that different choices of $\theta$ yield equivalent norms.

Since the frequency support of the symbol associated to $Q_t^u$ is supported in an annulus, we can find a smooth compactly supported bump function $\widehat{\theta}$  such that, for any $0<t<1$
\[
	Q_t^u=Q_t^u\brkt{\brkt{1-\widehat{\theta}}(D)}.
\]
Then one can write
\[
	\begin{split}
{I_4} &={\int \int_0^1  \brkt{Q_t^u\brkt{(1-\widehat{\theta})(D) T_{\nu_{m_1}}^{\varphi_1}f(x)}} \brkt{P_t^{2,v}  R_t T_\gamma^{\phase_2}g (x)} \tilde{Q}_{t} h(x)\, m(t,x) \frac{\dd t}{t}\dd x}.
\end{split}
\]
Taking $F=(1-\widehat{\theta})(D) T_{\nu_{m_1}}^{\varphi_1}f$, $G=h$ and $v(t,x)$ as in \eqref{eq:v_Yabuta}, and applying Corollary \ref{cor:yabuta}, \eqref{eq:local_H1} and the natural embedding of $L^\infty$ into $\mathrm{BMO}$, one has
that
\[
	\abs{I_4}\lesssim\norm{(1-\widehat{\theta})(D)T_{\nu_{m_1}}^{\varphi_1}f}_{H^1} \norm{g}_{L^\infty} \norm{h}_{\mathrm{BMO}}\lesssim \norm{T_{\nu_{m_1}}^{\varphi_1}f}_{h^1}\norm{g}_{L^\infty}\norm{h}_{L^\infty}.
\]
Finally, applying \ref{eq:Peloso} in Theorem \ref{linear_FIO} and the embedding of  $H^1$ into $h^1$, one obtains
\[
	\abs{I_4}\lesssim \norm{f}_{H^1}\norm{g}_{L^\infty}\norm{h}_{L^\infty}.
\]

To control $I_6$ we observe that a similar analysis to that of $I_4$ allows us to write
\[
	\begin{split}
{I_6} &={\int \int_0^1  \brkt{Q_t^u\brkt{(1-\widehat{\theta})(D) T_{\nu_{m_1}}^{\varphi_1}f(x)}} \brkt{Q_t^{2,v}  R_t T_\gamma^{\phase_2}g (x)} \tilde{P}_{t} h(x)\, m(t,x) \frac{\dd t}{t}\dd x}.
\end{split}
\]
Taking $v(t,x)=\tilde{P}_{t} h(x)\, m(t,x)$,  $F=(1-\widehat{\theta})(D) T_{\nu_{m_1}}^{\varphi_1}f(x)$, $G(t,x)=Q_t^{2,v}  R_t T_\gamma^{\phase_2}g (x)$ and $\nu_t=\delta_0$, and applying Proposition \ref{prop:yabuta} and the fact that $\abs{G(t,x)}^2\dd x \dd t/t$ is a Carleson measure (see  \eqref{eq:carleson}) yields
\[
	\abs{I_6}\lesssim\norm{(1-\widehat{\theta})(D)T_{\nu_{m_1}}^{\varphi_1}f}_{H^1} \norm{g}_{L^\infty} \norm{h}_{\mathrm{BMO}}\lesssim \norm{f}_{H^1}\norm{g}_{L^\infty}\norm{h}_{L^\infty}.
\]

\subsubsection*{{\bf Case II.d:} $m=-(n-1)$ and $L^\infty\times H^1 \to L^1$ boundedness}

In this case $g\in H^1$, $f,h\in L^\infty$ and $m_1=m_2=-(n-1)/2$.
Without loss of generality we can assume that $\norm{f}_{L^ \infty}=\norm{h}_{L^ \infty}=1$.

Arguing as in the previous case, we can find a smooth compactly supported bump function $\widehat{\theta}$  such that
\begin{equation}\label{eq:I_6}
	I_6={\int\int_0^{1} \brkt{Q_t^u T_{\nu_{m_1}}^{\varphi_1}f(x)} \brkt{Q_t^{2,v} R_t (1-\widehat{\theta})(D) T_\gamma^{\phase_2}g (x)} \tilde{P}_{t}h(x) m(t,x) \frac{\dd t}{t}\dd x}.
\end{equation}
Then, we use Proposition \ref{prop:yabuta} with $v(t,x)= \tilde{P}_{t}h(x) m(t,x)$, $G(t,x)=Q_t^u T_{\nu_{m_1}}^{\varphi_1}f(x)$, $F=(1-\widehat{\theta})(D)T_\gamma^{\phase_2}g (x)$ and $\dd\nu_t(x)=K_t(x)\dd x$, where $K_t$ is the kernel of $R_t$ defined in \eqref{eq:Miracolous_kernel}. We obtain
\[
\abs{I_6}\lesssim\norm{T_{\nu_{m_1}}^{\varphi_1}f}_{\mathrm{BMO}}\norm{(1-\widehat{\theta})(D)T_{\nu_{m_1}}^{\varphi_1}g}_{H^1} \norm{h}_{L^\infty}\lesssim \norm{f}_{L^\infty}\norm{g}_{H^1}\norm{h}_{L^\infty},
\]
where the last estimate follows from \eqref{eq:local_H1}, \ref{eq:SSS_infty} and \ref{eq:Peloso} in Theorem \ref{linear_FIO}.

In order to control $I_4$, take $G(t,x)=\brkt{Q_t^u T_{\nu_{m_1}}^{\varphi_1}f(x)} \tilde{Q}_{t}h(x) m(t,x)$, which, by a similar argument to that for \eqref{eq:Carleson_technic}, it is seen to give rise to a Carleson measure with norm bounded by a constant (here we use that $\norm{f}_{L^ \infty}=\norm{h}_{L^ \infty}=1$). Then
\[
	I_4=\int \int_0^ 1 \brkt{P_t^ {2,v} T_\gamma^{\phase_2}g }(x) R_t^*\brkt{G(t,\cdot)}(x)\frac{\dd t}{t}\dd x.
\]
By Lemma \ref{lem:composition}, $R_t^*\brkt{G(t,\cdot)}(x)$ gives rise to a Carleson measure, so we can apply  \cite{G}*{Corollary 7.3.6}  to obtain
\[
	\abs{I_4}\lesssim \int \sup_{0<t<1} \sup_{\abs{x-y}<t} \abs{P_t^ {2,v} T_\gamma^{\phase_2}g(x)}\dd x\lesssim \norm{T_\gamma^{\phase_2}g}_{h^ 1}\lesssim \norm{g}_{H^1},
\]
where the second inequality follows from the definition of the $h^1$ norm in \eqref{eq:h1} in such a way that the implicit constant depends polynomially on $v$, and the last estimate follows from \ref{eq:Peloso} in Theorem \ref{linear_FIO}.

\section{Endpoint cases for $n=1$}\label{sect:1d}

In this section we provide counterexamples in dimension $n=1$ to some of the endpoint results which are valid in dimension two and higher.

Let us start by discussing the failure of the $L^\infty\times L^\infty \to \mathrm{BMO}$ boundedness.

\begin{prop} \label{Prop_5.1}There exist $f,g\in L^\infty_c$ and $\sigma\in S^{0}_{1,0}(n,2)$ such that $T^{\phase_1,\phase_2}_\sigma(f,g)\not \in \mathrm{BMO}$ where $\phase_1(x,\xi)=x\xi+\abs{\xi}$ and $\phase_2(x,\eta)=x\eta$.
\end{prop}
\begin{proof}
Consider $\sigma(x,\xi,\eta)=\chi(x)\in S^0_{1,0}(n,2)$, with $\chi\in \mathcal{C}^\infty_0(\R^n)$ supported in $(-2,2)$, equal to $1$ in $[-1,1]$ and even.
Then for any $f$ and $g$ we have that
\begin{equation}\label{eq:model}
	T_\sigma^{\phase_1,\phase_2}(f,g)(x)=\chi(x)T_{1}^{\phase_1}(f)(x) g(x),
\end{equation}
where
\[
	T_{1}^{\phase_1}(f)(x):=\int \widehat{f}(\xi) e^{i \abs{\xi}+i x\xi}\ddd \xi=\frac{f(x+1)+f(x-1)}{2}+i\frac{Hf(x+1)-Hf(x-1)}{2},
\]
and $H$ stands for the Hilbert transform. Recall now that  $H\chi_{[-1,1]}(x)=\frac{1}{\pi} \log \abs{\frac{x+1}{x-1}}$. Therefore
\[
	T_{1}^{\phase_1}(\chi_{[-1,1]})(x)=S(x)+\frac{i}{2\pi}\log \abs{\frac{(x+2)(x-2)}{x^2}},
\]
where $S(x)=\frac{1}{2}\brkt{\chi_{[-2,0]}(x)-\chi_{[0,2]}(x)}$, which is trivially bounded by $1$. Define
\[
	h(x)=T_{1}^{\phase_1}(\chi_{[-1,1]})(x)-S(x).
\]
Now observe that for any $0<\varepsilon<1$,
\[
	-2\pi i \int_0^\varepsilon h(x)\dd x=\varepsilon\brkt{\log(\varepsilon+2)+\log(2-\varepsilon)-2\log \varepsilon}+2\brkt{\log(\varepsilon+2)-\log(2-\varepsilon)-2}.
\]
Then, if we define $h_+(x)=h(x)\chi_{[0,1]}(x)$,
\[
	\abs{{\rm Avg}_{(-\varepsilon,\varepsilon)}h_+}=\frac{1}{4\pi}\abs{{\log \frac{(2+\varepsilon)(2-\varepsilon)}{\varepsilon^2}}-\frac{2}{\varepsilon}\brkt{2-\log \frac{2+\varepsilon}{2-\varepsilon}}},
\]
which yields
\[
	\lim_{\varepsilon\to 0^+}\abs{{\rm Avg}_{(-\varepsilon,\varepsilon)}h_+}=+\infty.
\]
But then since
\[
	\frac{1}{2\varepsilon}\int_{-\varepsilon}^\varepsilon \abs{h_+(x)-{\rm Avg}_{(-\varepsilon,\varepsilon)}h_+}\dd x\geq \frac{1}{2\varepsilon}\int_{-\varepsilon}^0 \abs{{\rm Avg}_{(-\varepsilon,\varepsilon)}h_+}\dd x =\frac{1}{2}\abs{{\rm Avg}_{(-\varepsilon,\varepsilon)}h_+},
\]
it follows that $h_+\not \in \mathrm{BMO}$, and therefore, since $S\in L^\infty$,
\[
	T_\sigma^{\phase_1,\phase_2}(\chi_{[-1,1]}(x),\chi_{[0,1]})=T_{1}^{\phase_1}(\chi_{[-1,1]})(x)\chi_{[0,1]}(x)\not \in \mathrm{BMO}.
\]
\end{proof}

\begin{cor} The operator $T_ {\sigma}^{\phase_1,\phase_2}$ defined in \eqref{eq:model} is not bounded from $L^\infty\times H^1\to L^1$.
 \end{cor}
 \begin{proof} Observe that by the definition of the operator, for any $f,g,h$
 \[
 	\langle T_ {\sigma}^{\phase_1,\phase_2}(f,g),h\rangle=\langle T_ {\sigma}^{\phase_1,\phase_2}(f,h),g\rangle.
 \]
 Then, since $\mathrm{BMO}=(H^1)^*$, $T_ {\sigma}^{\phase_1,\phase_2}:L^\infty\times H^1\to L^1$ is equivalent to $T_ {\sigma}^{\phase_1,\phase_2}:L^\infty\times L^\infty\to \mathrm{BMO}$. But as we saw in Proposition \ref{Prop_5.1}, this is impossible, and therefore  $T_ {\sigma}^{\phase_1,\phase_2}$  is not bounded from $L^\infty\times H^1\to L^1$
 \end{proof}

\begin{prop}
The operator $T_ {\sigma}^{\phase_1,\phase_2}$ defined in \eqref{eq:model} is not bounded from $L^\infty\times L^p\to L^p$ for any $1<p<\infty$.
\end{prop}
\begin{proof} Assume on the contrary that there exists $1<p<\infty$ such that the operator is bounded. In particular, taking $f=\chi_{[-1,1]}$ and $g_\varepsilon(x)=\chi_{[0,\varepsilon]}(x)$ with $0<\varepsilon<1$, yields
\[
	\norm{T_\sigma^{\phase_1,\phase_2}(f,g)}_{L^p}^p=\int_{0}^\varepsilon \abs{T_{1}^{\phase_1}(f)(x)}^p \dd x \leq C \varepsilon
\]
with $C$ independent on $\varepsilon$. This would imply in particular that
\[
	\frac{1}{\varepsilon}\int_{0}^\varepsilon \abs{\log x }^p\dd x \leq C
\]
uniformly in $\varepsilon$. But letting $\varepsilon$ tend to zero, we will reach a contradiction.
\end{proof}
\begin{rem}
We recall that when the phases $\phase_1(x,\xi)=x\xi$ and $\phase_2(x,\eta)=x\eta$, i.e. when we are dealing with one dimensional bilinear pseudodifferential operators, then all the end-point estimates discussed above are actually valid and there are no counterexamples in these cases, see e.g. \cite{GK}, \cite{GT1}. Thus, the phenomenon observed above concerning the lack of boundedness at the end points in dimension one, is a feature of bilinear $\mathrm{FIO}$s.
\end{rem}
\section{The high frequency case: Non-Banach endpoints} \label{sect:non_banach}

The fact that the dual of $L^p$ spaces for $0<p<1$ is trivial precludes the use of duality arguments as in the previous sections. Nevertheless, the proof of the non-Banach results follows a similar line of thought as the Banach case, albeit with some modifications. In this section, we confine ourselves to indicating the main differences in the argument to be made for obtaining the desired boundedness results.

Proceeding similarly to the analysis of the Banach case, we can decompose the original operator as in  \eqref{low-high}. Replacing \eqref{eq:annulus} by its discrete version and using the fact that $\widetilde{\sigma}_1$ in \eqref{eq:newsigma1} has fixed compact support in $(\xi,\eta)$, we can expand $\widetilde{\sigma}_1$ in a Fourier Series. This will reduce matters to the study of an operator that is of the type
\begin{equation} \label{three3}
\sum_{u,v\in \Z^n} \sum_{k\in \Z} T^{\phase_1}_{\nu_1^{t,u}}(f)(x) T^{\phase_2}_{\mu_1^{t,v}}(g)(x) \frac{m(t,x,u,v)}{(1 + |u|^2 + |v|^2)^N},
\end{equation}
which can be regarded as a discrete counterpart of \eqref{three2}. In this expression $N$ is a sufficiently large integer. For the sake of brevity and to facilitate the comparison with the Banach case, we shall from now on set $t=2^{-k}$ for $k\in \Z$.

One also has a similar expression for $T_{\sigma_2}^{\phase_1, \phase_2}$, but once again, it will be enough to prove the boundedness of the part corresponding to $\sigma_1$. Thus, the study reduces to proving the estimates for an operator of the type
\[
	T(F,G)(x)=\sum_{k\geq 0} T^{\phase_1}_{\nu_1^{t,u}}(f)(x) T^{\phase_2}_{\mu_1^{t,v}}(g)(x){m(t,x,u,v)}
\]
with polynomial bounds in $\abs{u}$ and $\abs{v}$.

To obtain the desired boundedness, we shall first consider the case $n\geq 2$. For any fixed $-(n-1)\leq m<0$, we let $0<\alpha<2$ be such that $m=-(n-1)\brkt{\frac{1}{\alpha}-\frac{1}{2}}$. We claim that the operator $T$ defined above satisfies
\begin{equation}\label{eq:endpoints}
	T: H^{\alpha}\times L^2 \to L^{\beta}\qquad {\rm and }\qquad  T: L^2\times H^{\alpha}\to L^{\beta},
\end{equation}
whenever $\frac{1}{\beta}=\frac{1}{\alpha}+\frac{1}{2}=1-\frac{m}{n-1}$. These yield the same boundedness for $T_{\sigma_1}$.

By the results of the previous section, one has
\[
	T_{\sigma_1}: H^p\times H^q\to  L^1\qquad\text{and}\qquad T_{\sigma_1}: H^q\times H^p\to  L^1,
\]
with
\[
	\frac{1}{p}={\frac{3}{4}-\frac{1}{2\alpha}}=\frac{1}{2}+\frac{m}{2(n-1)}, \qquad\text{and}\qquad \frac{1}{q}=\frac{1}{2}-\frac{m}{2(n-1)}.
\]
Using these results, assuming that \eqref{eq:endpoints} holds, and applying complex interpolation gives us that
\[
	T_{\sigma_1}: H^p\times H^q\to  L^r
\]
boundedly provided $m=-(n-1)\brkt{\abs{\frac{1}{p}-\frac{1}{2}}+\abs{\frac{1}{q}-\frac{1}{2}}}$ and $\frac{1}{r}=\frac{1}{p}+\frac{1}{q}\geq 1$ and $1\leq p,q\leq \infty$. In an abuse of notation, we identify $H^\infty$ with $L^\infty$ in the previous expressions.

In order to prove \eqref{eq:endpoints}, by a similar argument to the Banach case, but using sums instead of integrals, one reduce matters to the study of operators
\begin{equation}\label{eq:disc_I1}
   I_1=\sum_{k\geq 0} \brkt{Q_t^u T_{\nu_{m_1}}^{\varphi_1}f(x)} \brkt{P_t^v T_{\mu_1^t}^{\phase_2} g(x)} m(t,x),
\end{equation}
where the error terms can be handled routinely as in the Banach case, where we use the well-known inequality
\[
	\brkt{\sum_{k} \abs{a_k}}^r\leq \sum_{k} \abs{a_k}^r, \qquad \text{for $0<r\leq 1$},
\]
instead of the Minkowski integral inequality.

Let us start by proving the $H^{\alpha}\times L^2 \to L^{r}$ boundedness, in which case we have $m_1=m$ and $m_2=0$  in \eqref{eq:disc_I1}. We observe that $\mu_1^t$ is actually independent of $t$. Therefore $I_1$ is a composition of a paraproduct $\tilde{\Pi}(F,G)=\sum_{k\geq 0} \brkt{Q_t^u F(x)} \brkt{P_t^v G(x)} m(t,x,u,v)$, satisfying similar properties as those considered in Lemma \ref{lem:Paraproduct}, and two FIOs. In particular, $\tilde{\Pi}$ is a bilinear pseudodifferential operator with symbol
\[
	\lambda(x,\xi,\eta)=\sum_{k\geq 0} \widehat{\psi}(t\xi)e^{it\xi \cdot v}\widehat{\theta}(t\eta)e^{it\eta \cdot u} m(t,x,u,v),
\]
belonging to $S^0_{1,0}(n,2)$ with seminorms depending polynomially on $\abs{u}$ and $\abs{v}$. Then, \cite{GK}*{Theorem 1.1} yields that this paraproduct is $H^{\alpha}\times L^2 \to L^{r}$ bounded. Since $0<t<1$, there exists a smooth bump function $\widehat{\theta}$ such that $Q_t^u=Q_t^u\brkt{(1-\widehat{\theta})(D)}$. Therefore, \cite{Gol}*{Lemma 4} yields that $\tilde{\Pi}$ maps $h^{p}\times L^2 \to L^{r}$ and \ref{eq:SSS}, \ref{eq:Peloso} in Theorem \ref{linear_FIO} yield the result.

For proving the $L^2\times H^{\alpha}\to L^{r}$ boundedness, one proceeds as in the Banach case and matters reduce to giving the desired estimate for
\[
    \sum_{k\geq0} \brkt{Q_t^u T_{\nu_{m_1}}^{\varphi_1}f(x)} \brkt{\tilde{P}_{t}^v R_t T_{\gamma}^{\phase_2}g(x)}m(t,x,u,v),
\]
which is the discrete counterpart of \eqref{eq:final_reduction}, where $m_1=0$, $m_2=m$ and $R_t$ is the discrete analogue of \eqref{eq:Miracle}. More precisely,
$R_t (G)(x)=\int K_t(x-y) G(y)\dd y$, with
\begin{equation*}
	K_t(z)=\sum_{-\log_2 \kappa\leq j\leq k}2^{j m_2} \Psi\brkt{2^{k-j}z} 2^{n(k-j)},
\end{equation*}
and as in \eqref{eq:new_psi}
\begin{equation}\label{eq:new_psi_2}
	\widehat{\Psi}(\eta)=\widehat{\psi}(\kappa_1\eta)^2\abs{\eta}^{m_2}:=\widehat{\psi}(\kappa_1\eta) \widehat{\tilde{\Psi}}(\eta),
\end{equation}
where $\widehat{\psi}$ is smooth, radial and positive with
\[
	\supp \widehat{\psi}\subset \set{\xi:\, 2^{-2}\leq \abs{\xi}\leq 1},
\]
and
\[
	\sum_{j\in \Z} \widehat{\psi}(2^j \eta)^2=1\qquad {\text{for any}\, \eta\neq 0}.
\]
Hence, the desired boundedness would follow if we establish that the operator given by
\begin{equation*}
    S(F,G)(x)=\sum_{k\geq 0} \brkt{Q_t^u F(x)} \brkt{\tilde{P}_{t}^v R_t G(x)}m(t,x,u,v)
\end{equation*}
is $L^2\times h^{\alpha}\to L^{r}$ bounded. Observe also that there exists a smooth bump function $\widehat{\theta}$ such that $R_t=R_t\brkt{(1-\widehat{\theta})(D)}$. Then \cite{Gol}*{Lemma 4} implies that it is enough to prove the $L^2\times H^\alpha\to L^r$ boundedness of the operator $S(F,G)$ defined above. To this end, the Cauchy-Schwarz inequality yields
\[
	\abs{S(F,G)(x)}\lesssim \brkt{\sum_{k\geq 0}  \abs{Q_t^u F(x)}^2}^{1/2} \brkt{\sum_{k\geq 0}  \abs{P_t^v R_t G(x)}^2}^{1/2}.
\]
Therefore, using H\"older's inequality and standard quadratic estimates, we have
\[
	\norm{S(F,G)}_{L^{r}}\lesssim \norm{F}_{L^2} \brkt{\int {\brkt{\sum_{k\geq 0} \abs{P_t^v R_t G(x)}^2}^{\alpha/2}}\dd x}^{1/\alpha}.
\]
Observe now that by \eqref{eq:new_psi_2}
\[
	P_t R_t G(x)=\sum_{-\log_2 \kappa\leq j\leq k} s^{m_2} \int \big(\theta_t*\tilde{\Psi}_{ts}\big)(y) \big(\psi_{ts\kappa_1}*G\big)(x-y)\dd y
\]
where we set $s=2^j$. Then for any $b>0$ (to be later determined)
\[
	\abs{P_t R_t G(x)}\leq \sum_{-\log_2 \kappa\leq j\leq k} s^{m_2} \int \abs{\theta_t*\tilde{\Psi}_{ts}(y)}\brkt{1+\frac{\abs{y}}{ts}}^b\dd y\, M^{**}_b(G,\psi_{ts\kappa_1})(x),
\]
where
\[
	M^{**}_b(G,\psi_u)(x):=\sup_{y\in \R^n} \frac{\abs{\psi_u*G(x-y)}}{\brkt{1+\frac{\abs{y}}{u}}^b},
\]
denotes Peetre's maximal function, see \cite{Trie}.

Now \cite[p. 16]{Trie} yields
\[
	M^{**}_b(G,\psi_u)(x)\lesssim \left[M\brkt{\abs{\psi_u*G}^{\frac{b}{n}}}(x)\right]^\frac{n}{b},
\]
for any $x\in \R^n$ and any $u>0$, where $M$ stands for the Hardy-Littlewood maximal operator.  Moreover, by the results in \cite{G}*{Appendix K}, one has for any $N>n+b$
\[
	 \abs{\theta_t*\tilde\Psi_{ts}(y)}\lesssim  (t\max\brkt{s,1})^{-n}
	 \brkt{1+\frac{\abs{y}}{t\max\brkt{s,1}}}^{-N},
\]
which in turn implies that
\[
	 \sup_{0<t<1, \kappa< s<1/t}\int \abs{\theta_t*\tilde\Psi_{ts}(y)}\brkt{1+\frac{\abs{y}}{ts}}^b\dd y\lesssim 1.
\]
Thus, we have the pointwise inequality
\[
	\abs{P_t R_t G(x)}\lesssim \sum_{-\log_2 \kappa\leq j\leq k}s^{m_2}\left[M\brkt{\abs{\psi_{st\kappa_1}*G}^{\frac{b}{n}}}(x)\right]^\frac{n}{b}.
\]
Therefore, for any $q>\max\brkt{\frac{b}{n},1}$
\[
	\begin{split}
	\brkt{\sum_{k\geq 0} \abs{P_t R_t G(x)}^{q}}^{1/q}&\lesssim \sum_{j=-\log_2 \kappa}^\infty s^{m_2}\brkt{ \sum_{k\geq j}\left[M\brkt{\abs{\psi_{st\kappa_1}*G}^{\frac{b}{n}}}(x)\right]^\frac{qn}{b}}^{\frac{1}{q}}\\
	&\leq C_{m_2} \brkt{ \sum_{k\geq 0} \left[M\brkt{\abs{\psi_{t\kappa_1}*G}^{\frac{b}{n}}}(x)\right]^\frac{qn}{b}}^{\frac{1}{q}}
	\end{split}
\]
where $C_{m_2}=\sum_{j=-\log_2\kappa}^\infty s^{m_2}=\sum_{j=-\log_2\kappa}^\infty 2^{j m_2}<+\infty$, and we have used the fact that $st=2^{j-k}$, which allows us to re-index the sum.

Hence, for any $p>\frac{b}{n}$ the boundedness of the vector-valued maximal operator
\[
	\norm{\brkt{\sum_{k\geq 0}\abs{M{f_k}(\cdot)}^{\frac{qn}{b}}}^{\frac{b}{qn}}}_{L^{\frac{pn}{b}}(\R^n)}\lesssim \norm{\brkt{\sum_{k\geq 0}\abs{{f_k}(\cdot)}^{\frac{qn}{b}}}^{\frac{b}{qn}}}_{L^{\frac{pn}{b}}(\R^n)},
\]
(proved by Fefferman-Stein in \cite{FS2}*{Theorem 1}) yields
\[
	\norm{\brkt{\sum_{k\geq 0} \abs{P_t R_t G}^{q}}^{1/q}}_{L^p(\R^n)}\lesssim \left[\int \brkt{\sum_{k\geq 0} {\abs{\psi_{t\kappa_1}*G(x)}}^q}^{\frac{p}{q}}\dd x\right]^{\frac{1}{p}}.
\]
Now if we define $\delta_{\kappa_1}G(y)=G(\kappa_1 y)$, a change of variables yields that the last term is equal to
\[
	\left[\int \brkt{\sum_{k\geq 0} {\abs{\psi_{t}*\delta_{\kappa_1}G(x)}}^q}^{\frac{p}{q}}\dd x\right]^{\frac{1}{p}}	\kappa_1^{\frac{n}{p}}\lesssim \norm{\delta_{\kappa_1}G}_{F^0_{p,q}}.
\]
where $F^0_{p,q}$ stands for the corresponding Triebel-Lizorkin space (see \cite{Trie} for the general definition and further properties of these spaces). In particular, taking $b<\alpha n<2n$, $q=2$ and using the identification of the Triebel-Lizorkin space $F_{\alpha,2}^0$ given in \cite{Trie}*{Theorem 1, p. 92}, we obtain
\[
	\norm{S(F,G)}_{L^{r}}\lesssim \norm{F}_{L^2}\norm{\delta_{\kappa_1}G}_{h^{\alpha}}.
\]
But this yields the desired result because
\[
	\norm{\delta_{\kappa_1}G}_{h^{p}}\leq \norm{\delta_{\kappa_1}G}_{H^{\alpha}}\thickapprox \norm{G}_{H^{\alpha}},
\]
where the last equality follows from the fact that
\[
	\norm{\delta_{\kappa_1}G}_{H^{\alpha}}^\alpha=\int \sup_{t>0} \abs{\widehat{\Theta}(tD) \delta_{\kappa_1}G(x)}^\alpha\dd x=\int \sup_{t>0} \abs{\widehat{\Theta}(\kappa_1 tD)G(\kappa_1 x)}^\alpha\dd x=\kappa_1^{-n}\norm{G}_{H^{\alpha}}^\alpha.
\]
\\

The case $n=1$ is simpler, in the sense that, as in the previous case, matters reduce to the study of an operator of the type \eqref{eq:disc_I1} with $\mu_1^t$ independent of $t$. That is, $I_1$ is the composition of the paraproduct $\tilde{\Pi}$ defined above and two FIOs. In particular, Theorem \ref{linear_FIO} and the boundedness of the paraproducts \cite{GK}*{Theorem 1.1} yield
\[
	T_{\sigma_1}:H^p\times H^q\to L^{r},
\]
provided $m=0$, $1\leq \frac{1}{r}=\frac{1}{p}+\frac{1}{q}$ and $p,q<\infty$.

\section{A composition formula} \label{sec:proof_of_asymptotic}

In this section we prove  the Theorem \ref{left composition with pseudo} concerning composition of a pseudodifferential operator and an FIO.

Let $\chi\in \mathcal{C}^{\infty}(\mathbb{R}^n)$ such that $0\leq \chi\leq 1,$ $\chi(z)=1$ for $|z|<{\epsilon}/{2}$ and $\chi(z)=0$ for $|z|>\epsilon$ where $0<\epsilon<C_1/(4C_0)$, with $C_0=\sup_{(x,\xi)\in \Omega}\sup_{\abs{\alpha}=2}\brkt{ \abs{\d^\alpha_x \varphi(x,\xi)}/\bra{\xi}}$ and $C_1$ is the constant in \ref{partone}. We decompose $\sigma_t(x,\xi)$ into two parts $\textbf{I}_1 (t,x,\xi)$ and $\textbf{I}_2 (t,x,\xi)$ where
\begin{equation*}
  \textbf{I}_1 (t,x,\xi)=\iint a_t(y,\xi)\rho(t\eta)\, (1-\chi(x-y))\,e^{i(x-y)\cdot\eta+i\phase(y,\xi)-i\phase(x,\xi)}\, \ddd\eta \dd y,
  \end{equation*}
  and
\begin{equation*}
  \textbf{I}_2 (t,x,\xi)=\iint a_t(y,\xi)\rho(t\eta)\,\chi(x-y)\,e^{i(x-y)\cdot\eta+i\phase(y,\xi)-i\phase(x,\xi)}\ddd\eta \dd y.
  \end{equation*}\\

We begin by analysing $\textbf{I}_1 (t,x,\xi)$. We claim that
\begin{equation}\label{eq:claim_0}
	\sup_{x\in \R^n }\abs{\d^\gamma_x \d^\beta_\xi \textbf{I}_1 (t,x,\xi)}\lesssim \bra{\xi}^{m-M\brkt{\frac{1}{2}-\varepsilon }-\abs{\gamma}} t^{M\varepsilon}.
\end{equation}
To simplify the exposition, we are going omit the estimates in the $x$-derivatives, that is we only study the case $\gamma=0$, as this argument contains the main ideas of the proof.

The chain rule and Leibniz formula tell us that $\d^\beta_\xi \textbf{I}_1 (t,x,\xi)$ is a finite linear combination of elements of the type
\[
	\iint \d^{\beta_1}_\xi a_t(y,\xi)\rho(t\eta)\,(1-\chi(x-y))\,\prod_{j=1}^k \d^{b_j}_\xi\brkt{\phase(y,\xi)-\phase(x,\xi)} e^{i(x-y)\cdot\eta+i\phase(y,\xi)-i\phase(x,\xi)}\ddd\eta \dd y.
\]
with $\beta_1+\beta_2=\beta$, $b_1+\ldots+b_k=\beta_2$ and, if $\beta_2\neq 0$, $\abs{b_j}\geq 1$ for each $j$. Since $\d^{b_j}_\xi\brkt{\phase(y,\xi)-\phase(x,\xi)}=\int_0^1 \nabla_x \d^{b_j}_\xi\phase(x+s(y-x),\xi))\dd s \cdot (y-x)$, we can express each of these terms as a finite linear combination of  terms of the type
\[
	\iint \d^{\beta_1}_\xi a_t(y,\xi)\rho(t\eta)\,(1-\chi(x-y))\,g_{k,\beta_2}(x,y,\xi) (x-y)^{\kappa} e^{i(x-y)\cdot\eta+i\phase(y,\xi)-i\phase(x,\xi)}\ddd\eta \dd y,
\]
for some multi-index $\kappa$, with $\abs{\kappa}=k$ and $g_{k,\beta_2}$ an smooth function satisfying
\[
	\abs{\d^\gamma_x \d^\nu_y g_{k,\beta_2}(x,y,\xi)}\lesssim \bra{\xi}^{k-\abs{\beta_2}}.
\]
Integrating by parts in $\eta$, the previous expression becomes a linear combination of terms like
\begin{equation}\label{eq:tecnichal_1}
	t^k \iint \d^{\beta_1}_\xi a_t(y,\xi) (\d^\kappa\rho)(t\eta)\,(1-\chi(x-y))\,g_{k,\beta_2}(x,y,\xi) e^{i(x-y)\cdot\eta+i\phase(y,\xi)-i\phase(x,\xi)}\ddd\eta \dd y.
\end{equation}
To estimate these terms, we introduce now the differential operators
\[
^{t}L_{\eta} =-i\sum_{j=1}^{n}\frac{x_j -y_j}{|x-y|^2}\partial_{\eta_j} \quad \text{and} \quad
^{t}L_{y} =\frac{1}{|\nabla_{y}\phase(y,\xi)|^2 -i\Delta_{y}\phase(y,\xi)}(1-\Delta_{y}).
\]
Now integration by parts yields that the term \eqref{eq:tecnichal_1} is equal to
\begin{equation*}
t^k \iint L^{N_{2}}_{y} \left[e^{-iy\cdot\eta}(1-\chi(x-y)) \d^{\beta_1}_\xi a_t(y,\xi) g_{k,\beta_2}(x,y,\xi) L^{N_1}_{\eta}[(\d^{\kappa}\rho)(t\eta)]\right]\,e^{ix\cdot\eta+i\phase(y,\xi)-i\phase(x,\xi)}\ddd\eta\, \dd y,
  \end{equation*}
which is a linear combination of terms of the type
\begin{equation}\label{eq:star}
 t^{k+N_1} \iint \Lambda(x,y,\xi,\eta,N_1,N_2)(\d^{\kappa+\gamma}\rho)(t\eta)\,e^{ix\cdot\eta+i\phase(y,\xi)-i\phase(x,\xi)}\ddd\eta\, dy,
  \end{equation}
  with $\Lambda(x,y,\xi,\eta,N_1,N_2)=L^{N_{2}}_{y} \left[e^{-iy\cdot\eta}(1-\chi(x-y)) \d^{\beta_1}_\xi a_t(y,\xi) g_{k,\beta_2}(x,y,\xi) \frac{{(x-y)^\gamma}}{\abs{x-y}^{2N_1}}\right]$ and $\abs{\gamma}=N_1$. Observe that since $t\leq 1$ and $\varepsilon <1/2$,
 \[
 	t^{k+N_1}\abs{(\d^{\kappa+\gamma} \rho) (t\eta)}\lesssim t^{k+N_1}\bra{t\eta}^{-k-N_1(1-\varepsilon )}\leq t^{N_1\varepsilon }\bra{\eta}^{-k-N_1(1-\varepsilon )}\leq t^{N_1\varepsilon }\bra{\eta}^{-N_1/2}.
 \]

Because of \ref{partone} one has
\[
	|\abs{\nabla_{y}\phase(y,\xi)}^2 -i\Delta_{y}\phase(y,\xi)|\geq \abs{\nabla_{y}\phase(y,\xi)}^2\gtrsim \abs{\xi}^2.
\]
Then, one can show that %
\[
	\abs{\Lambda(x,y,\xi,\eta,N_1,N_2)}\lesssim \bra{\eta}^{2N_2} \bra{\xi}^{m-\abs{\beta_1}} \bra{\xi}^{k-\abs{\beta_2}}\abs{x-y}^{-N_1} \bra{\xi}^{-2N_2}.
\]
Therefore, the term in \eqref{eq:star} is bounded by
\[
	t^{N_1 \varepsilon } \bra{\xi}^{m-\abs{\beta}} \bra{\xi}^{\abs{\beta}-2N_2}\iint_{|x-y|>\epsilon} \langle \eta\rangle^{2N_2-N_1/2} |x-y|^{-N_1} \ddd\eta\, \dd y.
\]
Thus, taking $	N_1=4N_2+M+2n$ {and} $N_2>(\abs{\beta}+M)/2$, claim \eqref{eq:claim_0} follows.\\

We now proceed to the analysis of $\textbf{I}_2 (t,x,\xi)$. First we make the change of variables $\eta=\nabla_x\phase(x,\xi)+\zeta$ in the integral defining $\textbf{I}_2 (t,x,\xi)$ and then expand $\rho(t\eta)$ in a Taylor series to obtain
\begin{equation*}\label{Eq:ralpha}
  \begin{aligned} \rho(t\nabla_x\phase(x,\xi)+t\zeta) & = \sum_{0 \leq |\alpha|<M} t^{|\alpha|}\frac{\zeta^\alpha}{\alpha !} (\partial_\xi^\alpha\rho)(t\nabla_x\phase(x,\xi)) + t^{M} \sum_{|\alpha|=M} C_\alpha {\zeta^\alpha} r_\alpha(t, x,\xi,\zeta),
  \end{aligned} \end{equation*}
where $\displaystyle  r_\alpha(t, x,\xi,\zeta)  = \int_0^1 (1-s)^{M-1} (\partial_\xi^{\alpha} \rho)(t\nabla_x\phase(x,\xi)+st\zeta) \dd s$. Then, setting
\[
\Phi(x,y,\xi)=\phase(y,\xi)-\phase(x,\xi)+(x-y)\cdot\nabla_x\phase(x,\xi),
\]
we can write
\[
\textbf{I}_2 (t,x,\xi)= \sum_{|\alpha|<M} \frac{t^{|\alpha|}}{\alpha!}\, \sigma_{\alpha}(t,x,\xi) + t^{M}\, \sum_{|\alpha|=M} C_\alpha\, R_{\alpha}(t,x,\xi),
\]
where
\begin{align*}
\sigma_{\alpha}(t,x,\xi) %
& = (\partial_\xi^\alpha \rho)(t\nabla_x\phase(x,\xi)) \partial_y^{\alpha}\left.\left[ e^{i\Phi(x,y,\xi)}a_t(y,\xi)\chi(x-y) \right]\right|_{y=x}
\end{align*}
and
\[
R_{\alpha}(t,x,\xi) = \iint e^{i(x-y)\cdot\zeta} e^{i\Phi(x,y,\xi)} \zeta^{\alpha} \chi(x-y) \,a_t(y,\xi)\, r_\alpha(t, x,\xi,\zeta) \dd y \ddd\zeta.
\]\\

We start by the study of $\sigma_\alpha$, for $\abs{\alpha}\geq 1$. Faa-Di Bruno's formula yields that for any multi-indices $\beta,\gamma$,
\begin{equation}\label{eq:FaadB}
	\abs{\d^\beta_x \d^\gamma_\xi \brkt{(\partial_\xi^\alpha \rho)(t\nabla_x\phase(x,\xi))}}\lesssim \sum_{j=1}^{\abs{\gamma}+\abs{\beta}} t^j \bra{t\xi}^{-j-\abs{\alpha}(1-\varepsilon )} \bra{t\xi}^{-\abs{\alpha}\varepsilon } \bra{\xi}^{j-\abs{\gamma}}.
\end{equation}
Since $t\leq 1$ we have that  $$\bra{t\xi}^{-j-\abs{\alpha}(1-\varepsilon )}\bra{t\xi}^{-\abs{\alpha}\varepsilon }\leq t^{-j-\abs{\alpha}(1-\varepsilon )} \bra{\xi}^{-j-\abs{\alpha}(1-\varepsilon )},$$ which yields
\begin{equation}\label{symbol estimate for rho}
	\abs{\d^\gamma_\xi \d_x^\beta \brkt{(\partial_\xi^\alpha \rho)(t\nabla_x\phase(x,\xi))}}\lesssim t^{\abs{\alpha}(\varepsilon -1)} \bra{\xi}^{-\abs{\alpha}(1-\varepsilon )-\abs{\gamma}}.
\end{equation}
We claim that, for any multi-indices $\beta,\gamma$ and any $\xi\in \R^n$,
\begin{equation}\label{eq:claim_1}
	\abs{\d^\gamma_\xi \d_x^\beta \brkt{\left.\partial_y^{\alpha}\left[ e^{i\Phi(x,y,\xi)} a_t(y,\xi) \chi(x-y)\right]\right|_{y=x}}}\lesssim  \bra{\xi}^{{|\alpha|/2}-\abs{\gamma}+m}.
\end{equation}
Observe that the last two estimates together imply that $\sigma_\alpha$ satisfies \eqref{eq:statement_1}.
In order to prove the claim, observe first that, since $\chi=1$ in a neighbourhood of the origin, $\d^\alpha \chi (0)=0$ for $\abs{\alpha}\geq 1$. This yields
\[
	\left.\partial_y^{\alpha}\left[ e^{i\Phi(x,y,\xi)} a_t(y,\xi) \chi(x-y)\right]\right|_{y=x}=
	\left.\partial_y^{\alpha}\left[ e^{i\Phi(x,y,\xi)} a_t(y,\xi) \right]\right|_{y=x},
\]
which is equal to
\begin{equation}\label{eq:1}
	\d^\alpha_x a_t(x,\xi)+\sum_{\alpha_1+\alpha_2=\alpha,\abs{\alpha_1}\geq 2}
	 C_{\alpha_1,\alpha_2}\left.\partial_y^{\alpha_1}\left[ e^{i\Phi(x,y,\xi)} \right]\right|_{y=x} \d^{\alpha_2}a_t(x,\xi),
\end{equation}
since $\left.\nabla_y\Phi(x,y,\xi)\right|_{y=x}=0$.  Faa-Di Bruno's formula yields, for $\abs{\gamma}\geq 1$,
\[
\partial_y^{\gamma} e^{i\Phi(x,y,\xi)}=\sum_{\gamma_1 + \cdots+ \gamma_k =\gamma} C_\gamma (\partial^{\gamma_{1}}_{y}\Phi(x,y,\xi))\cdots (\partial^{\gamma_{k}}_{y}\Phi(x,y,\xi))\,e^{i\Phi(x,y,\xi)},
\]
where the sum ranges of $\gamma_j$ such that $|\gamma_{j}|\geq 1$ for $j=1,2,\dots, k$ and $\gamma_1 + \cdots+ \gamma_k =\gamma$ for some $k \in \N$.  Taking into account that $\Phi(x,x,\xi)=0$ and $\left.\partial_y\Phi(x,y,\xi)\right|_{y=x}=0$, we have that
\[
\left.\partial_y^{\gamma} e^{i\Phi(x,y,\xi)}\right|_{y=x}=\sum_{\gamma_1 + \cdots+ \gamma_k =\gamma, \abs{\gamma_j}\geq 2} C_\gamma (\partial^{\gamma_{1}}_{x}\varphi(x,\xi))\cdots (\partial^{\gamma_{k}}_{x}\varphi(x,\xi)).
\]
Since  $\sum_{j=1}^{k} |\gamma_j |\leq |\gamma|$, we actually have $2k\leq |\gamma|$,
which yields
\[
	\abs{\d^\nu_\xi \d^\beta_x\left[ \left.\partial_y^{\gamma} e^{i\Phi(x,y,\xi)}\right|_{y=x}\right]}\lesssim \abs{\xi}^{k-\abs{\nu}}\leq \bra{\xi}^{ \frac{|\gamma|}{2}-\abs{\nu}}.
\]
Thus,  using the above estimates and \eqref{eq:1}, we have that
\[
	\abs{\d^\nu_\xi \d^\beta_x\left[ \left.\partial_y^{\alpha} e^{i\Phi(x,y,\xi)} a_t(x,\xi)\right|_{y=x}\right]}\lesssim %
	\bra{\xi}^{m-\abs{\nu}} \sum_{j=2}^{\abs{\alpha}} \bra{\xi}^{\frac{j}{2}}\lesssim
	 \bra{\xi}^{\frac{\abs{\alpha}}{2}+m-\abs{\nu}} \quad {\text{for $\abs{\alpha}\geq 2,$}}
\]
and
\[
\abs{\d^\nu_\xi \d^\beta_x\left[ \left.\partial_y^{\alpha} e^{i\Phi(x,y,\xi)} a_t(x,\xi)\right|_{y=x}\right]}=\abs{\d^\nu_\xi \d^{\beta+\alpha}_x a_t(x,\xi)}\lesssim \bra{\xi}^{m-\abs{\nu}}\quad {\text{for $\abs{\alpha}=1$}},
\]
which imply \eqref{eq:claim_1}.\\

Let $\alpha$ with $\abs{\alpha}=M$. To estimate the remainder $R_{\alpha}$ we observe that it is sufficient to control a term of the type
\begin{equation}\label{modified rest term}
\begin{split}
	&\tilde{R}_\alpha(t,s,x,\xi)=\iint e^{i(x-y)\cdot\zeta}  e^{i\Phi(x,y,\xi)}a_t(y,\xi) \chi(x-y) \zeta^\alpha \brkt{\d^\alpha \rho}\brkt{t\nabla_x\varphi(x,\xi)+st \zeta}\dd y \ddd\zeta\\
	&=\iint e^{i(x-y)\cdot\zeta}  D_y^{\alpha} \left[ e^{i\Phi(x,y,\xi)}\, \chi(x-y) a_t(y,\xi)\right] \brkt{\d^\alpha \rho}\brkt{t\nabla_x\varphi(x,\xi)+st \zeta}\dd y \ddd \zeta\\
	&=:R_\alpha^I(t,s,x,\xi)+ R_\alpha^{I\!\!I}(t,s,x,\xi)
\end{split}
\end{equation}
uniformly in $s\in (0,1)$, where  $R_\alpha^I(t,s,x,\xi)$ is the integral expression given by
\[
	\iint e^{i(x-y)\cdot\zeta}  D_y^{\alpha} \left[ e^{i\Phi(x,y,\xi)}\, \chi(x-y) a_t(y,\xi)\right] g\left(\frac{\zeta}{\bra{\xi}}\right) \brkt{\d^\alpha \rho}\brkt{t\nabla_x\varphi(x,\xi)+st \zeta}\dd y \ddd \zeta,
\]
$R_\alpha^{I\!\!I}=\tilde{R}_\alpha-R_\alpha^{I}$ and $g\in \mathcal{C}_0^\infty(\R^n)$ such that $g(x)=1$ for $|x|<r/2$ and $g(x)=0$ for $|x|>r$, for some small $8C_0\epsilon<r<C_1/2$, where $C_1$ is the constants appearing in  the assumption \ref{partone} and $C_0$ is defined above.\\

We claim that
\begin{equation}\label{eq:claim_2}
	\sup_{0<s<1}\sup_{x\in \R^n }\abs{\d^\gamma_x \d^\beta_\xi R^I_\alpha (t,s,x,\xi)}\lesssim t^{-\abs{\alpha}(1-\varepsilon )}\bra{\xi}^{m-\abs{\alpha}(\frac{1}2-\varepsilon )-\abs{\beta}}.
\end{equation}
As we did before, we are going omit the estimates in the $x$-derivatives to simplify the exposition. To control $R_\alpha^I$ we observe that assumption \ref{partone} on $\phase$ yields
\[
\begin{aligned}
\abs{\nabla_x\phase(x,\xi)+s\zeta}\leq & (C_2\sqrt{2}+r)\abs{\xi}\quad \text{and}\\
\abs{\nabla_x\phase(x,\xi)+s\zeta}\geq & |\nabla_x\phase|-|\zeta| \geq (C_1-r)\abs{\xi},
\end{aligned}
\]
for any $s\in [0,1]$. Therefore, $\abs{\nabla_x\phase(x,\xi)+s\zeta}$ and $\abs{\xi}$ are equivalent uniformly in $s\in [0,1]$. Arguing as in the derivation of \eqref{symbol estimate for rho}, this yields that for $|\zeta|\leq r\bra{\xi}$
\[
\abs{\d^\kappa_\zeta \d^\beta_x \d^\gamma_\xi \brkt{(\partial^\alpha \rho)(t\nabla_x\phase(x,\xi))+st \zeta}}\lesssim s^{\abs{\kappa}} t^{-\abs{\alpha}(1-\varepsilon )} \bra{\xi}^{-\abs{\alpha}(1-\varepsilon )-\abs{\kappa}-\abs{\gamma}}.
\]

On the other hand observe that
\[
	\abs{\d_\xi^\gamma \d^\kappa_\zeta g\left(\frac{\zeta}{\bra{\xi}}\right) }\lesssim \bra{\xi}^{-\abs{\gamma}-\abs{\kappa}}.
\]
We define
\[
	b_{\alpha,\alpha_2}(x,y,\xi,\zeta,s,t)=\d^{\alpha_2}_y a_t(y,\xi)g\left(\frac{\zeta}{\bra{\xi}}\right)  \brkt{\d^\alpha \rho}\brkt{t\nabla_x\varphi(x,\xi)+st \zeta}).
\]
It follows from the previous estimates that
\begin{equation}\label{eq:bs}
	\sup_{0<s<1}\abs{\d^\kappa_\zeta \d^\beta_x \d^\gamma_\xi \d^\nu_y b_{\alpha.\alpha_2}(x,y,\xi,\zeta,s,t)}\lesssim t^{-\abs{\alpha}(1-\varepsilon )}\bra{\xi}^{m-\abs{\alpha}(1-\varepsilon )-\abs{\kappa}-\abs{\gamma}}.
\end{equation}
 Observe that $R_\alpha^I(t,s,x,\xi)$ is a finite linear combination of expressions of the type
\[
	 \iint e^{i(x-y)\cdot\zeta}\,\partial_y^{\alpha_1} \chi(x-y)\, b_{\alpha,\alpha_2}(x,y,\xi,\zeta,s,t)\, \prod_{j=1}^ k \d^{a_j}_y(\Phi(x,y,\xi))\, e^{i\Phi(x,y,\xi)} \dd y \ddd\zeta,
\]
where $\alpha_1+\alpha_2+\alpha_3=\alpha$, $a_1+\ldots+a_k=\alpha_3$ with $0\leq k\leq \abs{\alpha_3}$ such that if $\alpha_3\neq 0$ then $\abs{a_j}\geq 1$ for $j=1,\ldots, k$. Therefore, $\d^\beta_\xi R^I_{\alpha}$ can be expressed as a finite linear combination of terms of the form
\[
	\begin{split}
	 &\iint \partial_y^{\alpha_1} \chi(x-y)\, \d^{\beta_1}_{\xi}b_{\alpha,\alpha_2}(x,y,\xi,\zeta,s,t)  \prod_{j=1}^ k \d^ {b_j}_{\xi}\d^{a_j}_y(\Phi(x,y,\xi))\, \prod_{j=1}^{l}\d^{d_j}_\xi \Phi(x,y,\xi)  \\
	 &\qquad \times  e^{i\Phi(x,y,\xi)+i(x-y)\cdot\zeta} \dd y \ddd\zeta,
	 \end{split}
\]
where  $\beta_1+\beta_2+\beta_3=\beta$, $b_1+\ldots+b_k=\beta_2$, $d_1+\ldots+d_l=\beta_3$ and if $\beta_3\neq 0$ then $\abs{d_ j}\geq 1$ for $j=1,\ldots l$ with $l\leq \abs{\beta_3}$. Observe that
\[
	\Phi(x,y,\xi)=(x-y)^T\cdot\int_0^1\int_0^1 u (\nabla^2_x  \varphi)(x+uv(y-x),\xi)\, \dd u\dd v \cdot (x-y).
\]
Thus
$
	\abs{\d^d_\xi \Phi(x,y,\xi)}\lesssim \abs{\xi}^{1-\abs{d}} |x-y|^2,
$
which implies that
\begin{equation} \label{bound}
	\abs{\prod_{j=1}^{l}\d^{d_j}_\xi \Phi(x,y,\xi)}\lesssim\abs{\xi}^{2l} \abs{x-y}^{2l} \abs{\xi}^{-\abs{\beta_3}-l}\leq (1+\abs{\xi}\abs{x-y})^{2\abs{\beta_3}}\abs{\xi}^{-\abs{\beta_3}}.
\end{equation}
We have also that, for $\abs{a}=1$,
\begin{equation}\label{eq:improvement}
	\d^{a}_y \Phi(x,y,\xi)= \int_0^1 (\nabla_x\d^{a}_x)\phi(x+u(y-x)),\xi)\dd u\cdot (y-x)=g_{a}(x,y,\xi)\cdot (x-y),
\end{equation}
with $g_{a}$ smooth satisfying
\begin{equation}\label{eq:improvement_2}
	\abs{\d^{b}_\xi \d^{\gamma}_y g_{a}\brkt{x,y,\xi}}\lesssim \abs{\xi}^{1-\abs{b}}.
\end{equation}
On the other hand, if $\abs{a}\geq 2$,
\[
	{\abs{\d^{b}_\xi \d^{a}_y \Phi(x,y,\xi)}}=\abs{\d^{b}_\xi \d^{a}_y \varphi(y,\xi)}\lesssim \abs{\xi}^{1-\abs{b}}.
\]

To continue with the proof, first assume that $k\leq \abs{\alpha}/2$. Then the previous estimates yield
\begin{equation}\label{eq:first_case}
	\abs{  \prod_{j=1}^ k \d^ {b_j}_{\xi}\d^{a_j}_y(\Phi(x,y,\xi))}\lesssim \abs{\xi}^{k-\abs{\beta_1}} \abs{x-y}^{k'}\leq \abs{\xi}^{\abs{\alpha}/2} (1+\abs{\xi}\abs{x-y})^{\abs{\alpha}/2}  \abs{\xi}^{-\abs{\beta_1}},
\end{equation}
where $k'\leq k$ is the cardinality of the set of indices $j$ such that $\abs{a_j}=1$.
Let
\begin{equation}\label{eq:operator}
L_\zeta=\frac{(1-\bra{\xi}^2\Delta_\zeta)}{1+\bra{\xi}^2|x-y|^2}, \qquad \text{so} \quad L_\zeta^N e^{i(x-y)\cdot\zeta}=e^{i(x-y)\cdot\zeta}.
\end{equation}
Integration by parts with $L_\zeta$ yields that  $\d^\beta_\xi R^I_{\alpha}$ can be written as a linear combination of  terms of the type
\[
\begin{aligned}
&  \iint \frac{e^{i(x-y)\cdot\zeta}\,\partial_y^{\alpha_1} \chi(x-y)\, \d^{\beta_1}_{\xi}b_{\alpha,\alpha_2}(x,y,\xi,\zeta,s,t)\,  \prod_{j=1}^ k \d^ {b_j}_{\xi}\d^{a_j}_y\Phi(x,y,\xi)\, \prod_{j=1}^{l}\d^{d_j}_\xi \Phi(x,y,\xi)}{(1+\bra{\xi}^2 |x-y|^2)^N}\\
&\qquad \times \bra{\xi}^{|\kappa|} \partial_\zeta^{\kappa} e^{i\Phi(x,y,\xi)} \dd y \ddd\zeta.
\end{aligned}
\]
Taking absolute values and using \eqref{eq:bs}, \eqref{bound} and \eqref{eq:first_case} we have that the previous expression is bounded by
\[
\begin{aligned}
&t^{-\abs{\alpha}(1-\varepsilon )}\bra{\xi}^{m-\abs{\alpha}(\frac{1}2-\varepsilon )-\abs{\beta}}
\int_{\abs{\zeta}\leq r \bra{\xi}}\int_{\abs{x-y}<\epsilon} \frac{(1+\abs{\xi}\abs{x-y})^{\abs{\alpha}/2+2\abs{\beta}} }{(1+\bra{\xi}^2 |x-y|^2)^N} \dd y \ddd\zeta,\\	
&\lesssim t^{-\abs{\alpha}(1-\varepsilon )}\bra{\xi}^{m-\abs{\alpha}(\frac{1}2-\varepsilon )-\abs{\beta}} \bra{\xi}^n
\int_{\abs{z}<\epsilon} (1+\bra{\xi}^2 |z|^2)^{\abs{\alpha}/2+2\abs{\beta}-N} \dd z\\
&\lesssim  t^{-\abs{\alpha}(1-\varepsilon )}\bra{\xi}^{m-\abs{\alpha}(\frac{1}2-\varepsilon )-\abs{\beta}},
\end{aligned}
\]
provided $N$ is large enough.

Suppose now that $k>\abs{\alpha}/2$. Recalling that $k'$ denotes the cardinality of the set of indices $j$ such that $\abs{a_j}=1$, set $k_0 = k - k'$, which is the cardinality of the set of $j$ such that $\abs{a_j}\geq 2$. We must have $2k_0 + k' \leq |\alpha|$, so therefore $2k-|\alpha| \leq k'$. Hence \eqref{eq:improvement} and \eqref{eq:improvement_2} yield
\begin{equation}\label{eq:improvement_3}
	\prod_{j=1}^ k \d^ {b_j}_{\xi}\d^{a_j}_y(\Phi(x,y,\xi))=\sum_{\abs{\gamma}\geq 2k-{\abs{\alpha}}}^{\abs{\alpha}} g_{\gamma}(x,y,\xi) (x-y)^{\gamma} \quad \text{with $\quad \abs{g_\gamma(x,y,\xi)}\lesssim \abs{\xi}^{k-\abs{\beta_2}}$.}
\end{equation}
This allows us to integrate by parts and reduce the problem of  estimating $\d^\beta_\xi R^I_{\alpha}$ to that of controlling an expression of the type
\[
	 \iint e^{i(x-y)\cdot\zeta}\,\partial_y^{\alpha_1} \chi(x-y)  g_\gamma(x,y,\xi)\,\d^{\beta_1}_{\xi}b_{\alpha,\alpha_2}(x,y,\xi,\zeta,s,t)\, \prod_{j=1}^{l}\d^{d_j}_\xi \Phi(x,y,\xi) \d^\gamma_{\zeta}\,e^{i\Phi(x,y,\xi)} \dd y \ddd\zeta,
\]
for $2k-\abs{\alpha}\leq \abs{\gamma}\leq \abs{\alpha}$. Integration by parts with the operator $L_\zeta$ defined in \eqref{eq:operator}, allows us to reduce the problem to control terms of the type
\[
\begin{aligned}
&  \iint \frac{\partial_y^{\alpha_1} \chi(x-y)g_\gamma(x,y,\xi)\,\d^{\beta_1}_{\xi}b_{\alpha,\alpha_2}(x,y,\xi,\zeta,s,t)\, \prod_{j=1}^{l}\d^{d_j}_\xi \Phi(x,y,\xi)}{(1+\bra{\xi}^2 |x-y|^2)^N}\\
&\qquad\qquad\times  \bra{\xi}^{|\kappa|}\d^{\gamma+\kappa}_{\zeta}\, e^{i\Phi(x,y,\xi)+i(x-y)\cdot\zeta} \dd y \ddd\zeta.
\end{aligned}
\]
Taking absolute values and using \eqref{eq:bs}, \eqref{bound}, \eqref{eq:improvement_3} and the fact that $k>\abs{\alpha}/2$ we obtain that this term is bounded by
\[
\begin{aligned}
&
t^{-\abs{\alpha}(1-\varepsilon )}\bra{\xi}^{m-\abs{\alpha}(1-\varepsilon )+(k-\abs{\gamma})-\abs{\beta}}\int_{\abs{\zeta}\leq  r \bra{\xi}}\int_{\abs{x-y}<\epsilon} \frac{ (1+\abs{\xi}\abs{x-y})^{2\abs{\beta_3}}}
{(1+\bra{\xi}^2 |x-y|^2)^N} \dd y \ddd\zeta,\\
&\lesssim  t^{-\abs{\alpha}(1-\varepsilon )}\bra{\xi}^{m-\abs{\alpha}(\frac{1}2-\varepsilon )-\abs{\beta}},
\end{aligned}
\]
provided $N$ is large enough. Thus, altogether this implies \eqref{eq:claim_2}.\\

We claim now that
\begin{equation}\label{eq:claim_3}
	\sup_{0<s<1}\sup_{x\in \R^n }\abs{\d^\gamma_\xi \d^\beta_x R^{I\!\!I}_\alpha (t,s,x,\xi)}\lesssim t^{-\abs{\alpha}(1-\varepsilon )}\bra{\xi}^{m-\abs{\alpha}(1-\varepsilon )-\abs{\gamma}}.
\end{equation}
As we did before, to simplify the exposition, we are going omit the estimates in the $x$-derivatives. To prove the estimates on $R_{\alpha}^{I\!\!I}(t,x,\xi)$ one defines
\[
\Psi(x,y,\xi,\zeta)=(x-y)\cdot\zeta+\Phi(x,y,\xi)= (x-y)\cdot(\nabla_x\phase(x,\xi)+\zeta)+\phase(y,\xi)-\phase(x,\xi).
\]
The assumptions \ref{partone} and \ref{parttwo} on the phase function $\phase$, that fact that we have chosen $\epsilon <r/8C_0$, $|x-y|<\epsilon$ in the support of $\chi$ and that we are in the region $|\zeta|\geq \frac{r}{2}\bra{\xi}$  imply that
\begin{equation*}\label{eq:rho}
\begin{aligned}
|\nabla_y\Psi| & =|-\zeta+\nabla_y\phase-\nabla_x\phase|\leq 2C_2(|\zeta|+\bra{\xi}), \quad \text{and} \\
|\nabla_y\Psi| & \geq |\zeta|-|\nabla_y\phase-\nabla_x\phase| \geq \frac{1}{2}|\zeta|+\left(\frac{r}{4}-C_0|x-y| \right)\bra{\xi}\geq C(|\zeta|+\bra{\xi}).
\end{aligned}
\end{equation*}
Arguing as before,  $R^{I\!\!I}_{\alpha}$ can be expressed as a finite linear combination of terms of the form
\[
	 \iint e^{i\Psi(x,y,\xi,\zeta)}\,\partial_y^{\alpha_1} \chi(x-y)  \prod_{j=1}^ k \d^{a_j}_y(\Phi(x,y,\xi))\, c_{\alpha,\alpha_2}(x,y,\xi,\zeta,s,t) \dd y \ddd\zeta,
\]
where $\alpha_1+\alpha_2+\alpha_3=\alpha$ and  $a_1+\ldots +a_k=\alpha_3$, with $\abs{a_j}\geq 1$ provided $\alpha_3\neq 0$, and
\[
	c_{\alpha,\alpha_2}(x,y,\xi,\zeta,s,t)=\brkt{1-g\left(\frac{\zeta}{\bra{\xi}}\right)} \d^{\alpha_2}_y a_t(y,\xi) \brkt{\d^\alpha \rho}\brkt{t\nabla_x\varphi(x,\xi)+st \zeta}).
\]
So, $c_{\alpha,\alpha_2}$ is supported in $\abs{\zeta}\geq \bra{\xi}$ and arguing as in \eqref{eq:bs}, one has that
\begin{equation}\label{eq:cs}
	\sup_{0<s<1}\abs{\d^\beta_x \d^\gamma_\xi \d^\nu_y c_{\alpha.\alpha_2}(x,y,\xi,\zeta,s,t)}\lesssim t^{-\abs{\alpha}(1-\varepsilon )}\bra{\xi}^{m-\abs{\alpha}(1-\varepsilon )-\abs{\gamma}}.
\end{equation}
Therefore $\d^\gamma_\xi R^{I\!\!I}_{\alpha}$ can be written as a finite linear combination of terms of the type
\[
	 \iint e^{i\Psi(x,y,\xi,\zeta)}\,\partial_y^{\alpha_1} \chi(x-y)  \prod_{j=1}^ k \d^{b_j}_\xi \d^{a_j}_y(\Phi(x,y,\xi)) \prod_{j=1}^ l \d^{d_j}_\xi \Phi(x,y,\xi)\, \d^{\gamma_2}_\xi c_{\alpha,\alpha_2}(x,y,\xi,\zeta,s,t) \dd y \ddd\zeta,
\]
For the differential operator defined via the expression $^tL_y=i|\nabla_y\Psi|^{-2}\sum_{j=1}^n (\partial_{y_j}\Psi) \partial_{y_j}$, induction shows that $L_y^N$ has the form
\begin{equation*}\label{EQ:LNy} L_y^N=\frac{1}{|\nabla_y\Psi|^{4N}}\sum_{|\nu|\leq N} P_{\nu,N}\partial_y^\nu,\quad \text{where} \quad P_{\nu,N}=\sum_{|\mu|=2N} c_{\nu\mu\lambda_j}(\nabla_y\Psi)^\mu \partial_y^{\lambda_1}\Psi\cdots \partial_y^{\lambda_N}\Psi,
\end{equation*}
$|\mu|=2N$, $|\lambda_j|\geq 1$ and $\sum_{M}^N |\lambda_j|+|\nu|=2N$. It follows from assumption \ref{parttwo} on $\phase$ that $|P_{\nu,N}|\leq C(|\zeta|+\bra{\xi})^{3N}.$ Now Leibniz's rule yields that  $\d^\gamma_\xi R^{I\!\!I}_{\alpha}$ can be written as a finite linear combination of terms of the type
\[
\begin{aligned}
& \iint e^{i\Psi(x,y,\xi,\zeta)} |\nabla_y\Psi|^{-4N} P_{\nu,N}(x,y,\xi,\zeta)\prod_{j=1}^ k \d^{a_j+c_j}_y\d^{b_j}_\xi \Phi(x,y,\xi) \prod_{j=1}^ l \d^{e_j}_y\d^{d_j}_\xi \Phi(x,y,\xi)\\
&\quad \times \d^{\nu_2}_y\d^{\gamma_2}_\xi c_{\alpha,\alpha_2}(x,y,\xi,\zeta,s,t) \partial_y^{\alpha_1+\nu_1} \chi(x-y) \dd y \ddd\zeta
\end{aligned}
\]
with $\abs{\nu}\leq N$.
These terms are bounded by  $t^{-\abs{\alpha}(1-\varepsilon )}\bra{\xi}^{m-\abs{\alpha}(1-\varepsilon )-\abs{\gamma}}$ times
\[
\begin{aligned}
&\int_{|\zeta|\geq \frac{r}{2}\bra{\xi}} \int_{|x-y|<\epsilon} (|\zeta|+\bra{\xi})^{-N} \bra{\xi}^{|\alpha|+\abs{\gamma}}\, \dd y \ddd\zeta & \lesssim\bra{\xi}^{|\alpha|+\abs{\gamma}} \int_{|\zeta|\geq \frac{r}{2}\bra{\xi}}|\zeta|^{-N}\ddd\zeta \lesssim \bra{\xi}^{|\alpha|+\abs{\gamma}+n-N}.
\end{aligned}
\]
That is
\[
	\abs{\d^\gamma_\xi R_{\alpha}^{I\!\!I}(t,x,\xi)}\lesssim t^{-\abs{\alpha}(1-\varepsilon )}\bra{\xi}^{m-\abs{\alpha}(1-\varepsilon )-\abs{\gamma}} \bra{\xi}^{|\alpha|+\abs{\gamma}+n-N},
\]
which implies \eqref{eq:claim_3} (with $\beta=0$) if we take $N\geq M+\abs{\gamma}+n$.

From these estimates it readily follows that the term $r(t,x,\xi)$ in the statement of Theorem \ref{left composition with pseudo} belongs to the class $L^\infty S^{m-M(1/2-\varepsilon )}_1$ with seminorms bounded by $t^{M\varepsilon}$. Demonstration of the fact that $r(t,x,\xi)\in S^{m-M(1/2-\varepsilon )}_{1,0}$ amounts to control the $x$ derivatives of $\d^\gamma_\xi \tilde{R}_{\alpha}(t,x,\xi)$ where $\tilde{R}_{\alpha}$ is as in \eqref{modified rest term}, but since the estimates involved are carried out in a similar way as in the case of $\sigma_{\alpha}(t, x,\xi)$, we shall leave the details of this verification to the interested reader. This concludes the proof of Theorem \ref{sec:proof_of_asymptotic}.

\begin{rem}
The fact that $r(t,x,\xi)\in L^\infty S^{m-M(1/2-\varepsilon )}_1$ is actually enough for proving our main Theorem $\ref{main}$. Indeed  by \ref{eq:DSS} in Theorem $\ref{linear_FIO}$, the $\mathrm{FIOs}$ with such amplitudes are bounded on $L^p$ spaces with $1\leq p\leq\infty$.

\end{rem}

\end{document}